%% file: hat.tex
\documentclass[11pt]{amsart}

\input{preamble.tex}

\usepackage{xifthen}  

\usepackage{breqn}

\definecolor{blue-unshaded}{rgb}{0.55,0.55,1}
\definecolor{blue-shaded}{rgb}{0.2,0.2,1}
\definecolor{red-unshaded}{rgb}{1,0.55,0.55}
\definecolor{red-shaded}{rgb}{1,0.2,0.2}

\newcommand{\ybelow}{\begin{tikzpicture}[scale=.8, baseline]
	\clip (-.9,-2.9) rectangle (3.9,1.9);
	\filldraw[fill=red-unshaded] (0,2) --(0,0).. controls (0,-.8) .. (1,-1) -- (1,2)--(0,2);
	\filldraw[fill=red-shaded] (1,2)--(1,-1) .. controls (2,-1.2) and (1,-1.8) .. (2,-2) -- (2,2)--(1,2);
	\filldraw[fill=red-unshaded] (4,2)--(2,2)--(2,-2) .. controls (3,-2.2) .. (3,-3)--(4,-3)--(4,2);
	\filldraw[fill=blue-shaded] (3,-3)--(3,-3) .. controls (3,-2.2) .. (2,-2) -- (2,-3)--(3,-3);
	\filldraw[fill=blue-unshaded] (2,-2)--(2,-3)--(1,-3)--(1,-1) .. controls (2,-1.2) and (1,-1.8) .. (2,-2);
	\filldraw[fill=blue-shaded] (0,2)-- (0,0).. controls (0,-.8) .. (1,-1) -- (1,-3) -- (-1,-3)--(-1,2)--(0,2);
	\node[minimum size=.7cm, shape=circle, fill=white, draw]  at (1,-1) {\tiny$K$};
	\node[minimum size=.7cm, shape=circle, fill=white, draw]  at (2,-2) {\tiny$K^*$};
	\node[minimum size=1.2cm, shape=rectangle, rounded corners = 4mm, fill=white, draw]  at (1.5,.5) {$y$};
\end{tikzpicture}
}
\newcommand{\xabove}{\begin{tikzpicture}[scale=.8, baseline]
	\clip (-.9,-2.9) rectangle (3.9,1.9);
	\filldraw[fill=red-unshaded] (0,2) .. controls (0,1.2) .. (1,1) -- (1,2)--(0,2);
	\filldraw[fill=red-shaded] (1,2)--(1,1) .. controls (2,.8) and (1,.2) .. (2,0) -- (2,2)--(1,2);
	\filldraw[fill=red-unshaded] (4,2)--(2,2)--(2,0) .. controls (3,-.2) .. (3,-1) -- (3,-3)--(4,-3)--(4,2);
	\filldraw[fill=blue-shaded] (3,-3)--(3,-1) .. controls (3,-.2) .. (2,0) -- (2,-3)--(3,-3);
	\filldraw[fill=blue-unshaded] (2,0)--(2,-3)--(1,-3)--(1,1) .. controls (2,.8) and (1,.2) .. (2,0);
	\filldraw[fill=blue-shaded] (0,2) .. controls (0,1.2) .. (1,1) -- (1,-3) -- (-1,-3)--(-1,2)--(0,2);
	\node[minimum size=.7cm, shape=circle, fill=white, draw]  at (1,1) {\tiny$K$};
	\node[minimum size=.7cm, shape=circle, fill=white, draw]  at (2,0) {\tiny$K^*$};
	\node[minimum size=1.2cm, shape=rectangle, rounded corners = 4mm, fill=white, draw]  at (1.5,-1.5) {$x$};
\end{tikzpicture}
}

\tikzstyle{conn}=[circle, draw, thick, fill=white, opaque, inner sep = .7mm]
\tikzstyle{gauge}=[circle, draw, thick, fill=white, opaque, inner sep = .7mm]
\tikzstyle{invgauge}=[circle, draw, thick, fill=black!70!white, inner sep = .7mm]

\title[The little desert?]{The little desert? Some subfactors with index in the interval $(5,3+\sqrt{5})$}
\author{Scott Morrison and Emily Peters}
\begin{document}

\begin{abstract}
Progress on classifying small index subfactors has revealed an almost empty landscape. In this paper we give some evidence that this desert continues up to index $3+\sqrt{5}$.
There are two known quantum-group subfactors with index in this interval, and we show that these subfactors are the only way to realize the corresponding principal graphs.  One of these subfactors is 1-supertransitive, and we demonstrate that it is the only 1-supertransitive subfactor with index between $5$ and $3+\sqrt{5}$.  Computer evidence shows that any other subfactor in this interval would need to have rank at least 38. We prove our uniqueness results by showing that there is a unique flat connection on each graph.  The result on 1-supertransitive subfactors is proved by an argument using  intermediate subfactors, running the `odometer' from the {\tt FusionAtlas` Mathematica} package and paying careful attention to dimensions.
\end{abstract}

\maketitle
At this point, we have a complete classification of subfactor planar algebras with index less than 5 (and hence of amenable subfactors in that range). The work of \cite{MR996454, MR1193933, MR1145672, MR1313457, MR1308617, MR1278111, MR1213139,MR1308617,MR1054961} established this up to index 4.   Haagerup's landmark work \cite{MR1317352}, followed by the results in \cite{MR1686551, MR1625762, 0909.4099, 1004.0665}, gave the classification up to index $3+\sqrt{3}$, and finally the classification up to index less than 5 appears in \cite{index5-part1, index5-part2, index5-part3, index5-part4, MR999799, MR1832764, 1102.2052, 1004.0665}. At index exactly 5 the classification is known (there are 5 group-subgroup subfactors) but has not yet appeared in the literature. 
The great surprise from these classifications has been just how few small index subfactors there are: just 10 subfactors in the index range $(4,5)$, coming in 5 pairs which are either dual or conjugate to each other. We know that at higher indices there is an incredible profusion of subfactors, and even by index $6$ there are certain wild phenomena. As we have worked through low indices, on the other hand, we see a desert; the `little desert' of our title. In this paper, we give evidence, and a conjecture, that the desert continues further.
Beyond 5, the techniques used previously seem to lose traction. The combinatorial growth in possible principal graphs becomes too rapid, and we don't have effective obstructions at the level of graphs for principal graph which begin with quadruple or higher branches. Nevertheless, this paper gives some preliminary results on the range $(5, 3+\sqrt{5})$; we completely classify the 1-supertransitive case, and prove that the two known principal graphs are uniquely realized by quantum group subfactors. 

There are two subfactors with index in the interval $(5,3+\sqrt{5})$ which are easy to construct from quantum groups.
Given a quantum group $U_q(\mathfrak{g})$, irreducible representation $V$ and root of unity $\zeta = \exp(2 \pi i/ \ell)$, there is a corresponding subfactor $\cQ(\mathfrak{g}, V, \ell)$, as long as a certain positivity condition is satisfied \cite{MR1470857}. (This condition has been completely analysed in \cite{MR1090432, MR2180373, MR2414692, MR2783128}.) We are interested in the subfactors $\cA = \cQ(\mathfrak{su}_2, V_{(2)}, 14)$ and $\cB = \cQ(\mathfrak{su}_3, V_{(1,0)}, 14)$. Here $V_{(2)}$ is the three dimensional adjoint representation of $\mathfrak{su}_2$, and $V_{(1,0)}$ is the three dimensional standard representation of $\mathfrak{su}_3$. Each has quantum dimension $q^{-2} + 1 + q^2$, and at $q = \exp(2 \pi i / 14)$ this is $d\approx2.24698$, the largest root of $x^3-2x^2-x+1$. These subfactors have principal graphs
\begin{align*}
\Gamma(\cA) & = \left(\bigraph{bwd1v1p1v1x0p1x1duals1v1x2}, \bigraph{bwd1v1p1v1x0p1x1duals1v1x2}\right) \\
\Gamma(\cB) & = \left(\bigraph{bwd1v1v1p1v1x0p0x1p0x1v0x1x0p1x0x1duals1v1v2x1x3}, \bigraph{bwd1v1v1p1v1x0p0x1p0x1v0x1x0p1x0x1duals1v1v2x1x3}\right)
\end{align*}
and both have index $d^2 \approx 5.04892$. In the language of `levels', $\ell = 14$ corresponds to $SU(2)_5$ and $SU(3)_4$. The subfactor $\cA$ is the reduced subfactor construction for the third vertex in the $A_6$ subfactor. See the case $\ell=7$, $k=3$ of \cite[Proposition 6.1]{MR1090432} for another realization of the subfactor $\cA$.  The subfactor $\cB$ first appears in \cite{MR936086}, where its index (but not its principal graph) is calculated.  Its principal graph was shortly known to the experts, but its first appearance in print is as a special case of the last section of \cite{MR1470857}.

The object of this paper is to show that these are the only subfactors with these principal graphs, and moreover that $\cA$ is the only $1$-supertransitive subfactor with index in the interval $(5, 3+\sqrt{5})$.   Our three main theorems are

\newtheorem*{thm:only1ST}{Theorem \ref{thm:only1ST}}

\begin{thm:only1ST}
The only 1-supertransitive subfactors with index in the range $(5, 3+\sqrt{5})$ have principal graph $\Gamma(\cA)$.
\end{thm:only1ST}

\newtheorem*{thm:uniquehat}{Theorem \ref{thm:uniquehat}}

 \begin{thm:uniquehat}
 There is a unique subfactor with principal graph $\Gamma(\cA)$.
 \end{thm:uniquehat}
 
 \newtheorem*{thm:uniquehex}{Theorem \ref{thm:uniquehex}}
 
  \begin{thm:uniquehex}
 There is a unique subfactor with principal graph $\Gamma(\cB)$.
 \end{thm:uniquehex}

We conjecture  that these two subfactors are the only non-$A_\infty$ subfactors (of any supertransitivity) with index in that interval, although we cannot prove this at present. See Conjecture \ref{conj} and Theorem \ref{thm:conj} for some evidence of this.

We begin with a section on graph planar algebras, introducing some new notions which are required for the rest of the paper. In \S \ref{1STclassification}, we prove Theorem \ref{thm:only1ST}, establishing that $\Gamma(\cA)$ is the only possible principal graph for a $1$-supertransitive subfactor with index in the interval $(5, 3+\sqrt{5})$. Next, in \S \ref{connections} we find that  
the eigenvalue condition from \cite[Theorem 1.7]{index5-part3} ensures that there are at most two  two gauge equivalence classes of bi-unitary connections on $\Gamma(\cA)$ which are flat. We show that there is exactly one bi-unitary connection on $\Gamma(\cB)$. 
This shows there is at most one subfactor with principal graph $\Gamma(\cB)$; since the subfactor coming from the standard representation in $SU(3)_4$ has principal graph $\Gamma(\cB)$ in fact there is exactly one such subfactor (and the bi-unitary connection we found must in fact be flat).

We then turn to finding flat elements in the graph planar algebra, for each of the two connections on $\Gamma(\cA)$. By definition, for a bi-unitary connection on a $k$-supertransitive graph to be flat, there must be a low weight flat vector in the $(k+1)$-box space of the graph planar algebra. Conversely, the existence of such a low weight flat vector guarantees the existence of \emph{some} $k$-supertransitive subfactor at the same index (although not necessary with the expected principal graph). In \S \ref{flatlowweight}, we show that one of the two connections on $\Gamma(\cA)$ has such a flat low weight vector; the classification statement from \S \ref{1STclassification} ensures that the resulting $2$-supertransitive subfactor in fact has principal graph $\Gamma(\cA)$. We show that the other connection has no such flat low weight vectors, so can not be flat. 

A word about our use of computers:  Theorem \ref{thm:only1ST} uses the `odometer' of {\tt FusionAtlas} in an essential way (see \cite{index5-part1}, which also makes essential use of the odometer, for the details of this routine).  However, with the addition of the assumption that the principal graph is finite depth, we can (and do) prove this theorem by hand.  Theorems \ref{thm:uniquehat} and \ref{thm:uniquehex} are proved by explicitly providing connections and flat low-weight elements.  Checking that these elements are flat under the given connection can be done either by hand or by computer, and we expect that you'll have more faith that we did this correctly when we say that we did it with a computer.  (Also, the {\tt Mathematica} code we used for this purpose is available with the {\tt arXiv} sources for this article.)

\section{Graph planar algebras}

Planar algebras were first defined by Jones in \cite{math.QA/9909027}, which also explains their relation to subfactors.  We do not reproduce the definition here.

\subsection{Lopsided and spherical planar algebras}
Starting with any shaded planar algebra, it is possible to rescale the pivotal structure in two different ways. This means introducing a scalar factor for each critical point in strings in the action of planar tangles. 
The identity $$
\begin{tikzpicture}[baseline=0]
	\clip (-1.1,-.9) rectangle (1.5,.9);
 	\filldraw[fill=white!50!gray] (-1,-1)--(-1,0) arc (180:0:5mm) arc (-180:0:5mm) -- (1,1)--(2,1)--(2,-1);
\end{tikzpicture}
\quad
=
\quad
\begin{tikzpicture}[baseline=0]
	\clip (.5,-.9) rectangle (1.5,.9);
 	\filldraw[fill=white!50!gray] (1,-1) -- (1,1)--(2,1)--(2,-1);
\end{tikzpicture}
$$
(and the corresponding identity in the opposite shading) contrains the possible rescalings to:
\begin{align*}
\begin{tikzpicture}[baseline=0]
	\clip (-1.5,-.5) rectangle (.5,.6);
 	\filldraw[fill=white!50!gray] (-2,1)--(-2,-1) -- (-1,-1)--(-1,0) arc (180:0:5mm) --(0,-1)--(1,-1)--(1,1);
\end{tikzpicture}
\quad
&
\mapsto x \quad
\begin{tikzpicture}[baseline=0]
	\clip (-1.5,-.5) rectangle (.5,.6);
 	\filldraw[fill=white!50!gray] (-2,1)--(-2,-1) -- (-1,-1)--(-1,0) arc (180:0:5mm) --(0,-1)--(1,-1)--(1,1);
\end{tikzpicture}
&
\begin{tikzpicture}[baseline=0,yscale=-1]
	\clip (-1.5,-.5) rectangle (.5,.6);
 	\filldraw[fill=white!50!gray] (-2,1)--(-2,-1) -- (-1,-1)--(-1,0) arc (180:0:5mm) --(0,-1)--(1,-1)--(1,1);
\end{tikzpicture}
\quad
&
\mapsto y^{-1}
\begin{tikzpicture}[baseline=0, yscale=-1]
	\clip (-1.5,-.5) rectangle (.5,.6);
 	\filldraw[fill=white!50!gray] (-2,1)--(-2,-1) -- (-1,-1)--(-1,0) arc (180:0:5mm) --(0,-1)--(1,-1)--(1,1);
\end{tikzpicture}
\\
\begin{tikzpicture}[baseline=0,yscale=-1]
	\clip (-1.5,-.5) rectangle (.5,.6);
 	\filldraw[fill=white!50!gray] (-1,-1)--(-1,0) arc (180:0:5mm) --(0,-1);
\end{tikzpicture}
\quad
&
\mapsto x^{-1} 
\begin{tikzpicture}[baseline=0, yscale=-1]
	\clip (-1.5,-.5) rectangle (.5,.6);
 	\filldraw[fill=white!50!gray] (-1,-1)--(-1,0) arc (180:0:5mm) --(0,-1);
\end{tikzpicture}
&
\begin{tikzpicture}[baseline=0]
	\clip (-1.5,-.5) rectangle (.5,.6);
 	\filldraw[fill=white!50!gray] (-1,-1)--(-1,0) arc (180:0:5mm) --(0,-1);
\end{tikzpicture}
\quad
&
\mapsto y \quad
\begin{tikzpicture}[baseline=0]
	\clip (-1.5,-.5) rectangle (.5,.6);
 	\filldraw[fill=white!50!gray] (-1,-1)--(-1,0) arc (180:0:5mm) --(0,-1);
\end{tikzpicture}
\end{align*}
If we want this rescaling to respect the star structure, we must have $y^{-1} = x^*$, since caps and cups are adjoint to each other. However, we will generally not do this; in fact throughout below we will take $y=1$.

We will call the planar algebra obtained from $\cP$ in this way $\cP^{\cap x,y}$. Even though $\cP$ and $\cP^{\cap x,y}$ are not equivalent, it is easy to describe the relationship between the actions of planar tangles. Say $\psi: \cP \to \cP^{\cap x,y}$ is the identity on the underlying vector spaces. If $T$ is a planar tangle and $z_i$ are elements in $\cP$, then
\begin{equation}
\label{eq:intertwiner}
\psi(T(z_i)) = x^{n} y^{m} T(\psi(z_i)),
\end{equation}
where $n$ is the signed count (minimums are positive, maximums are negative) of critical points which are shaded above in $T$ and $m$ is the signed count of critical points shaded below.

We call a shaded planar algebra \emph{spherical} if the two circles (shaded or unshaded inside) have the same value (that is, they are the same multiple of the appropriate empty diagram), usually called $\delta$. We call a shaded planar algebra \emph{lopsided} if the circle shaded inside has value $1$. We can always obtain a lopsided planar algebra from a spherical one, by choosing $x=\delta, y=1$ above. We can also go the other way.
For every rescaling, the product of the value of the two circles is constant; in particular in the lopsided planar algebra the unshaded circle has value $\delta^2$.


Generally, we have found that the lopsided pivotal structure is extremely helpful. Its essential importance is that it often allows us to work over a fixed number field (the values of loops must certainly lie in the scalars; only having to include the index, not the square root of the index, is a promising start). Once we are working in a fixed number field, a great many calculations become much easier, and it is possible to have a computer perform exact arithmetic very efficiently. The first use of the lopsided pivotal structure (although somewhat hidden) was in the construction of the extended Haagerup subfactor planar algebra in \S 6 of \cite{0909.4099}, where we needed to compute the moments of some elements in the graph planar algebra. It has been used subsequently in \cite{1002.0168} and \cite{MPspokes}. This paper is the first time the lopsided pivotal structure has been used alongside the theory of connections; we are now able to explicitly check flatness in cases that would have been more difficult in the spherical pivotal structure.

\subsection{The lopsided graph planar algebra}
The planar algebra of a bipartite graph was first definied in \cite{MR1865703}.  We recall that definition, as well as the lopsided version of a graph planar algebra, which we explicitly describe. For comparison, we'll show the definition of both $\cG^{\text{spherical}}(\Gamma)$ and $\cG^{\text{lopsided}}(\Gamma)$. (Somewhat confusingly, it's not the case that  $\cG^{\text{lopsided}}(\Gamma)$ is just  $\cG^{\text{spherical}}(\Gamma)^{\cap \delta,1}$, the lopsided version of $\cG^{\text{spherical}}(\Gamma)$. It's also essential, to obtain the desired number-theoretic properties, to rescale the basis.)

In both, the underlying vector spaces $\cG^{\bullet}(\Gamma)_{n, \pm}$ are just functionals on loops of length $n$ on $\Gamma$, with the base point at either an even or odd depth vertex depending on $\pm$. We'll often abuse notation and think of a loop on $\Gamma$ in place of its indicator function. To define the action of a planar tangle $T$, we specify its values $T(\gamma_i)$, where  the $\gamma_i$ are the indicator functions for loops corresponding to the input vector spaces for $T$. This element $T(\gamma_i) \in \cG^\bullet_{n}$ (here $n$ is the number of points on the outer boundary of $T$) is itself a functional on loops corresponding to the outside boundary of $T$, so we specify it by giving its values on loops $\gamma_0$:
\begin{equation}
\label{eq:action}
T(\gamma_i)(\gamma_0) = \sum_{b \in \cL} c(T, b) ,
\end{equation}
where the label set $\cL$ consists of all ways to compatibly color the strands of $T$ with edges of $\Gamma$ and the regions of $T$ with vertices of $\Gamma$, such that around each inner or outer boundary of $T$ the colors agree with the loops $\gamma_i$. The coefficients $c(T,b)$ are the so-called `critical point coefficients', which determine the difference between the spherical and lopsided versions of the graph planar algebra. In each case, $c(T,b)$ is a product over the critical points in the strings in $T$ of some function of the labels given by $b$ appearing above and below the critical point. In the spherical case, we have
$$c^{\text{spherical}}(T,b) = \prod_{\substack{\text{critical} \\ \text{points $x$}}} \sqrt{\frac{d_{\text{above }x}}{d_{\text{below }x}}}^{\operatorname{sign}(x)}$$
while in the lopsided case we have
$$c^{\text{lopsided}}(T,b) = \prod_{\substack{\text{critical} \\ \text{points $x$}}} \left(\frac{d_{\text{above }x}}{\delta^{\text{shading above }x}}\right)^{\operatorname{sign}(x)}.$$
In these formulas, $d_{\text{above }x}$ means the Frobenius-Perron dimension of the graph vertex appearing above the critical point in the labelling given by $b$, and similarly for $d_{\text{below }x}$. When we write $\delta^{\text{shading above }x}$ we just mean $1$ if the critical point is not shaded above, and $\delta$ if it is shaded above. (In  the two-sided graph planar algebra, defined in \S \ref{twosided}, sometimes this quantity is $\delta^{-1}$.) Henceforth, we'll just call the ratio $d_a / \delta^{\operatorname{shading}(a)}$ the `lopsided dimension' of $a$, written $d^{\text{lopsided}}_a$.  The quantity $\operatorname{sign}(x)$ is the sign of the second derivative at the critical point: $+1$ if the critical point is a local minimum, $-1$ if it is a local maximum.

The careful reader will note that in the above definition we have implicitly chosen a Morse function on our tangles $T$, so that we can talk about critical points and their signs. This is at first sight incompatible with the definition of a planar algebra, where the maps associated to tangles must be invariant under planar isotopies (even those which rotate the boundary discs). To resolve this, for each space $\cG(\Gamma)_{n,\pm}$ we introduce $n+1$ different bases, which we think of as having $k$ boundary points pointing upwards and $n-k$ boundary points pointing downwards. In Equation \eqref{eq:action} above, this removes the ambiguity in choosing the Morse function needed to classify boundary points.  The basis corresponding to all boundary points upwards is just the basis of indicator functions on loops, as above, and the other bases are all defined by the transporting this first basis via the tangles
$$
\newcommand{\turncontents}{
	\draw (-0.5,0)--(-0.5,2);
	\draw (0.5,0)--(0.5,2);
	\draw (1.5,0)--(1.5,2);
	\draw (2.5,0)--(2.5,1) arc (180:0:.5cm) -- (3.5,-2);
	
	\draw (.2,0)--(.2,-2);
	\draw (1.8,0)--(1.8,-2);
	
	\filldraw[fill=white] (-1,-.5) rectangle (3,.5);
}
\begin{tikzpicture}[baseline=0,scale=0.5]
\turncontents
\end{tikzpicture},
\qquad
\begin{tikzpicture}[baseline=0,scale=0.5, xscale=-1]
\turncontents
\end{tikzpicture},
\qquad
\begin{tikzpicture}[baseline=0,scale=0.5, yscale=-1]
\turncontents
\end{tikzpicture}
\quad \text{ and } \quad
\begin{tikzpicture}[baseline=0,scale=0.5, xscale=-1, yscale=-1]
\turncontents
\end{tikzpicture}
$$
These bases are coherent (that is, any two ways to modify the division into upper and lower boundary give equal maps) because in Equation \eqref{eq:action} the coefficient $c(\rho^{2\pi},b)$ for the $2\pi$ rotation is always $1$.

Consider now the linear map $\natural : \cG^{\text{spherical}} \to \cG^{\text{lopsided}}$ which rescales loops according to 
$$\natural(\gamma) = \sqrt{\left(\prod_{{\text{$a$ above}}} d^{\text{lopsided}}_a \right)\left( \prod_{{\text{$b$ below}}} {d^{\text{lopsided}}_b}\right)^{-1}} \; \; \gamma.$$
Taking care here, the vertices of $\gamma$ appearing on the left and right sides (that is, at the changeovers between upper boundary points and the lower boundary points) do not count in either of the products.
One can readily check that this intertwines the actions of planar tangles on the spherical and lopsided graph planar algebras according to the formula
\begin{equation}
\label{eq:intertwiner2}
\natural(T(z_i)) = \sqrt{\delta}^{\; n} \sqrt{\delta}^{\; -m} T(\natural(z_i)),
\end{equation}
where again $n$ is the signed count (minimums are positive, maximums are negative) of critical points which are shaded above in $T$ and $m$ is the signed count of critical points shaded below.

Note that this intertwining condition means that we can locate the lowest weight spaces, or rotational eigenspaces, using the lopsided graph planar algebra, where arithmetic is easier. In particular, $\natural$ restricts to an isomorphism between the lowest weight spaces, and an isomorphism between each rotational eigenspace.

If one transported the action of planar tangles on $\cG^{\text{lopsided}}$ across to $\cG^{\text{spherical}}$ via the map $\natural$ and its inverse, this formula shows that the action is a rescaling (as described in the previous section) of the usual action of planar tangles on $\cG^{\text{lopsided}}$ with $x = \delta ^{1/2}$ and $y=\delta^{-1/2}$. 

Nevertheless, for our purposes it wouldn't have been enough to simply rescale the original action. The map $\natural$, which rescales the basis, allows us to define the lopsided graph planar algebra over the field $\mathbb{Q}(d^{\text{lopsided}})$ generated by the lopsided dimensions. (Often, but not always, this is no bigger than the field $\mathbb{Q}(\delta^2)$ generated by the index of the graph.) The spherical graph planar algebra is defined instead over the field $\mathbb{Q}\left(\left\{\sqrt{d_a / d_b} \; \mid  \text{$a$ and $b$ adjacent}\right\}\right)$ which is generally much much larger. Indeed, usually it's impossible to identify a single generator of this resulting number field.

Finally, we define the $*$ action on the lopsided graph planar algebra simply by transporting across the $*$ action form the spherical graph planar algebra, via $\natural$ and $\natural^{-1}$. Explicitly, this gives
$$\gamma^* = \left(\prod_{{\text{$a$ above}}} d^{\text{lopsided}}_a \right)\left( \prod_{{\text{$b$ below}}} {d^{\text{lopsided}}_b}\right)^{-1}\operatorname{reverse}(\gamma)$$
on loops, extending antilinearly to the entire space.

\subsection{The two-sided graph planar algebra}\label{twosided}
Usually, the graph planar algebra is defined in terms of a single principal graph. We now introduce the `two-sided' graph planar algebra for a pair of principal graphs (with dual data) $(\Gamma, \Gamma')$. (The name comes from an interpretation of graph planar algebras, connections and flatness coming from Turaev-Viro theory; c.f. \cite{tvc}). The two-sided graph planar algebra has region colours and strand types indexed by the square 
$$\begin{tikzpicture}
	\node(NN) [rectangle, fill=red-unshaded] at (0,0) {$N-N$};
	\node(NM) [rectangle, fill=red-shaded] at (0,1) {$N-M$};
	\node(MM) [rectangle, fill=blue-unshaded] at (2,1) {$M-M$};
	\node(MN) [rectangle, fill=blue-shaded] at (2,0) {$M-N$};
	\draw (NN)--(NM)--(MM)--(MN)--(NN);
\end{tikzpicture}
$$
That is, there is a vector space $\cG_\pi$ for each closed loop $\pi$ on this square (namely, each sequence of `unshaded red', `shaded red', `unshaded blue' and `shaded blue', subject to the condition that the shadings alternate).  This vector space has basis given by the loops in the 4-partite graph for $(\Gamma, \Gamma')$ which descend to the  loop $\pi$ on the square. The even and odd vertices of $\Gamma$ lie over the unshaded red and shaded red vertices of the square, while the even and odd vertices of $\Gamma'$ lie over the unshaded blue and shaded blue vertices of the square.

As before we have two versions of the two-sided graph planar algebra, which we call spherical and lopsided.
The two actions of planar tangles are exactly as above for the one-sided graph planar algebra, with the obvious restriction that the labelings $b$ in Equation \eqref{eq:action} above respect the four different shadings in $T$, and a further interpretation of the quantity $\operatorname{shading}(a)$: this is $0$ if the shading is $N-N$ or $M-M$, $+1$ when the shading is $N-M$ and $-1$ when the shading is $M-M$.

The notion of a connection was originally formulated by Ocneanu in \cite{MR996454}.   Notice that a connection, as usually defined, is exactly an element $K $ of the space $\cG^{\text{spherical}}_\zeta$, where $\zeta$ is the loop `unshaded red, shaded red, unshaded blue, shaded blue'. 
The renormalization axiom is no longer an axiom; it is just a statement about the one-click rotation of a four-box.  
The biunitarity condition is a pair of planar equations
$$
\begin{tikzpicture}[scale=.7, baseline=0]
	\clip (-1,-1.8) rectangle (1,1.8);

	\fill[fill=red-unshaded] (-1,-2)--(-.3,-2)--(-.3,2)--(-1,2);
	\fill[fill=red-shaded] (-.3,-2)--(-.3,-1)--(.3,-1)--(.3,-2);
	\fill[fill=red-shaded] (-.3,2)--(-.3,1)--(.3,1)--(.3,2);
	\fill[blue-unshaded] (1,-2)--(.3,-2)--(.3,3)--(1,2);
	\fill[blue-shaded] (-.3,-1)--(-.3,1)--(.3,1)--(.3,-1);
	
	\draw (-.3,-2)--(-.3,2);
	\draw (.3,-2)--(.3,2);
	
	\node[conn, minimum size=8mm] at (0,1) {$K$};
	\node[conn,  minimum size=8mm] at (0,-1) {$K^*$};
	
\end{tikzpicture}
=
\begin{tikzpicture}[scale=.7, baseline=0]
	\clip (-1,-1.8) rectangle (1,1.8);

	\fill[fill=red-unshaded] (-1,-2)--(-.3,-2)--(-.3,2)--(-1,2);
	\fill[fill=red-shaded] (-.3,-2)--(-.3,2)--(.3,2)--(.3,-2);
	\fill[blue-unshaded] (1,-2)--(.3,-2)--(.3,3)--(1,2);
	
	\draw (-.3,-2)--(-.3,2);
	\draw (.3,-2)--(.3,2);
\end{tikzpicture}
\quad
\text{and}
\quad
\begin{tikzpicture}[scale=.7,baseline=0]
	\clip (-1.9,-1) rectangle (1.9,1);

	\fill[red-shaded] (-2,1) rectangle (2,0);
	\fill[blue-shaded] (-2,-1) rectangle (2,0);

	\filldraw[fill=blue-unshaded] (-1,.3) .. controls (0,.6) .. (1,.3)--(1,-.3) .. controls (0,-.6) .. (-1,-.3) -- (-1,.3);
	\filldraw[fill=red-unshaded] (-2,.6) -- (-1,.3) -- (-1,-.3) -- (-2,-.6);
	\filldraw[fill=red-unshaded] (2,.6) -- (1,.3) -- (1,-.3) -- (2,-.6);
	
	\node[conn, minimum size=8mm] at (-1,0) {$K$};
	\node[conn, minimum size=8mm] at (1,0) {$K^*$};

\end{tikzpicture}
=
\begin{tikzpicture}[scale=.7, baseline=0]
	\clip (-1.9,-1) rectangle (1.9,1);
	
	\fill[red-shaded] (-2,1) rectangle (2,0);
	\fill[blue-shaded] (-2,-1) rectangle (2,0);

	\filldraw[fill=red-unshaded] (-2,-.6) .. controls (0,0) .. (2,-.6)--(2,.6) .. controls (0,0) .. (-2,.6);
\end{tikzpicture} \; ,
$$
where $K^*$ is defined using the usual $*$-structure (i.e., $K^*(e_1 e_2 e_3 e_4) = \overline{K(e_4 e_3 e_2 e_1)}$.)

Let $\cG_{\text{red}}$ (respectively $\cG_{\text{blue}}$) be the space indexed by paths which alternate between the two shades of red (respectively blue).  Note that $\cG_{\text{red}}$ is a copy of the graph planar algebra of the principal graph $\Gamma$, and $\cG_{\text{ blue}}$ is the graph planar algebra for the dual graph $\Gamma'$.
We say that a pair of elements $(x, y) \in \cG_{\text{ blue}} \times  \cG_{\text{ red}}$ is {\em flat with respect to a connection $K$} if
\begin{equation*}
\xabove = \ybelow.
\end{equation*}
(This picture illustrates flatness of four-boxes.  The $2n$-box generalization has $n$ copies of the connection or its star sitting above $x$ and below $y$.)

Using the map $\natural$ we can push a connection across to the lopsided analogue of the two-sided graph planar algebra. It is still a biunitary, although we have to be careful because the spherical $*$ structure transported across via $\natural$ is now more complicated.  On loops, it is 
$$ \gamma^* = \prod_{{\text{$a$ above}}} d^{\text{lopsided}}_a \cdot \prod_{{\text{$b$ below}}} {d^{\text{lopsided}}_b}^{-1}  \cdot \operatorname{reverse}(\gamma).$$

It is worth noting at this point that a pair of elements $(x,y)$ is flat with respect to a connection $K$ in the spherical graph planar algebra exactly if $(\natural(x), \natural(y))$ is flat with respect to $\natural(K)$ in the lopsided graph planar algebra. This will be essential to our later calculations.

We abuse notation by saying that $x$ itself is flat if there exists a $y$ so that $(x,y)$ is flat. (The element $y$ is necessarily uniquely determined!) Notice that bi-unitarity immediately implies that every Temperley-Lieb diagram is flat. The flat elements  form a sub planar algebra which we call the flat subalgebra

We will need the following results about flat elements:
\begin{fact}\label{fact:flat-subalg}
If there are no flat elements besides Temperley-Lieb for a connection $K$ in each of the spaces $P_{j,+}$ for $j < n$, and a $k$-dimensional space of flat lowest weight vectors in $P_{n,+}$, then the flat subalgebra is $n-1$ supertransitive with excess $k$.
\end{fact}

\begin{thm}\label{thm:spa-and-flat-conns}
Any subfactor planar algebra $P$ defines a bi-unitary connection $K(P) \in \cG(\Gamma(P))_\zeta$, and $P$ is isomorphic to the flat subalgebra for $K(P)$ inside $\cG(\Gamma(P))$.
\end{thm}
Although not stated in this language, this result is well known and appears in \cite{MR1642584}. See also \cite{tvc} for a proof via Turaev-Viro theory.

\subsection{The gauge group}
The {\em gauge group} for a given 4-partite graph $\Xi$ is a copy of the unit circle for every edge in $\Xi$, i.e., 
$$\text{Gauge}(\Xi) = \{g: \Xi \rightarrow \mathbb{T} \} \simeq \mathbb{T}^{|E(\Xi)|}.$$  
Elements of the gauge group can be thought of as $2$-boxes in the graph planar algebra:  If $g \in \text{Gauge}(\Xi)$ and $\gamma$ is the 2-loop going through edges $e_1$ and $e_2$, then $g (\gamma) = \delta_{e_1,e_2} g(e_1) $. 
Thus, $\text{Gauge}(\Xi)$ acts on the graph planar algebra of $\Xi$, by 
$$
\begin{tikzpicture}[baseline=.5cm]
	\draw (0,-.8)--(0,1.8);
	\draw (1,-.8)--(1,1.8);
	\draw (2,-.8)--(2,1.8);
	\filldraw[fill=white, thick, rounded corners = 4mm] (-.5,0) rectangle (2.5,1);
	\node at (1,.5) {$X$};
\end{tikzpicture}
\mapsto
\begin{tikzpicture}[baseline=.5cm]
	\draw (0,-.8)--(0,1.8);
	\draw (1,-.8)--(1,1.8);
	\draw (2,-.8)--(2,1.8);
	\filldraw[fill=white, thick, rounded corners = 4mm] (-.5,0) rectangle (2.5,1);
	\node at (1,.5) {$X$};

	\node[gauge] at (0, -.4) {$g$};
	\node[gauge] at (0, 1.4) {$g$};
	\node[gauge] at (1, -.4) {$g$};
	\node[gauge] at (1, 1.4) {$g$};
	\node[gauge] at (2, -.4) {$g$};
	\node[gauge] at (2, 1.4) {$g$};

\end{tikzpicture}
$$

If $K$ is a biunitary connection, then $gK$ is again a biunitary connection.  Let a black bead represent $g$, and a white bead represent $g^*=g^{-1}$; then the first biunitarity equation for $gK$ holds:
$$
\begin{tikzpicture}[scale=.7, baseline=0]
	\clip (-1,-2) rectangle (1,2);

	\fill[fill=red-unshaded] (-1,-2)--(-.3,-2)--(-.3,2)--(-1,2);
	\fill[fill=red-shaded] (-.3,-2)--(-.3,-1)--(.3,-1)--(.3,-2);
	\fill[fill=red-shaded] (-.3,2)--(-.3,1)--(.3,1)--(.3,2);
	\fill[blue-unshaded] (1,-2)--(.3,-2)--(.3,3)--(1,2);
	\fill[blue-shaded] (-.3,-1)--(-.3,1)--(.3,1)--(.3,-1);
	
	\draw (-.3,-2)--(-.3,2);
	\draw (.3,-2)--(.3,2);
	
	\node[conn, minimum size=8mm] at (0,1) {$K$};
	\node[conn, minimum size=8mm] at (0,-1) {$K^*$};
	
	\node[invgauge] at (-.3,1.8) {};
	\node[invgauge] at (.3,1.8) {};
	\node[invgauge] at (-.3,.2) {};
	\node[invgauge] at (.3,.2) {};
	
	\node[gauge] at (-.3,-1.8) {};
	\node[gauge] at (.3,-1.8) {};
	\node[gauge] at (-.3,-.2) {};
	\node[gauge] at (.3,-.2) {};
	
\end{tikzpicture}
=
\begin{tikzpicture}[scale=.7, baseline=0]
	\clip (-1,-2) rectangle (1,2);

	\fill[fill=red-unshaded] (-1,-2)--(-.3,-2)--(-.3,2)--(-1,2);
	\fill[fill=red-shaded] (-.3,-2)--(-.3,-1)--(.3,-1)--(.3,-2);
	\fill[fill=red-shaded] (-.3,2)--(-.3,1)--(.3,1)--(.3,2);
	\fill[blue-unshaded] (1,-2)--(.3,-2)--(.3,3)--(1,2);
	\fill[blue-shaded] (-.3,-1)--(-.3,1)--(.3,1)--(.3,-1);
	
	\draw (-.3,-2)--(-.3,2);
	\draw (.3,-2)--(.3,2);
	
	\node[conn, minimum size=8mm] at (0,1) {$K$};
	\node[conn, minimum size=8mm] at (0,-1) {$K^*$};
	
	\node[invgauge] at (-.3,1.8) {};
	\node[invgauge] at (.3,1.8) {};
	
	\node[gauge] at (-.3,-1.8) {};
	\node[gauge] at (.3,-1.8) {};
	
\end{tikzpicture}
=
\begin{tikzpicture}[scale=.7, baseline=0]
	\clip (-1,-2) rectangle (1,2);

	\fill[fill=red-unshaded] (-1,-2)--(-.3,-2)--(-.3,2)--(-1,2);
	\fill[fill=red-shaded] (-.3,-2)--(-.3,2)--(.3,2)--(.3,-2);
	\fill[blue-unshaded] (1,-2)--(.3,-2)--(.3,3)--(1,2);
	
	\draw (-.3,-2)--(-.3,2);
	\draw (.3,-2)--(.3,2);
	
	\node[invgauge] at (-.3,1.8) {};
	\node[invgauge] at (.3,1.8) {};
	
	\node[gauge] at (-.3,-1.8) {};
	\node[gauge] at (.3,-1.8) {};

\end{tikzpicture}
=
\begin{tikzpicture}[scale=.7, baseline=0]
	\clip (-1,-2) rectangle (1,2);

	\fill[fill=red-unshaded] (-1,-2)--(-.3,-2)--(-.3,2)--(-1,2);
	\fill[fill=red-shaded] (-.3,-2)--(-.3,2)--(.3,2)--(.3,-2);
	\fill[blue-unshaded] (1,-2)--(.3,-2)--(.3,3)--(1,2);
	
	\draw (-.3,-2)--(-.3,2);
	\draw (.3,-2)--(.3,2);
\end{tikzpicture}
$$ 
and the second biunitarity equation is verified similarly.

The {\em complex gauge group} for a given 4-partite graph $\Xi$ is a copy of the non-zero complex numbers for every edge in $\Xi$, i.e., 
$$\text{ComplexGauge}(\Xi)= \{g: \Xi \rightarrow  (\mathbb{C}^{\times})\} \simeq (\mathbb{C}^{\times})^{|E(\Xi)|}.$$  
Again, elements of the complex gauge group are $2$-boxes in the graph planar algebra, and act accordingly.   For $g$ a complex gauge group element, $gK$ is no longer necessarily a biunitary connection, but it is bi-invertible and this is often sufficient.

We define $\text{Alt}(g)(X)$ to be the result of surrounding $X$ with alternating $g$'s and $g^{-1}$'s. It's easy to see that $\operatorname{Alt}(g)$ is an isomorphism of planar algebras (it fixes Temperley-Lieb diagrams, and commutes with disjoint union and applying caps). In fact, $\text{Alt}(g)$ acts trivially on any non-essential loop on $\Gamma$, and so for principal graphs without loops this action is always trivial.

\begin{lem}
\label{lem:flat-elements-transform}
Let $X$ be an element in the graph planar algebra of $\Xi$,  $K$ be a connection on $\Xi$, and $g$ be an element of $\text{ComplexGauge}(\Xi)$.
$X$ is flat with respect to $K$ if and only if $\text{Alt}(g^{-1})(X)$ is flat with respect to $gK$.  
\end{lem}

\begin{proof}
Letting a black bead represent $g \in \text{ComplexGauge}(\Xi)$, and a white bead be $g^{-1}$, flatness of $X$ is the assertion that there is a $Y$ such that
$$
\begin{tikzpicture}[baseline=1.5cm]
	\draw (0,-1)--(0,3.5);
	\draw (1.5,-1)--(1.5,3.5);
	\draw (3,-1)--(3,3.5);
	\filldraw[fill=white, thick, rounded corners = 4mm] (-.5,0) rectangle (3.5,1);
	\node at (1.5,.5) {$X$};
	
	\draw (-1,3.5) .. controls (-1,2.9) .. (0,2.6) -- (3,2) .. controls (4,1.7) .. (4,0) -- (4,-1);	

	\node[conn, minimum size=7mm] at (0,2.6) {\tiny$K$};
	\node[conn, minimum size=7mm] at (1.5,2.3) {\tiny$K^{{\hbox{-}\!}1}$};
	\node[conn, minimum size=7mm] at (3,2) {\tiny$K$};
\end{tikzpicture}
\quad
=
\quad
\begin{tikzpicture}[baseline=3.5cm]
	\draw (0,1)--(0,5.5);
	\draw (1.5,1)--(1.5,5.5);
	\draw (3,1)--(3,5.5);
	\filldraw[fill=white, thick, rounded corners = 4mm] (-.5,4) rectangle (3.5,5);
	\node at (1.5,4.5) {$Y$};
	
	\draw (-1,5.5) -- (-1,3.5) .. controls (-1,2.9) .. (0,2.6) -- (3,2) .. controls (4,1.7) .. (4,1);	

	\node[conn, minimum size=7mm] at (0,2.6) {\tiny$K$};
	\node[conn, minimum size=7mm] at (1.5,2.3) {\tiny$K^{{\hbox{-}\!}1}$};
	\node[conn, minimum size=7mm] at (3,2) {\tiny$K$};
\end{tikzpicture}
$$

This is true if and only if
$$
\begin{tikzpicture}[baseline=1.5cm]
	\draw (0,-1)--(0,3.5);
	\draw (1.5,-1)--(1.5,3.5);
	\draw (3,-1)--(3,3.5);
	\filldraw[fill=white, thick, rounded corners = 4mm] (-.5,0) rectangle (3.5,1);
	\node at (1.5,.5) {$X$};
	
	\draw (-1,3.5) .. controls (-1,2.9) .. (0,2.6) -- (3,2) .. controls (4,1.7) .. (4,0) -- (4,-1);	

	\node[conn, minimum size=7mm] at (0,2.6) {\tiny$K$};
	\node[conn, minimum size=7mm] at (1.5,2.3) {\tiny$K^{{\hbox{-}\!}1}$};
	\node[conn, minimum size=7mm] at (3,2) {\tiny$K$};

	\node[invgauge] at (4, -.5) {};	
	\node[invgauge] at (0, -.5) {};
	\node[gauge] at (1.5, -.5) {};
	\node[invgauge] at (3, -.5) {};
	\node[invgauge] at (0, 3.2) {};
	\node[gauge] at (1.5, 2.9) {};
	\node[invgauge] at (3, 2.6) {};
	\node[invgauge] at (-1, 3.2) {};

\end{tikzpicture}
\quad
=
\quad
\begin{tikzpicture}[baseline=3.5cm]
	\draw (0,1)--(0,5.5);
	\draw (1.5,1)--(1.5,5.5);
	\draw (3,1)--(3,5.5);
	\filldraw[fill=white, thick, rounded corners = 4mm] (-.5,4) rectangle (3.5,5);
	\node at (1.5,4.5) {$Y$};
	
	\draw (-1,5.5) -- (-1,3.5) .. controls (-1,2.9) .. (0,2.6) -- (3,2) .. controls (4,1.7) .. (4,1);	

	\node[conn, minimum size=7mm] at (0,2.6) {\tiny$K$};
	\node[conn, minimum size=7mm] at (1.5,2.3) {\tiny$K^{{\hbox{-}\!}1}$};
	\node[conn, minimum size=7mm] at (3,2) {\tiny$K$};
	
	\node[invgauge] at (-1, 5.3) {};	
	\node[invgauge] at (0, 5.3) {};
	\node[gauge] at (1.5, 5.3) {};
	\node[invgauge] at (3, 5.3) {};
	\node[invgauge] at (0, 1.9) {};
	\node[gauge] at (1.5, 1.6) {};
	\node[invgauge] at (3, 1.3) {};
	\node[invgauge] at (4, 1.3) {};
\end{tikzpicture}
$$

And inserting lots of instances of the relation $1=g \cdot g^{-1}$ along strands, we see that this is equivalent to the equality
$$
\begin{tikzpicture}[baseline=1.5cm]
	\draw (0,-1)--(0,3.5);
	\draw (1.5,-1)--(1.5,3.5);
	\draw (3,-1)--(3,3.5);
	\filldraw[fill=white, thick, rounded corners = 4mm] (-.5,0) rectangle (3.5,1);
	\node at (1.5,.5) {$X$};
	
	\draw (-1,3.5) .. controls (-1,2.9) .. (0,2.6) -- (3,2) .. controls (4,1.7) .. (4,0) -- (4,-1);	

	\node[conn, minimum size=7mm] at (0,2.6) {\tiny$K$};
	\node[conn, minimum size=7mm] at (1.5,2.3) {\tiny$K^{{\hbox{-}\!}1}$};
	\node[conn, minimum size=7mm] at (3,2) {\tiny$K$};

	\node[invgauge] at (4, -.5) {};	
	\node[invgauge] at (0, -.5) {};
	\node[gauge] at (1.5, -.5) {};
	\node[invgauge] at (3, -.5) {};
	\node[invgauge] at (0, 3.2) {};
	\node[gauge] at (1.5, 2.9) {};
	\node[invgauge] at (3, 2.6) {};
	\node[invgauge] at (-1, 3.2) {};
	
	\node[invgauge] at (.6,2.48) {};
	\node[gauge] at (.9,2.42) {};
	
	\node[gauge] at (2.1,2.18) {};
	\node[invgauge] at (2.4,2.12) {};
	
	\node[gauge] at (0,1.2) {};
	\node[invgauge] at (0,1.5) {};

	\node[invgauge] at (1.5,1.2) {};
	\node[gauge] at (1.5,1.5) {};

	\node[gauge] at (3,1.2) {};
	\node[invgauge] at (3,1.5) {};

\end{tikzpicture}
\quad
=
\quad
\begin{tikzpicture}[baseline=3.5cm]
	\draw (0,1)--(0,5.5);
	\draw (1.5,1)--(1.5,5.5);
	\draw (3,1)--(3,5.5);
	\filldraw[fill=white, thick, rounded corners = 4mm] (-.5,4) rectangle (3.5,5);
	\node at (1.5,4.5) {$Y$};
	
	\draw (-1,5.5) -- (-1,3.5) .. controls (-1,2.9) .. (0,2.6) -- (3,2) .. controls (4,1.7) .. (4,1);	

	\node[conn, minimum size=7mm] at (0,2.6) {\tiny$K$};
	\node[conn, minimum size=7mm] at (1.5,2.3) {\tiny$K^{{\hbox{-}\!}1}$};
	\node[conn, minimum size=7mm] at (3,2) {\tiny$K$};
	
	\node[invgauge] at (-1, 5.3) {};	
	\node[invgauge] at (0, 5.3) {};
	\node[gauge] at (1.5, 5.3) {};
	\node[invgauge] at (3, 5.3) {};
	\node[invgauge] at (0, 1.9) {};
	\node[gauge] at (1.5, 1.6) {};
	\node[invgauge] at (3, 1.3) {};
	\node[invgauge] at (4, 1.3) {};
	
	\node[invgauge] at (.6,2.48) {};
	\node[gauge] at (.9,2.42) {};
	
	\node[gauge] at (2.1,2.18) {};
	\node[invgauge] at (2.4,2.12) {};
	
	\node[gauge] at (0,3.8) {};
	\node[invgauge] at (0,3.5) {};

	\node[invgauge] at (1.5,3.8) {};
	\node[gauge] at (1.5,3.5) {};

	\node[gauge] at (3,3.8) {};
	\node[invgauge] at (3,3.5) {};

\end{tikzpicture}
$$
\end{proof}

In fact, this argument easily shows that two gauge equivalent connections give isomorphic planar algebras; $\operatorname{Alt}(g)$ provides the isomorphism between the corresponding flat elements. Conversely, if two connections give the same planar algebra, they must be gauge equivalent, although we will not need this here.

\section{Classification of 1-supertransitive subfactors with index in the range $(5,3+\sqrt{5})$.}
 \label{1STclassification}

Our main result in this section is

\begin{thm}\label{thm:only1ST}
The only 1-supertransitive subfactors with index in the range $(5, 3+\sqrt{5})$ have principal graph $\left(\bigraph{bwd1v1p1v1x0p1x1duals1v1x2}, \bigraph{bwd1v1p1v1x0p1x1duals1v1x2}\right)$.
\end{thm}
The proof uses a small amount of computer enumeration of possible principal graphs, using the {\tt FusionAtlas`} software described in \cite{index5-part1}, although this could easily be replaced by tedious hand calculations. 
We separately prove the same statement with an additional hypothesis of finite depth as Proposition \ref{1STfinitedepth}, because we think the proof is interesting.

In fact, we think a much stronger statement holds
\begin{conj}
\label{conj}
Any subfactor with index in the range $(5, 3+\sqrt{5})$ has principal graph $A_\infty$,  $\left(\bigraph{bwd1v1p1v1x0p1x1duals1v1x2}, \bigraph{bwd1v1p1v1x0p1x1duals1v1x2}\right)$ or $\left(\bigraph{bwd1v1v1p1v1x0p0x1p0x1v0x1x0p1x0x1duals1v1v2x1x3}, \bigraph{bwd1v1v1p1v1x0p0x1p0x1v0x1x0p1x0x1duals1v1v2x1x3}\right)$.
\end{conj}
Computer evidence from the {\tt FusionAtlas`} enables us to prove the following theorem in support of this conjecture, but we won't give the details here.
\begin{thm}
\label{thm:conj}
Any subfactor with index in range $(5, 3+\sqrt{5})$ either appears in Conjecture \ref{conj} or has rank at least 38.
\end{thm}

To establish Theorem \ref{thm:only1ST} we need a few preliminary results.
The following lemma is stated in \cite{index5-part1} and was known to experts well before then.  It probably follows from \cite{MR1262294}; we find the following proof more direct.  

\begin{lem}\label{intermediate}
A subfactor $N\subset M$ in which  some, but not all, depth-two projections have dimension 1 has a proper intermediate subfactor.   
\end{lem}

\begin{proof}
This proof is written in the language of algebra objects for tensor categories; a dictionary between subfactors and algebra objects is given in \cite[\S 2]{1102.2631}.

$M$ is an algebra object in the category of $N-N$ bimodules; it is isomorphic to $ \sum_{p \in P_4} p$, where $P_4$ is the set of irreducible summands of $M$ as an $N-N$ bimodule (ie, vertices at depth zero and two in the principal graph).  

Let $P'_4$ be the subset of $P_4$ consisting of objects of dimension 1.  The objects in $P'_4$ form a group under tensor product, as the tensor product of two projections is again a projection, and dimension is multiplicative under tensor product.  

Thus, $ \sum_{p \in P'_4} p$ is also an algebra object, and as it is intermediate between $N \simeq 1$ and $M \simeq  \sum_{p \in P_4} p$, it corresponds to an intermediate subfactor for $N \subset M$.  
\end{proof}

\begin{lem}\label{lem:depth2dim2}
A subfactor which has index in $(5,3+\sqrt{5})$ and an object at depth two with dimension in $(1,2)$ has principal graph $A=\bigraph{gbg1v1p1v1x1p0x1}$.
\end{lem}

\begin{proof}
Let $X$ be the generating object for this subfactor, and $Z$ be the object at depth $2$ with  $1 < \dim{Z} < 2$.  Then $Z$ generates an ADET fusion category.  Since $X \in X \tensor Z$, the principal graph for $- \tensor Z$ must contain a $T_n$, with $X$ the last vertex.  If $X_0$ is the first vertex in the $T_n$, we have $\dim{X}/ \dim{X_0} = [n]_{q=\zeta_{4n+2}}$.

{\em Case 2:} If $n=2$ and if $\dim X_0 = 1$, then $\dim{X} < 2$, so we were below index 4. But if $\dim X_0 \geq \sqrt{2}$, then the index is at least $3+\sqrt{5}$.

{\em Case 3:} If $n=3$, then the only possible dimensions are $\dim{X_0}=1$, and $\dim{X} \approx 2.247$ is the largest root of $x^3 - 2 x^2 - x + 1$.  We want to know what the $X\otimes -$ principal graph is.  

The dimensions of the objects at depth two in the  $X\otimes -$ principal graph are all at least $\sqrt{2}$, and because the index is $\dim{X}^2 \approx 5.049$, there are exactly two depth-two vertices in the principal graph.  One is $Z$; call the other $Y$.  We know $\dim{Z}$ is the graph norm of $T_3$, and then $\dim{Y}$ must be $\dim{X}$.  We abbreviate these as
\begin{align*}
d & = \dim{X} = \dim{Y} \approx 2.247; \\
e & = \dim{Z} = d^2-d-1 \approx 1.802.
\end{align*}

In the following illustrations of the forms of principal graphs, filled-in circles are vertices whose neighbors are all known, while open circles are vertices which may connect to other vertices at the next depth.

 The triple-point obstruction tells us that if our principal graph were of the first form below, then our dual principal graph would be of the second form:
$$
\left(
\begin{tikzpicture}[baseline=0]
	\draw (0,0)--(1,0)--(2,-.5)--(3,-.5);
	\draw (1,0)--(2,.5)--(3,.5);

	\draw[fill=black] (0,0) circle (1mm);
	\draw[fill=black] (1,0) circle (1mm);
	\draw[fill=black] (2,-.5) circle (1mm);
	\draw[fill=black] (2,.5) circle (1mm);
	\draw[fill=white] (3,-.5) circle (1mm);
	\draw[fill=white] (3,.5) circle (1mm);
\end{tikzpicture}
\quad , \quad
\begin{tikzpicture}[baseline=0]
	\draw (0,0)--(1,0)--(2,-.5)--(3,-.5);
	\draw (1,0)--(2,.5);
	\draw (2,-.5)--(3,.5);

	\draw[fill=black] (0,0) circle (1mm);
	\draw[fill=black] (1,0) circle (1mm);
	\draw[fill=black] (2,-.5) circle (1mm);
	\draw[fill=black] (2,.5) circle (1mm);
	\draw[fill=white] (3,-.5) circle (1mm);
	\draw[fill=white] (3,.5) circle (1mm);
\end{tikzpicture}.
\right)
$$
But the univalent vertex at depth two has dimension 1, and  we get a contradiction because this subfactor cannot have an intermediate subfactor (as the index does not factor into allowed index values).  

Thus, our principal graph is either of the form
$$
\begin{tikzpicture}[baseline=0]
	\draw (0,0)--(1,0)--(2,-.5)--(3,0);
	\draw (1,0)--(2,.5)--(3,0);

	\draw[fill=black] (0,0) circle (1mm);
	\node at (0,0) [below left] {$1$};
	\draw[fill=black] (1,0) circle (1mm);
	\node at (1,0) [below left] {$X$};
	\draw[fill=white] (2,-.5) circle (1mm);
	\node at (2,-.5) [below left] {$Z$};
	\draw[fill=white] (2,.5) circle (1mm);	
	\node at (2,.5) [above left] {$Y$};
	\draw[fill=white] (3,0) circle (1mm);
\end{tikzpicture} \quad ,
\quad \text{or} \quad
\begin{tikzpicture}[baseline=0]
	\draw (0,0)--(1,0)--(2,-.5)--(3,-.5);
	\draw (2,.5)--(3,0);
	\draw (1,0)--(2,.5)--(3,.5);
	\draw[fill=black] (0,0) circle (1mm);
	\node at (0,0) [below left] {$1$};
	\draw[fill=black] (1,0) circle (1mm);
	\node at (1,0) [below left] {$X$};
	\draw[fill=white] (2,-.5) circle (1mm);
	\node at (2,-.5) [below left] {$Z$};
	\draw[fill=white] (2,.5) circle (1mm);
	\node at (2,.5) [above left] {$Y$};
	\draw[fill=white] (3,-.5) circle (1mm);
	\draw[fill=white] (3,0) circle (1mm);
	\draw[fill=white] (3,.5) circle (1mm);
\end{tikzpicture} \quad .
$$

In both cases, note that $\dim{Z} \dim{X} - \dim{X} = de-d = e < 2$, so $Z$ must have degree two.  In the first case, dimensions imply $Y$ also connects to a vertex of dimension 1.  In the second case, $\dim{Y} \dim{X} - \dim{X} = d^2-d \approx 2.802$ can only be partitioned into a sum of allowable dimensions as $e +1$, whence $Y$ connects to three other vertices:  $X$, a vertex of dimension $e$ and a vertex of dimension 1.  These stronger restrictions say our graph is actually of the form
$$
\begin{tikzpicture}[baseline=0]
	\draw (0,0)--(1,0)--(2,-.5)--(3,0);
	\draw (1,0)--(2,.5)--(3,0);
	\draw (2,.5)--(3,.5);

	\draw[fill=black] (0,0) circle (1mm);
	\node at (0,0) [below left] {$1$};
	\node at (0,0) [above left, blue] {$1$};
	\draw[fill=black] (1,0) circle (1mm);
	\node at (1,0) [below left] {$X$};
	\node at (1,0) [above left, blue] {$d$};
	\draw[fill=black] (2,-.5) circle (1mm);
	\node at (2,-.5) [below ] {$Z$};
	\node at (2,-.5) [above , blue] {$e$};
	\draw[fill=black] (2,.5) circle (1mm);
	\node at (2,.5) [below ] {$Y$};
	\node at (2,.5) [above right, blue] {$d$};
	\draw[fill=white] (3,0) circle (1mm);
	\node at (3,0) [ right, blue] {$e$};
	\draw[fill=black] (3,.5) circle (1mm);
	\node at (3,.5) [ right, blue] {$1$};

\end{tikzpicture} \quad ,
\quad
\text{or}
\quad
\begin{tikzpicture}[baseline=0]
	\draw (0,0)--(1,0)--(2,-.5)--(3,-.5);
	\draw (2,.5)--(3,0);
	\draw (1,0)--(2,.5)--(3,.5);
	\draw[fill=black] (0,0) circle (1mm);
	\node at (0,0) [below left] {$1$};
	\node at (0,0) [above left, blue] {$1$};
	\draw[fill=black] (1,0) circle (1mm);
	\node at (1,0) [below left] {$X$};
	\node at (1,0) [above left, blue] {$d$};
	\draw[fill=black] (2,-.5) circle (1mm);
	\node at (2,-.5) [below ] {$Z$};
	\node at (2,-.5) [above right, blue] {$e$};
	\draw[fill=black] (2,.5) circle (1mm);
	\node at (2,.5) [below ] {$Y$};
	\node at (2,.5) [above right, blue] {$d$};
	\draw[fill=white] (3,-.5) circle (1mm);
	\node at (3,-.5) [ right, blue] {$e$};
	\draw[fill=white] (3,0) circle (1mm);
	\node at (3,0) [ right, blue] {$e$};
	\draw[fill=white] (3,.5) circle (1mm);
	\node at (3,.5) [ right, blue] {$1$};
\end{tikzpicture} \quad .
$$

In the first case, the graph 
$$
\begin{tikzpicture}[baseline=0]
	\draw (0,0)--(1,0)--(2,-.5)--(3,0);
	\draw (1,0)--(2,.5)--(3,0);
	\draw (2,.5)--(3,.5);

	\draw[fill=black] (0,0) circle (1mm);
	\node at (0,0) [below left] {$1$};
	\draw[fill=black] (1,0) circle (1mm);
	\node at (1,0) [below left] {$X$};
	\draw[fill=black] (2,-.5) circle (1mm);
	\node at (2,-.5) [below left] {$Z$};
	\draw[fill=black] (2,.5) circle (1mm);
	\node at (2,.5) [above left] {$Y$};
	\draw[fill=black] (3,0) circle (1mm);
	\draw[fill=black] (3,.5) circle (1mm);

\end{tikzpicture}
$$
has the correct norm, and is therefore is the only possible graph of this form.  

It remains to rule out  graphs of the first form. To do so, we need to move from considering bigraphs to considering pairs of bigraphs with dual data.  First, name the vertices at depth three $A$, $B$ and $C$:
$$
\begin{tikzpicture}[baseline=0]
	\draw (0,0)--(1,0)--(2,-.5)--(3,-.5);
	\draw (2,.5)--(3,0);
	\draw (1,0)--(2,.5)--(3,.5);
	\draw[fill=black] (0,0) circle (1mm);
	\node at (0,0) [below left] {$1$};
	\node at (0,0) [above left, blue] {$1$};
	\draw[fill=black] (1,0) circle (1mm);
	\node at (1,0) [below left] {$X$};
	\node at (1,0) [above left, blue] {$d$};
	\draw[fill=black] (2,-.5) circle (1mm);
	\node at (2,-.5) [below ] {$Z$};
	\node at (2,-.5) [above right, blue] {$e$};
	\draw[fill=black] (2,.5) circle (1mm);
	\node at (2,.5) [below ] {$Y$};
	\node at (2,.5) [above right, blue] {$d$};
	\draw[fill=white] (3,-.5) circle (1mm);
	\node at (3,-.5) [ right, blue] {$e$};
	\draw[fill=white] (3,0) circle (1mm);
	\node at (3,0) [ right, blue] {$e$};
	\draw[fill=white] (3,.5) circle (1mm);
	\node at (3,.5) [ right, blue] {$1$};
	\node at (3,-.5) [below] {$C$};
	\node at (3,0) [left] {$B$};
	\node at (3,.5) [above] {$A$};

\end{tikzpicture} \quad .
$$

 Standard path-counting arguments show that the `other graph' has the same form up to depth two; call its  vertices at depth two $\hat{Y}$ and $\hat{Z}$.  The depth-three vertices $A$, $B$ and $B$ must each connect to $\hat{Y}$ or $\hat{Z}$.  Path counting implies that each of these three vertices connects to only one of $\hat{Y}$ or $\hat{Z}$; further neither $\hat{Y}$ nor $\hat{Z}$ can be univalent (or else there would be an intermediate subfactor).  So one of the depth-two vertices has degree two, and the other has degree three.  Without loss of generality say $\hat{Y}$ is the degree-three vertex.  Then $A$ must connect to $\hat{Y}$ (because otherwise $\dim{Z} = \frac{d+1}{d}$ is not an allowed dimension).  Thus, $Y$ and $Z$, and $\hat{Y}$ and $\hat{Z}$, are self-dual (since the two vertices in each pair have distinct norms).

 We've just shown that the two possibilities for the 4-partite principal graph are 
$$
\begin{tikzpicture}[inner sep=.7mm, xscale=1.5]
\begin{scope}[xshift=-4mm]
\node at (-1,0) {$1$};
\node at (-1,2) {$\hat{1}$};
\node at (-1,4) {$1$};
\node at (0,1) {$\bar{X}$};
\node at (0,3) {$X$};
\node at (1,0) {$Z$};
\node at (1,2) {$\hat{Z}$};
\node at (1,4) {$Z$};
\node at (2,0) {$Y$};
\node at (2,2) {$\hat{Y}$};
\node at (2,4) {$Y$};
\node at (3,1) {$\bar{C}$};
\node at (3,3) {$C$};
\node at (4,1) {$\bar{B}$};
\node at (4,3) {$B$};
\node at (5,1) {$\bar{A}$};
\node at (5,3) {$A$};
\end{scope}

\node(1) at (-1,0) [circle,fill] {};
\node(h1) at (-1,2)[circle,fill] {} ;
\node(11) at (-1,4)[circle,fill] {};
\node(bX) at (0,1) [circle,fill]{};
\node(X) at (0,3) [circle,fill]{};
\node(Z) at (1,0) [circle,fill]{};
\node(hZ) at (1,2) [circle,fill]{};
\node(ZZ) at (1,4)[circle,fill] {};
\node(Y) at (2,0) [circle,fill]{};
\node(hY) at (2,2) [circle,fill]{};
\node(YY) at (2,4) [circle,fill]{};
\node(bC) at (3,1) [circle, draw, fill=white]{};
\node(C) at (3,3) [circle, draw, fill=white]{};
\node(bB) at (4,1)[circle,draw, fill=white] {};
\node(B) at (4,3)[circle,draw, fill=white] {};
\node(bA) at (5,1)[circle, draw, fill=white] {};
\node(A) at (5,3)[circle, draw, fill=white] {};

\draw (11)--(X)--(ZZ)--(C);
\draw (X)--(YY)--(B);
\draw (YY)--(A);
\draw (h1)--(X)--(hZ)--(C);
\draw (X)--(hY)--(B);
\draw (hY)--(A);
\draw (h1)--(bX)--(hZ)--(bC);
\draw (bX)--(hY)--(bB);
\draw (hY)--(bA);
\draw (1)--(bX)--(Z)--(bC);
\draw (bX)--(Y)--(bB);
\draw (Y)--(bA);
\end{tikzpicture}
$$
 and 
$$
\begin{tikzpicture}[inner sep=.7mm, xscale=1.5]
\begin{scope}[xshift=-4mm]
\node at (-1,0) {$1$};
\node at (-1,2) {$\hat{1}$};
\node at (-1,4) {$1$};
\node at (0,1) {$\bar{X}$};
\node at (0,3) {$X$};
\node at (1,0) {$Z$};
\node at (1,2) {$\hat{Z}$};
\node at (1,4) {$Z$};
\node at (2,0) {$Y$};
\node at (2,2) {$\hat{Y}$};
\node at (2,4) {$Y$};
\node at (3,1) {$\bar{C}$};
\node at (3,3) {$C$};
\node at (4,1) {$\bar{B}$};
\node at (4,3) {$B$};
\node at (5,1) {$\bar{A}$};
\node at (5,3) {$A$};
\end{scope}

\node(1) at (-1,0) [circle,fill] {};
\node(h1) at (-1,2)[circle,fill] {} ;
\node(11) at (-1,4)[circle,fill] {};
\node(bX) at (0,1) [circle,fill]{};
\node(X) at (0,3) [circle,fill]{};
\node(Z) at (1,0) [circle,fill]{};
\node(hZ) at (1,2) [circle,fill]{};
\node(ZZ) at (1,4)[circle,fill] {};
\node(Y) at (2,0) [circle,fill]{};
\node(hY) at (2,2) [circle,fill]{};
\node(YY) at (2,4) [circle,fill]{};
\node(bC) at (3,1) [circle, draw, fill=white]{};
\node(C) at (3,3) [circle, draw, fill=white]{};
\node(bB) at (4,1)[circle,draw, fill=white] {};
\node(B) at (4,3)[circle,draw, fill=white] {};
\node(bA) at (5,1)[circle, draw, fill=white] {};
\node(A) at (5,3)[circle, draw, fill=white] {};

\draw (11)--(X)--(ZZ)--(C);
\draw (X)--(YY)--(B);
\draw (YY)--(A);
\draw (h1)--(X)--(hZ)--(B);
\draw (X)--(hY)--(C);
\draw (hY)--(A);
\draw (h1)--(bX)--(hZ)--(bB);
\draw (bX)--(hY)--(bC);
\draw (hY)--(bA);
\draw (1)--(bX)--(Z)--(bC);
\draw (bX)--(Y)--(bB);
\draw (Y)--(bA);
\end{tikzpicture}
$$
The first of these cases is ruled out by the triple point obstruction (as stated in \cite{index5-part1}, where it is attributed to \cite{MR1317352} and Ocneanu), applied to the vertices $Y$ and $\hat{Y}$.  The second is ruled out by using dimension data to extend the 4-partite graph up to depth 4; it must have the form
$$
\begin{tikzpicture}[inner sep=.7mm, xscale=1.5]
\begin{scope}[xshift=-4mm]
\node at (-1,0) {$1$};
\node at (-1,2) {$\hat{1}$};
\node at (-1,4) {$1$};
\node at (0,1) {$\bar{X}$};
\node at (0,3) {$X$};
\node at (1,0) {$Z$};
\node at (1,2) {$\hat{Z}$};
\node at (1,4) {$Z$};
\node at (2,0) {$Y$};
\node at (2,2) {$\hat{Y}$};
\node at (2,4) {$Y$};
\node at (3,1) {$\bar{C}$};
\node at (3,3) {$C$};
\node at (4,1) {$\bar{B}$};
\node at (4,3) {$B$};
\node at (5,1) {$\bar{A}$};
\node at (5,3) {$A$};
\node at (6,0) {$D$};
\node at (6,2) {$\hat{D}$};
\node at (6,4) {$D$};
\node at (7,0) {$E$};
\node at (7,2) {$\hat{E}$};
\node at (7,4) {$E$};
\end{scope}

\node(1) at (-1,0) [circle,fill] {};
\node(h1) at (-1,2)[circle,fill] {} ;
\node(11) at (-1,4)[circle,fill] {};
\node(bX) at (0,1) [circle,fill]{};
\node(X) at (0,3) [circle,fill]{};
\node(Z) at (1,0) [circle,fill]{};
\node(hZ) at (1,2) [circle,fill]{};
\node(ZZ) at (1,4)[circle,fill] {};
\node(Y) at (2,0) [circle,fill]{};
\node(hY) at (2,2) [circle,fill]{};
\node(YY) at (2,4) [circle,fill]{};
\node(bC) at (3,1) [circle,fill]{};
\node(C) at (3,3) [circle,fill]{};
\node(bB) at (4,1)[circle,fill] {};
\node(B) at (4,3)[circle,fill] {};
\node(bA) at (5,1)[circle,fill] {};
\node(A) at (5,3)[circle,fill] {};
\node(D) at (6,0) [circle, draw, fill=white]{};
\node(hD) at (6,2) [circle,draw, fill=white]{};
\node(DD) at (6,4) [circle,draw, fill=white]{};
\node(E) at (7,0) [circle,draw, fill=white]{};
\node(hE) at (7,2) [circle,draw, fill=white]{};
\node(EE) at (7,4) [circle,draw, fill=white]{};

\draw (11)--(X)--(ZZ)--(C);
\draw (X)--(YY)--(B);
\draw (YY)--(A);
\draw (h1)--(X)--(hZ)--(B);
\draw (X)--(hY)--(C);
\draw (hY)--(A);
\draw (h1)--(bX)--(hZ)--(bB);
\draw (bX)--(hY)--(bC);
\draw (hY)--(bA);
\draw (1)--(bX)--(Z)--(bC);
\draw (bX)--(Y)--(bB);
\draw (Y)--(bA);
\draw (DD)--(C)--(hE)--(bC)--(D);
\draw (EE)--(B)--(hD)--(bB)--(E);
\end{tikzpicture}
$$
but then note that the number of paths from $A$ to $\bar{C}$ is different if you go down, versus up and around.  Thus this cannot be a principal graph.

{\em Case 4:} If $n \geq 4$, the index is greater than $8.29$.
\end{proof}

\begin{prop}\label{1STfinitedepth}
The only finite-depth $1$-supertransitive subfactor with index in $(5, 3+\sqrt{5})$ has principal graph $A=\bigraph{gbg1v1p1v1x1p0x1}$.
\end{prop}
\begin{proof}
We analyze the possibilities in terms of the dimensions of their depth-two objects.

If there were an object of dimension 1 at depth two, then we would be in one of the following two cases.  Either all objects at depth two would have dimension one (in which case, the principal graph is a star and has integer index), or there would be an intermediate subfactor by Lemma \ref{intermediate} (requiring the index to factor into allowed index values, which does not happen in $(5,3+\sqrt{5})$).

 An even object with dimension in $(2,\sqrt{5})$ would have dimension $\frac{\sqrt{3}+\sqrt{7}}{2}$, either by the analysis of possible small dimensions in fusion categories from \cite{1004.0665}, or by the classification of subfactors with index less than $5$, from \cite{index5-part1,index5-part2,index5-part3,index5-part4}.
 If this held for all objects at depth 2, $\dim{X}^2 \geq 1+\sqrt{3} + \sqrt{7} > 3 + \sqrt{5}$.

Thus, some even object $Z$ at depth $2$ has  $1 < \dim{Z} < 2$.  Lemma \ref{lem:depth2dim2} shows that this subfactor has principal graph $A$.
\end{proof}

What happens if we drop the finite depth condition?
The above argument fails --- an object at depth $2$ could have dimension in the interval $(2, \sqrt{5})$, generating an infinite depth $A_\infty$ subcategory.
Happily, we can simply run the odometer, as in \cite{index5-part2}.

\begin{proof}[Proof of Theorem \ref{thm:only1ST}]
We can ignore all the graphs in Figure \ref{fig:ignore}, because they must have an intermediate subfactor, which isn't possible with index in $(5,3+\sqrt{5})$.
We can also ignore the graphs
$$\left(\bigraph{bwd1v1p1p1v1x0x0p0x1x0p0x0x1duals1v1x2x3}, \bigraph{bwd1v1p1p1v1x0x0p0x1x0p0x0x1duals1v1x2x3}\right)\qquad
\left(\bigraph{bwd1v1p1p1v1x0x0p0x1x0p0x0x1duals1v1x3x2}, \bigraph{bwd1v1p1p1v1x0x0p0x1x0p0x0x1duals1v1x3x2}\right)
$$
by Lemma \ref{lem:depth2dim2}; at least one of the three objects at depth 2 has dimension $<2$.

\begin{figure}[ht]
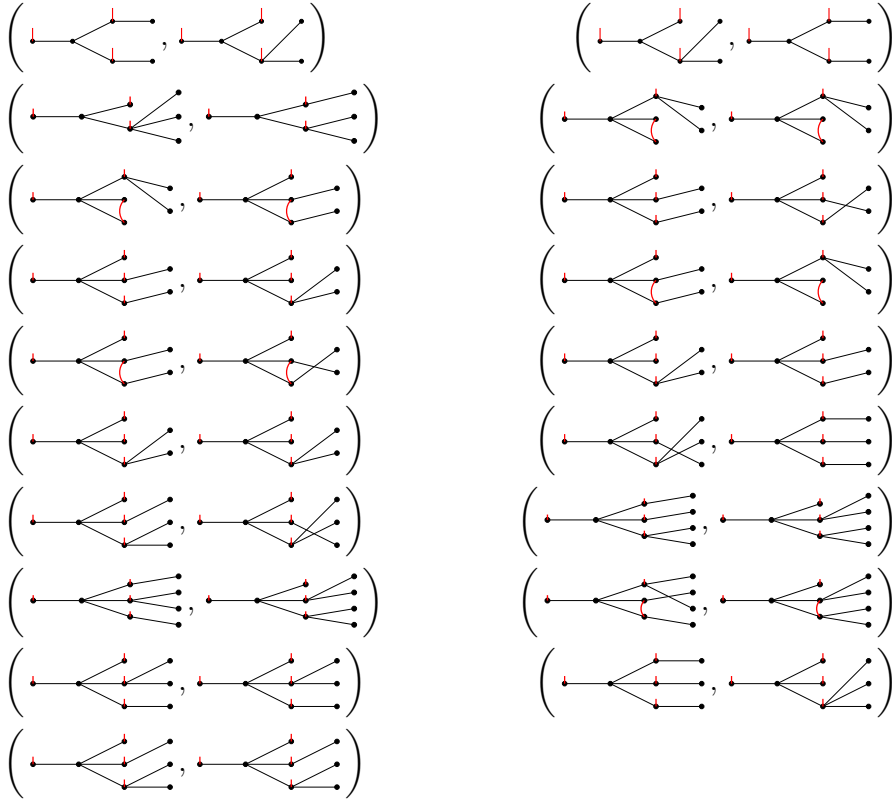


\begin{align*}
&
\left(\bigraph{bwd1v1p1v1x0p0x1duals1v1x2}, \bigraph{bwd1v1p1v1x0p1x0duals1v1x2}\right) & \qquad
\left(\bigraph{bwd1v1p1v1x0p1x0duals1v1x2}, \bigraph{bwd1v1p1v1x0p0x1duals1v1x2}\right)\displaybreak[1] \\ &
\left(\bigraph{bwd1v1p1v1x0p1x0p1x0duals1v1x2}, \bigraph{bwd1v1p1v1x0p1x0p0x1duals1v1x2}\right)&\qquad
\left(\bigraph{bwd1v1p1p1v0x0x1p0x0x1duals1v2x1x3}, \bigraph{bwd1v1p1p1v0x0x1p0x0x1duals1v2x1x3}\right)\displaybreak[1] \\&
\left(\bigraph{bwd1v1p1p1v0x0x1p0x0x1duals1v2x1x3}, \bigraph{bwd1v1p1p1v1x0x0p0x1x0duals1v2x1x3}\right)&\qquad
\left(\bigraph{bwd1v1p1p1v1x0x0p0x1x0duals1v1x2x3}, \bigraph{bwd1v1p1p1v0x1x0p1x0x0duals1v1x2x3}\right)\displaybreak[1] \\&
\left(\bigraph{bwd1v1p1p1v1x0x0p0x1x0duals1v1x2x3}, \bigraph{bwd1v1p1p1v1x0x0p1x0x0duals1v1x2x3}\right)&\qquad
\left(\bigraph{bwd1v1p1p1v1x0x0p0x1x0duals1v2x1x3}, \bigraph{bwd1v1p1p1v0x0x1p0x0x1duals1v2x1x3}\right)\displaybreak[1] \\&
\left(\bigraph{bwd1v1p1p1v1x0x0p0x1x0duals1v2x1x3}, \bigraph{bwd1v1p1p1v0x1x0p1x0x0duals1v2x1x3}\right)&\qquad
\left(\bigraph{bwd1v1p1p1v1x0x0p1x0x0duals1v1x2x3}, \bigraph{bwd1v1p1p1v1x0x0p0x1x0duals1v1x2x3}\right)\displaybreak[1] \\&
\left(\bigraph{bwd1v1p1p1v1x0x0p1x0x0duals1v1x2x3}, \bigraph{bwd1v1p1p1v1x0x0p1x0x0duals1v1x2x3}\right)&\qquad
\left(\bigraph{bwd1v1p1p1v0x1x0p1x0x0p1x0x0duals1v1x2x3}, \bigraph{bwd1v1p1p1v1x0x0p0x1x0p0x0x1duals1v1x2x3}\right)\displaybreak[1] \\&
\left(\bigraph{bwd1v1p1p1v1x0x0p1x0x0p0x1x0duals1v1x2x3}, \bigraph{bwd1v1p1p1v0x1x0p1x0x0p1x0x0duals1v1x2x3}\right)&\qquad
\left(\bigraph{bwd1v1p1p1v1x0x0p1x0x0p0x1x0p0x0x1duals1v1x2x3}, \bigraph{bwd1v1p1p1v1x0x0p1x0x0p0x1x0p0x1x0duals1v1x2x3}\right)\displaybreak[1] \\&
\left(\bigraph{bwd1v1p1p1v1x0x0p0x1x0p0x1x0p0x0x1duals1v1x2x3}, \bigraph{bwd1v1p1p1v1x0x0p1x0x0p0x1x0p0x1x0duals1v1x2x3}\right)&\qquad
\left(\bigraph{bwd1v1p1p1v1x0x0p0x0x1p0x1x0p0x0x1duals1v2x1x3}, \bigraph{bwd1v1p1p1v1x0x0p1x0x0p0x1x0p0x1x0duals1v2x1x3}\right)\displaybreak[1] \\&
\left(\bigraph{bwd1v1p1p1v1x0x0p0x1x0p0x1x0duals1v1x2x3}, \bigraph{bwd1v1p1p1v1x0x0p0x1x0p0x1x0duals1v1x2x3}\right)&\qquad
\left(\bigraph{bwd1v1p1p1v1x0x0p0x1x0p0x0x1duals1v1x2x3}, \bigraph{bwd1v1p1p1v1x0x0p1x0x0p1x0x0duals1v1x2x3}\right)\displaybreak[1] \\&
\left(\bigraph{bwd1v1p1p1v1x0x0p1x0x0p0x1x0duals1v1x2x3}, \bigraph{bwd1v1p1p1v1x0x0p1x0x0p0x1x0duals1v1x2x3}\right)
\end{align*}
\caption{Principals graphs for which there must be an intermediate subfactor.}
\label{fig:ignore}
\end{figure}

The result of running the odometer and ignoring the weeds mentioned above is the tree shown in Figure \ref{fig:odometer} (in which the leaves, i.e. the red and blue graphs, are the remaining weeds). 
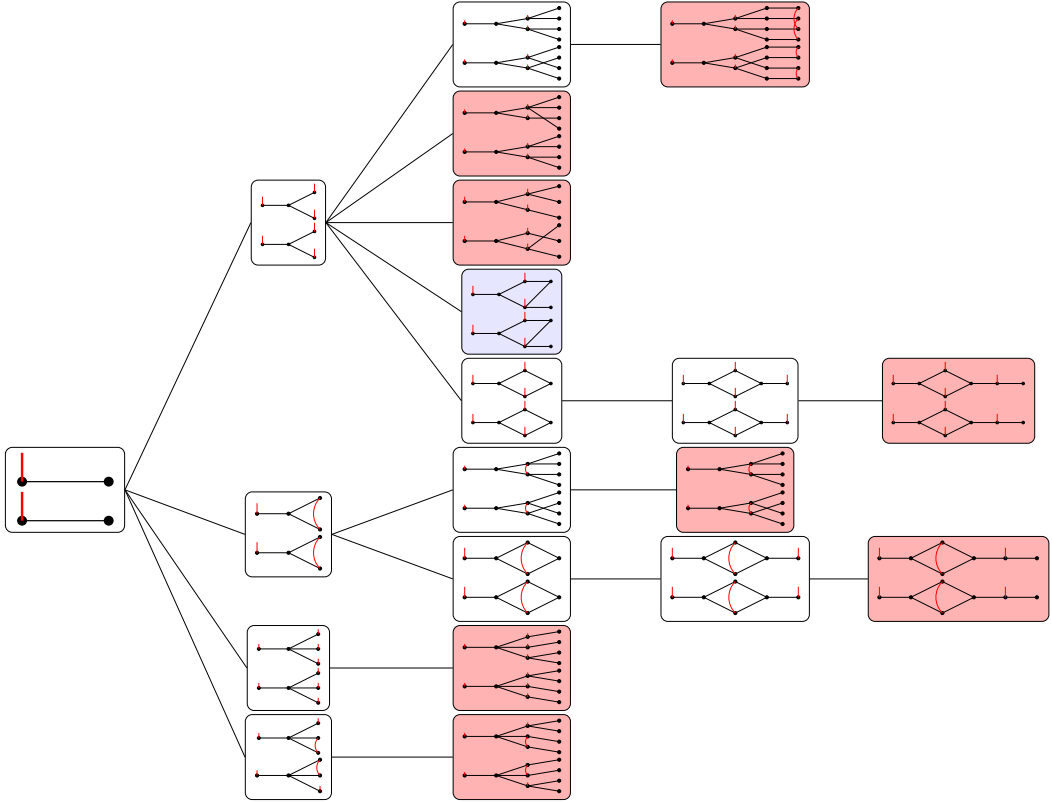
\begin{figure}[th]
$$
\scalebox{0.65}{
 \begin{tikzpicture}
\tikzset{grow=right,level distance=130pt}
\tikzset{every tree node/.style={draw,fill=white,rectangle,rounded corners,inner sep=2pt}}
\Tree
[.\node{$\!\!\begin{array}{c}\bigraph{bwd1duals1}\\\bigraph{bwd1duals1}\end{array}\!\!$};
	[.\node{$\!\!\begin{array}{c}\bigraph{bwd1v1p1p1duals1v2x1x3}\\\bigraph{bwd1v1p1p1duals1v1x3x2}\end{array}\!\!$};
		[.\node[fill=red!30]{$\!\!\begin{array}{c}\bigraph{bwd1v1p1p1v1x0x0p0x1x0p0x0x1p0x0x1duals1v2x1x3}\\\bigraph{bwd1v1p1p1v1x0x0p1x0x0p0x1x0p0x0x1duals1v1x3x2}\end{array}\!\!$};]	
	]
	[.\node{$\!\!\begin{array}{c}\bigraph{bwd1v1p1p1duals1v1x2x3}\\\bigraph{bwd1v1p1p1duals1v1x2x3}\end{array}\!\!$};
		[.\node[fill=red!30]{$\!\!\begin{array}{c}\bigraph{bwd1v1p1p1v1x0x0p1x0x0p0x1x0p0x0x1duals1v1x2x3}\\\bigraph{bwd1v1p1p1v1x0x0p0x1x0p0x0x1p0x0x1duals1v1x2x3}\end{array}\!\!$};]	
	]
	[.\node{$\!\!\begin{array}{c}\bigraph{bwd1v1p1duals1v2x1}\\\bigraph{bwd1v1p1duals1v2x1}\end{array}\!\!$};
		[.\node{$\!\!\begin{array}{c}\bigraph{bwd1v1p1v1x1duals1v2x1}\\\bigraph{bwd1v1p1v1x1duals1v2x1}\end{array}\!\!$};
			[.\node{$\!\!\begin{array}{c}\bigraph{bwd1v1p1v1x1v1duals1v2x1v1}\\\bigraph{bwd1v1p1v1x1v1duals1v2x1v1}\end{array}\!\!$};
				[.\node[fill=red!30]{$\!\!\begin{array}{c}\bigraph{bwd1v1p1v1x1v1v1duals1v2x1v1}\\\bigraph{bwd1v1p1v1x1v1v1duals1v2x1v1}\end{array}\!\!$};]
			]
		]
		[.\node{$\!\!\begin{array}{c}\bigraph{bwd1v1p1v1x0p1x0p0x1p0x1duals1v2x1}\\\bigraph{bwd1v1p1v1x0p0x1p1x0p0x1duals1v2x1}\end{array}\!\!$};
			[.\node[fill=red!30]{$\!\!\begin{array}{c}\bigraph{bwd1v1p1v1x0p1x0p0x1p0x1duals1v2x1}\\\bigraph{bwd1v1p1v1x0p0x1p1x0p0x1duals1v2x1}\end{array}\!\!$};]
		]
	]
	[.\node{$\!\!\begin{array}{c}\bigraph{bwd1v1p1duals1v1x2}\\\bigraph{bwd1v1p1duals1v1x2}\end{array}\!\!$};
		[.\node{$\!\!\begin{array}{c}\bigraph{bwd1v1p1v1x1duals1v1x2}\\\bigraph{bwd1v1p1v1x1duals1v1x2}\end{array}\!\!$};
			[.\node{$\!\!\begin{array}{c}\bigraph{bwd1v1p1v1x1v1duals1v1x2v1}\\\bigraph{bwd1v1p1v1x1v1duals1v1x2v1}\end{array}\!\!$};
				[.\node[fill=red!30]{$\!\!\begin{array}{c}\bigraph{bwd1v1p1v1x1v1v1duals1v1x2v1}\\\bigraph{bwd1v1p1v1x1v1v1duals1v1x2v1}\end{array}\!\!$};]
			]	
		]	
		[.\node[fill=blue!10]{$\!\!\begin{array}{c}\bigraph{bwd1v1p1v1x0p1x1duals1v1x2}\\\bigraph{bwd1v1p1v1x0p1x1duals1v1x2}\end{array}\!\!$};]	
		[.\node[fill=red!30]{$\!\!\begin{array}{c}\bigraph{bwd1v1p1v1x0p0x1p0x1duals1v1x2}\\\bigraph{bwd1v1p1v1x0p0x1p1x0duals1v1x2}\end{array}\!\!$};]	
		[.\node[fill=red!30]{$\!\!\begin{array}{c}\bigraph{bwd1v1p1v0x1p1x0p0x1p0x1duals1v1x2}\\\bigraph{bwd1v1p1v1x0p1x0p0x1p0x1duals1v1x2}\end{array}\!\!$};]	
		[.\node{$\!\!\begin{array}{c}\bigraph{bwd1v1p1v1x0p1x0p0x1p0x1duals1v1x2}\\\bigraph{bwd1v1p1v1x0p0x1p1x0p0x1duals1v1x2}\end{array}\!\!$};
			[.\node[fill=red!30]{$\!\!\begin{array}{c}\bigraph{bwd1v1p1v1x0p1x0p0x1p0x1v1x0x0x0p0x1x0x0p0x0x1x0p0x0x0x1duals1v1x2v3x4x1x2}\\\bigraph{bwd1v1p1v1x0p0x1p1x0p0x1v1x0x0x0p0x1x0x0p0x0x1x0p0x0x0x1duals1v1x2v2x1x4x3}\end{array}\!\!$};]	
		]	
	]
]
\end{tikzpicture}
}
$$
\caption{Running the odometer.}
\label{fig:odometer}
\end{figure}

The graph 
$$(\tikz[baseline=-3pt]{\node[inner sep=0pt] {$\bigraph{bwd1v1p1v1x0p1x1duals1v1x2}$};}, \tikz[baseline=-3pt]{\node[inner sep=0pt] {$\bigraph{bwd1v1p1v1x0p1x1duals1v1x2}$};} )$$ 
has a depth 2 object with dimension less than $2$, and again Lemma \ref{lem:depth2dim2} shows that the only subfactors coming from this weed have principal graph $\Gamma(\cA)$.

 For each red pair, some object has an impossible dimension.  The hardest case is for the graphs of the form $$\bigraph{gbg1v1p1v1x1v1v1}.$$
The last two objects have dimensions 
\begin{align*}
d_1 & = \frac{1}{2}(q^4-q^2-2-q^{-2}+q^{-4}), \\
d_2 & = \frac{1}{2}(q^5-q^3-3q-3q^{-1}-q^{-3}+q^{-5}).
\end{align*}
Now $d_2 < 1.145$, and $d_2 = 1$ only if $q=1.69068...^{\pm 1}$. But then $d_1 = 1.54231...$, which is not of the form $2 \cos(\pi/n)$.

Putting this together, we have the proof of Theorem \ref{thm:only1ST}.
\end{proof}

\newpage

\section{Connections}\label{connections}

\subsection{Bi-unitary connections on $\Gamma(\cA)$}
\label{biunitaryhat}
The four-partite principal graph for $\cA$ is 
$$
\begin{tikzpicture}[inner sep=.7mm, xscale=1.5,yscale=1.2]
\begin{scope}[xshift=-4mm]
\node at (-1,0) {$1$};
\node at (-1,2) {$\hat{1}$};
\node at (-1,4) {$1$};
\node at (0,1) {$\bar{X}$};
\node at (0,3) {$X$};
\node at (1,0) {$Z$};
\node at (1,2) {$\hat{Z}$};
\node at (1,4) {$Z$};
\node at (2,0) {$Y$};
\node at (2,2) {$\hat{Y}$};
\node at (2,4) {$Y$};
\node at (3,1) {$\bar{W}$};
\node at (3,3) {$W$};
\node at (4,1) {$\bar{g}$};
\node at (4,3) {$g$};
\end{scope}

\node(1) at (-1,0) [circle,fill] {};
\node(h1) at (-1,2)[circle,fill] {} ;
\node(11) at (-1,4)[circle,fill] {};
\node(bX) at (0,1) [circle,fill]{};
\node(X) at (0,3) [circle,fill]{};
\node(Z) at (1,0) [circle,fill]{};
\node(hZ) at (1,2) [circle,fill]{};
\node(ZZ) at (1,4)[circle,fill] {};
\node(Y) at (2,0) [circle,fill]{};
\node(hY) at (2,2) [circle,fill]{};
\node(YY) at (2,4) [circle,fill]{};
\node(bW) at (3,1) [circle,fill]{};
\node(W) at (3,3) [circle,fill]{};
\node(bg) at (4,1)[circle,fill] {};
\node(g) at (4,3)[circle,fill] {};

\draw (11)--(X)--(YY)--(W)--(ZZ)--(X);
\draw (YY)--(g);
\draw (h1)--(X)--(hY)--(W)--(hZ)--(X);
\draw (hY)--(g);
\draw (h1)--(bX)--(hY)--(bW)--(hZ)--(bX);
\draw (hY)--(bg);
\draw (1)--(bX)--(Y)--(bW)--(Z)--(bX);
\draw (Y)--(bg);
\end{tikzpicture}
$$

We determine which biunitary connections are flat, using a condition from \cite[Theorem 1.7]{index5-part3} which follows from flatness.

Recall the notion of `diagrammatic gauge' from \cite[Lemma 5.6]{index5-part3}. It is a subset of the full gauge group orbit of any bi-unitary connection, and characterized by having all connection entries in the 1-by-1 and 2-by-2 matrices corresponding to the edges of the principal graph before the first branch point real, and the connection matrix at the branch point having real first row and first column, with the top-left entry having sign $(-1)^{n+1}$ and the other entries in the first row and column being positive. It is called the diagrammatic gauge because it corresponds to using Temperley-Lieb diagrams as the basis for the appropriate 1-dimensional trivalent vertex spaces.

In particular, if a connection $K$ is in the diagrammatic gauge, and $U$ is the connection matrix at the first branch point of the principal graph, then the eigenvalues of $U U^t$ are $+1$ with multiplicity $2$, along with rotational eigenvalues of the planar algebra generators corresponding to the branches. For $\cA$, there is just one such generator (since the branch point is a triple point), and its rotational eigenvalue is $+1$, since the vertices past the branch point are self-dual. Thus all the eigenvalues of $U U^t$ are $+1$, so this matrix is the identity. We use this condition to drastically simplify the calculation of possible flat bi-unitary connections.

We in fact find there are two bi-unitary connections satisfying this eigenvalue condition, and subsequently, in \S \ref{flatlowweight} discover that only one of them is actually flat, by explicitly looking for flat elements in the graph planar algebra. 

\input{cells}

\newcommand{\houses}{
\tileA{r1X}{c1X}{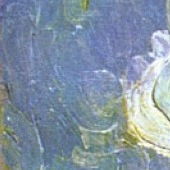}
\tileA{r1X}{cZX}{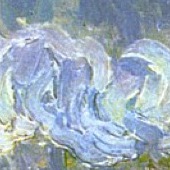}
\tileA{r1X}{cYX}{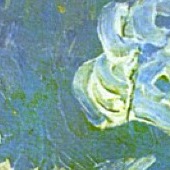}
\tileA{rZX}{c1X}{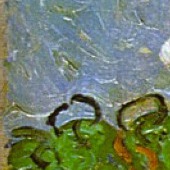}
\tileA{rZX}{cZX}{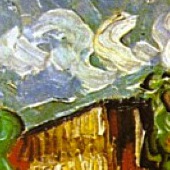}
\tileA{rZX}{cZW}{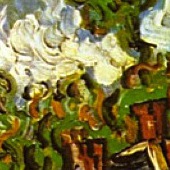}
\tileA{rZX}{cYX}{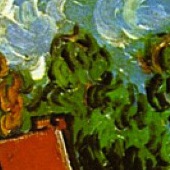}
\tileA{rZX}{cYW}{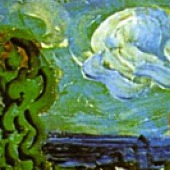}
\tileA{rZW}{cZX}{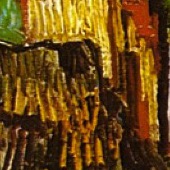}
\tileA{rZW}{cZW}{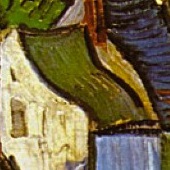}
\tileA{rZW}{cYX}{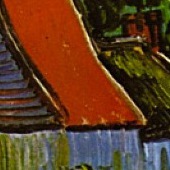}
\tileA{rZW}{cYW}{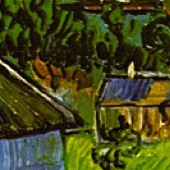}
\tileA{rYX}{c1X}{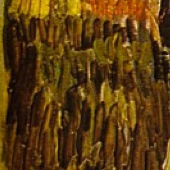}
\tileA{rYX}{cZX}{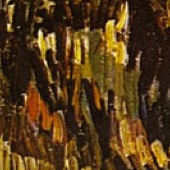}
\tileA{rYX}{cZW}{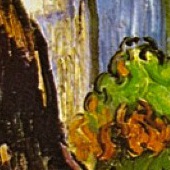}
\tileA{rYX}{cYX}{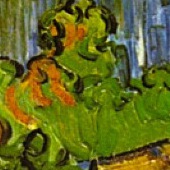}
\tileA{rYX}{cYW}{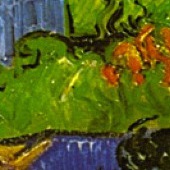}
\tileA{rYX}{cYg}{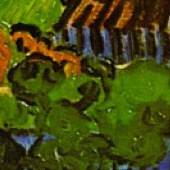}
\tileA{rYW}{cZX}{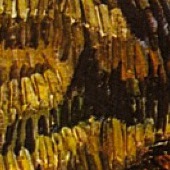}
\tileA{rYW}{cZW}{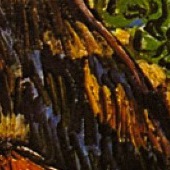}
\tileA{rYW}{cYX}{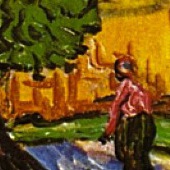}
\tileA{rYW}{cYW}{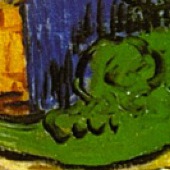}
\tileA{rYW}{cYg}{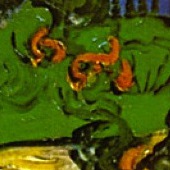}
\tileA{rYg}{cYX}{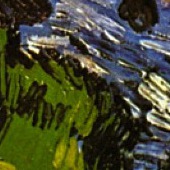}
\tileA{rYg}{cYW}{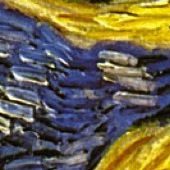}
\tileA{rYg}{cYg}{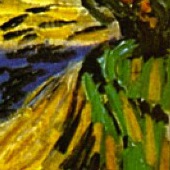}
}

\newcommand{\shuffledhouses}{
\tileA{rX1}{cX1}{00}
\tileA{rX1}{cXZ}{01}
\tileA{rX1}{cXY}{03}
\tileA{rXZ}{cX1}{06}
\tileA{rXZ}{cXZ}{07}
\tileA{rXZ}{cXY}{09}
\tileA{rXZ}{cWZ}{13}
\tileA{rXZ}{cWY}{15}
\tileA{rXY}{cX1}{18}
\tileA{rXY}{cXZ}{19}
\tileA{rXY}{cXY}{21}
\tileA{rXY}{cWZ}{25}
\tileA{rXY}{cWY}{27}
\tileA{rXY}{cgY}{33}
\tileA{rWZ}{cXZ}{08}
\tileA{rWZ}{cXY}{10}
\tileA{rWZ}{cWZ}{14}
\tileA{rWZ}{cWY}{16}
\tileA{rWY}{cXZ}{20}
\tileA{rWY}{cXY}{22}
\tileA{rWY}{cWZ}{26}
\tileA{rWY}{cWY}{28}
\tileA{rWY}{cgY}{34}
\tileA{rgY}{cXY}{23}
\tileA{rgY}{cWY}{29}
\tileA{rgY}{cgY}{35}
}

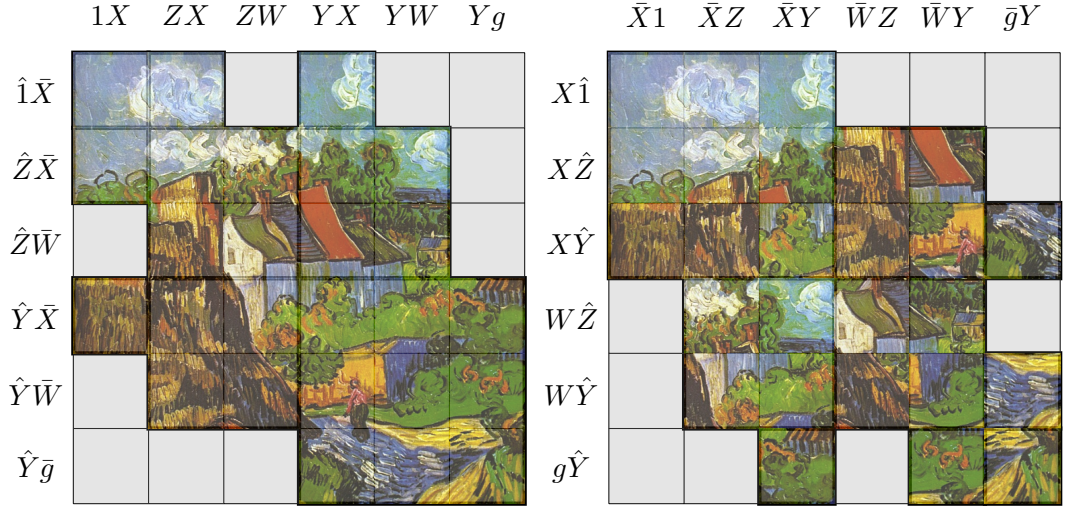
\begin{figure}[t]
\begin{equation*}
\begin{tikzpicture}
\principalmatrixA
\houses
\end{tikzpicture}
\begin{tikzpicture}
\dualmatrixA
\shuffledhouses
\end{tikzpicture}
\end{equation*}
\caption{Houses in Auvers, 2, by Vincent van Gogh, illustrating the bijection between nonzero connection entries on the principal and dual principal graphs.}
\label{fig:Auvers}
\end{figure}

The connection entries $1X\hat{1}\bar{X}, ZX\hat{1}\bar{X}, YX\hat{1}\bar{X}, 1X\hat{Z}\bar{X}$ and  $1X\hat{Y}\bar{X}$ each lie in a 1-by-1 matrix, so have norm 1. The diagrammatic gauge choice means these are all exactly $+1$. We then use the renormalization axiom to transfer these across to the dual graph. Below $d=\dim{X}=\dim{Y} \approx 2.24698$ and $e=\dim{Z}=\dim{W}= d^2 - d - 1 \approx 1.80194$.

\newcommand{\stepAp}{
\cellA{r1X}{c1X}{}{}{}{1}
\cellA{r1X}{cZX}{}{}{}{1}
\cellA{r1X}{cYX}{}{}{}{1}
\cellA{rZX}{c1X}{}{}{}{1}
\cellA{rYX}{c1X}{}{}{}{1}
}

\newcommand{\stepAd}{
\cellA{rX1}{cX1}{}{}{}{$\frac{1}{d}$}
\cellA{rX1}{cXZ}{}{}{}{$\frac{\sqrt{e}}{d}$}
\cellA{rX1}{cXY}{}{}{}{$\frac{\sqrt{d}}{d}$}
\cellA{rXZ}{cX1}{}{}{}{$\frac{\sqrt{e}}{d}$}
\cellA{rXY}{cX1}{}{}{}{$\frac{\sqrt{d}}{d}$}
}

\begin{equation*}
\begin{tikzpicture}[yscale=.8]
\principalmatrixA
\stepAp
\end{tikzpicture}
\begin{tikzpicture}[yscale=.8]
\dualmatrixA
\stepAd
\end{tikzpicture}
\end{equation*}

We quickly see that the 3-by-3 $X \bar{X}$ matrix in the dual connection must be symmetric. Call this matrix $U$ and recall $U U^t = \bf{1}$. Now $U = U^t = U^{-1} = U^*$ tells us that $U$ is real, and this allows us to completely solve for $U$. We obtain two solutions, 
$$U = \begin{pmatrix}
\frac{1}{d} & \frac{\sqrt{e}}{d} & \frac{\sqrt{d}}{d} \\
\frac{\sqrt{e}}{d} & \frac{e}{d} r^{(i)}_1 & \sqrt{\frac{e}{d}} r^{(i)}_2 \\
\frac{\sqrt{d}}{d} & \sqrt{\frac{e}{d}} r^{(i)}_2 & r^{(i)}_3
\end{pmatrix}
$$
where
\begin{align*}
r^{(1)}_1 & = d^2 -4d+3 \\
r^{(1)}_2 & = d^2 -3d +2\\
r^{(1)}_3 & = d^2 -3d +1\\
\intertext{and}
r^{(2)}_1 & = -d^2+2d+1\\
r^{(2)}_2 & = -d^2+d+2\\
r^{(2)}_3 & = -d^2+d+3
\end{align*}

We next transfer these entries back to the principal graph.

\newcommand{\stepBp}{
\stepAp
\cellA{rZX}{cZX}{}{}{}{$r_1$}
\cellA{rYX}{cZX}{}{}{}{$r_2$}
\cellA{rZX}{cYX}{}{}{}{$r_2$}
\cellA{rYX}{cYX}{}{}{}{$r_3$}
}

\newcommand{\stepBd}{
\stepAd
\cellA{rXZ}{cXZ}{}{}{}{$\frac{e}{d} r_1$}
\cellA{rXZ}{cXY}{}{}{}{$\sqrt{\frac{e}{d}} r_2$}
\cellA{rXY}{cXZ}{}{}{}{$\sqrt{\frac{e}{d}} r_2$}
\cellA{rXY}{cXY}{}{}{}{$r_3$}
}

\begin{equation*}
\begin{tikzpicture}[yscale=.8]
\principalmatrixA
\stepBp
\end{tikzpicture}
\begin{tikzpicture}[yscale=.8]
\dualmatrixA
\stepBd
\end{tikzpicture}
\end{equation*}

With $s_i = \sqrt{1-r_i^2}$, we next introduce six new variables $\xi_1, \xi_2, \xi_2', \alpha_1, \alpha_2$ and $\alpha_2' \in \bbT$ on the unit circle. Using these we complete the 2-by-2 matrices on the principal graph, and transfer all the new entries back over the to dual graph.

\newcommand{\stepCp}{
\stepBp
\cellA{rZW}{cZW}{}{}{}{$r_1 \xi_1$}
\cellA{rZW}{cZX}{}{}{}{$s_1 \alpha_1$}
\cellA{rZX}{cZW}{}{}{}{$s_1 \frac{- \xi_1}{ \alpha_1}$}
\cellA{rYW}{cZW}{}{}{}{$r_2 \xi_2$}
\cellA{rYW}{cZX}{}{}{}{$s_2 \alpha_2$}
\cellA{rYX}{cZW}{}{}{}{$s_2 \frac{- \xi_2}{ \alpha_2}$}
\cellA{rZW}{cYW}{}{}{}{$r_2 \xi_2'$}
\cellA{rZW}{cYX}{}{}{}{$s_2 \alpha_2'$}
\cellA{rZX}{cYW}{}{}{}{$s_2 \frac{- \xi_2'}{ \alpha_2'}$}
}

\newcommand{\stepCd}{
\stepBd
\cellA{rWZ}{cXZ}{}{}{}{$\sqrt{\frac{e}{d}} s_1\frac{- \xi_1}{ \alpha_1}$}
\cellA{rWZ}{cXY}{}{}{}{$s_2\frac{- \xi_2'}{ \alpha_2' }$}
\cellA{rWY}{cXZ}{}{}{}{$s_2\frac{- \xi_2}{ \alpha_2}$}
\cellA{rXZ}{cWZ}{}{}{}{$\sqrt{\frac{e}{d}} s_1 \alpha_1$}
\cellA{rXZ}{cWY}{}{}{}{$s_2 \alpha_2'$}
\cellA{rXY}{cWZ}{}{}{}{$s_2 \alpha_2$}
\cellA{rWZ}{cWZ}{}{}{}{$r_1 \xi_1$}
\cellA{rWZ}{cWY}{}{}{}{$\sqrt{\frac{d}{e}} r_2 \xi_2'$}
\cellA{rWY}{cWZ}{}{}{}{$\sqrt{\frac{d}{e}} r_2 \xi_2$}
}

\begin{equation*}
\begin{tikzpicture}[yscale=.8]
\principalmatrixA
\stepCp
\end{tikzpicture}
\begin{tikzpicture}[yscale=.8]
\dualmatrixA
\stepCd
\end{tikzpicture}
\end{equation*}

We can now fill in the bottom right entries of these 2-by-2 matrices; we also introduce variables $\beta_1, \ldots \beta_5 \in \bbT$ on the unit circle for the remaining 1-by-1 matrices in the dual connection, and transfer everything back to the principal connection matrix:

\newcommand{\stepDp}{
\stepCp
\cellA{rYX}{cYW}{}{}{}{${\frac{e}{d}} s_1\frac{\alpha_1 \xi_2 \xi_2'}{\xi_1  \alpha_2 \alpha_2'} $}
\cellA{rYX}{cYg}{}{}{}{$\sqrt{\frac{1}{d}} \beta_3$}
\cellA{rYW}{cYX}{}{}{}{${\frac{e}{d}} s_1 \frac{-\alpha_2 \alpha_2'}{ \alpha_1}$}
\cellA{rYW}{cYW}{}{}{}{${\frac{e}{d}} r_1 \frac{- \xi_2 \xi_2'}{\xi_1}$}
\cellA{rYW}{cYg}{}{}{}{$\sqrt{\frac{e}{d^2}} \beta_4$}
\cellA{rYg}{cYX}{}{}{}{$\sqrt{\frac{1}{d}} \beta_1$}
\cellA{rYg}{cYW}{}{}{}{$\sqrt{\frac{e}{d^2}} \beta_2$}
\cellA{rYg}{cYg}{}{}{}{${\frac{1}{d}} \beta_5$}
}

\newcommand{\stepDd}{
\stepCd
\cellA{rWY}{cXY}{}{}{}{$\sqrt{\frac{e}{d}} s_1\frac{\alpha_1 \xi_2 \xi_2'}{\xi_1  \alpha_2 \alpha_2'}$}
\cellA{rXY}{cWY}{}{}{}{$\sqrt{\frac{e}{d}} s_1 \frac{-\alpha_2 \alpha_2'}{ \alpha_1}$}
\cellA{rWY}{cWY}{}{}{}{$r_1 \frac{- \xi_2 \xi_2'}{\xi_1}$}
\cellA{rXY}{cgY}{}{}{}{$\beta_1$}
\cellA{rWY}{cgY}{}{}{}{$\beta_2$}
\cellA{rgY}{cXY}{}{}{}{$\beta_3$}
\cellA{rgY}{cWY}{}{}{}{$\beta_4$}
\cellA{rgY}{cgY}{}{}{}{$\beta_5$}
}

\begin{align*}
\begin{tikzpicture}[yscale=.8]
\principalmatrixA
\stepDp
\end{tikzpicture}
\begin{tikzpicture}[yscale=.8]
\dualmatrixA
\stepDd
\end{tikzpicture}
\end{align*}

We check our work up to this point by verifying that the rows of the 3-by-3 matrix in the principal connection have unit norm.  Next, we eliminate some phases using orthogonality of the rows and of the columns of this matrix. All dot products of rows with each other, and columns with each other, have the form
$$m_1 \phi_1 + m_2 \phi_2 + m_3 \phi_3=0, \qquad m_i \in \mathbb{R}, \qquad \phi_i \in \mathbb{T}$$ with $m_1 + m_2+m_3=0$.
This implies (via the triangle inequality) that $\phi_1 = \phi_2 = \phi_3$, from which we deduce
\begin{align*} 
\beta_2 &= \sigma^{(i)} \frac{\beta_1 \alpha_1 \xi_2 \xi_2'}{\xi_1 \alpha_2 \alpha_2'} \\
\beta_4 &= - \sigma^{(i)} \frac{\beta_3 \alpha_2 \alpha_2'}{ \alpha_1} \\
\beta_5 &= \beta_1 \beta_3
\end{align*}
where $\sigma^{(1)}=+1$ and $\sigma^{(2)}=-1$.

The connection matrices are thus:

\newcommand{\stepEp}{
\stepCp
\cellA{rYX}{cYW}{}{}{}{${\frac{e}{d}} s_1\frac{\alpha_1 \xi_2 \xi_2'}{\xi_1  \alpha_2 \alpha_2'} $}
\cellA{rYX}{cYg}{}{}{}{$\sqrt{\frac{1}{d}} \beta_3$}
\cellA{rYW}{cYX}{}{}{}{${\frac{e}{d}} s_1 \frac{-\alpha_2 \alpha_2'}{ \alpha_1}$}
\cellA{rYW}{cYW}{}{}{}{${\frac{e}{d}} r_1 \frac{- \xi_2 \xi_2'}{\xi_1}$}
\cellA{rYW}{cYg}{}{}{}{$\sqrt{\frac{e}{d^2}} \sigma \frac{-\beta_3 \alpha_2 \alpha_2'}{\alpha_1}$}
\cellA{rYg}{cYX}{}{}{}{$\sqrt{\frac{1}{d}} \beta_1$}
\cellA{rYg}{cYW}{}{}{}{$\sqrt{\frac{e}{d^2}} \sigma \frac{\beta_1 \alpha_1 \xi_2 \xi_2'}{\xi_1 \alpha_2 \alpha_2'}$}
\cellA{rYg}{cYg}{}{}{}{${\frac{1}{d}} \beta_1 \beta_3$}
}

\newcommand{\stepEd}{
\stepCd
\cellA{rWY}{cXY}{}{}{}{$\sqrt{\frac{e}{d}} s_1\frac{\alpha_1 \xi_2 \xi_2'}{\xi_1  \alpha_2 \alpha_2'}$}
\cellA{rXY}{cWY}{}{}{}{$\sqrt{\frac{e}{d}} s_1 \frac{-\alpha_2 \alpha_2'}{ \alpha_1}$}
\cellA{rWY}{cWY}{}{}{}{$r_1 \frac{- \xi_2 \xi_2'}{\xi_1}$}
\cellA{rXY}{cgY}{}{}{}{$\beta_1$}
\cellA{rWY}{cgY}{}{}{}{$\sigma \frac{\beta_1 \alpha_1 \xi_2 \xi_2'}{\xi_1 \alpha_2 \alpha_2'}$}
\cellA{rgY}{cXY}{}{}{}{$\beta_3$}
\cellA{rgY}{cWY}{}{}{}{$- \sigma \frac{\beta_3 \alpha_2 \alpha_2'}{\alpha_1}$}
\cellA{rgY}{cgY}{}{}{}{$\beta_1 \beta_3$}
}

\begin{align*}
\begin{tikzpicture}[yscale=.8]
\principalmatrixA
\stepEp
\end{tikzpicture}
\begin{tikzpicture}[yscale=.8]
\dualmatrixA
\stepEd
\end{tikzpicture}
\end{align*}

One can readily verify that each block is a unitary matrix.
Now we act by the remaining gauge subgroup, namely the subgroup corresponding to edges $ZW, YW$ and $Yg$ (columns in the principal matrix),  $\hat{Z}\bar{W}, \hat{Y}\bar{W}$ and $\hat{Y}\bar{g}$ (rows in the principal matrix), $\bar{W}{Z}, \bar{W}{Y}$ and $\bar{g} Y$ (columns in the dual matrix) and $W\hat{Z}$, $W\hat{Y}$ and $g\hat{Y}$ (rows in the dual matrix). In particular, we take the gauge element
\begin{align*}
\mu \left(W,\hat{Z}\right) & = -\frac{\alpha _2'}{\xi _2'} & 
\mu \left(g,\hat{Y}\right) & = \frac{1}{\beta _3} & 
\mu \left(\bar{W},Z\right) & = \frac{1}{\alpha _2} \\ 
\mu \left(\bar{g},Y\right) & = \frac{1}{\beta _1} & 
\mu (Z,W) & = \frac{\alpha _1 \xi _2'}{\xi _1 \alpha _2'} & 
\mu \left(\hat{Z},\bar{W}\right) & = \frac{\alpha _2}{\alpha _1} \\ 
\mu \left(\bar{W},Y\right) & = \frac{\alpha _1}{\alpha _2 \alpha _2'} & 
\mu \left(W,\hat{Y}\right) & = -\frac{\alpha _2 \xi _1 \alpha _2'}{\alpha _1 \xi _2 \xi _2'} & 
\mu (Y,g) & = 1 \\ 
\mu (Y,W) & = 1 & 
\mu \left(\hat{Y},\bar{g}\right) & = 1 & 
\mu \left(\hat{Y},\bar{W}\right) & = 1  
\end{align*}
and see that any bi-unitary connection is gauge equivalent to $K^{(1)}$ or $K^{(2)}$, obtained from the matrices below by substituting in the values $r^{(i)}_j$ and $s^{(i)}_j$ given above for $r_j$ and $s_j$. (Equivalently, one can think of $K^{(i)}$ as obtained from the matrices immediately above by setting each $\alpha_i, \alpha'_i, \beta_i$ to $1$, and each $\xi_i, \xi'_i$ to $-1$.)

\newcommand{\Gaugedp}{
\cellA{r1X}{c1X}{}{}{}{$1$}
\cellA{r1X}{cZX}{}{}{}{$1$}
\cellA{r1X}{cYX}{}{}{}{$1$}
\cellA{rZX}{c1X}{}{}{}{$1$}
\cellA{rZX}{cZX}{}{}{}{$r_1$}
\cellA{rZX}{cZW}{}{}{}{$s_1$}
\cellA{rZX}{cYX}{}{}{}{$r_2$}
\cellA{rZX}{cYW}{}{}{}{$s_2$}
\cellA{rZW}{cZX}{}{}{}{$s_1$}
\cellA{rZW}{cZW}{}{}{}{$-r_1$}
\cellA{rZW}{cYX}{}{}{}{$s_2$}
\cellA{rZW}{cYW}{}{}{}{$-r_2$}
\cellA{rYX}{c1X}{}{}{}{$1$}
\cellA{rYX}{cZX}{}{}{}{$r_2$}
\cellA{rYX}{cZW}{}{}{}{$s_2$}
\cellA{rYX}{cYX}{}{}{}{$r_3$}
\cellA{rYX}{cYW}{}{}{}{$-\frac{e s_1}{d}$}
\cellA{rYX}{cYg}{}{}{}{$\frac{1}{\sqrt{d}}$}
\cellA{rYW}{cZX}{}{}{}{$s_2$}
\cellA{rYW}{cZW}{}{}{}{$-r_2$}
\cellA{rYW}{cYX}{}{}{}{$-\frac{e s_1}{d}$}
\cellA{rYW}{cYW}{}{}{}{$\frac{e r_1}{d}$}
\cellA{rYW}{cYg}{}{}{}{$-\frac{\sqrt{e} \sigma }{d}$}
\cellA{rYg}{cYX}{}{}{}{$\frac{1}{\sqrt{d}}$}
\cellA{rYg}{cYW}{}{}{}{$-\frac{\sqrt{e} \sigma }{d}$}
\cellA{rYg}{cYg}{}{}{}{$\frac{1}{d}$}
}

\newcommand{\Gaugedd}{
\cellA{rX1}{cX1}{}{}{}{$\frac{1}{d}$}
\cellA{rX1}{cXZ}{}{}{}{$\frac{\sqrt{e}}{d}$}
\cellA{rX1}{cXY}{}{}{}{$\frac{1}{\sqrt{d}}$}
\cellA{rXZ}{cX1}{}{}{}{$\frac{\sqrt{e}}{d}$}
\cellA{rXZ}{cXZ}{}{}{}{$\frac{e r_1}{d}$}
\cellA{rXZ}{cXY}{}{}{}{$\frac{\sqrt{e} r_2}{\sqrt{d}}$}
\cellA{rXZ}{cWZ}{}{}{}{$\frac{\sqrt{e} s_1}{\sqrt{d}}$}
\cellA{rXZ}{cWY}{}{}{}{$s_2$}
\cellA{rXY}{cX1}{}{}{}{$\frac{1}{\sqrt{d}}$}
\cellA{rXY}{cXZ}{}{}{}{$\frac{\sqrt{e} r_2}{\sqrt{d}}$}
\cellA{rXY}{cXY}{}{}{}{$r_3$}
\cellA{rXY}{cWZ}{}{}{}{$s_2$}
\cellA{rXY}{cWY}{}{}{}{$-\frac{\sqrt{e} s_1}{\sqrt{d}}$}
\cellA{rXY}{cgY}{}{}{}{$1$}
\cellA{rWZ}{cXZ}{}{}{}{$\frac{\sqrt{e} s_1}{\sqrt{d}}$}
\cellA{rWZ}{cXY}{}{}{}{$s_2$}
\cellA{rWZ}{cWZ}{}{}{}{$-r_1$}
\cellA{rWZ}{cWY}{}{}{}{$-\frac{\sqrt{d} r_2}{\sqrt{e}}$}
\cellA{rWY}{cXZ}{}{}{}{$s_2$}
\cellA{rWY}{cXY}{}{}{}{$-\frac{\sqrt{e} s_1}{\sqrt{d}}$}
\cellA{rWY}{cWZ}{}{}{}{$-\frac{\sqrt{d} r_2}{\sqrt{e}}$}
\cellA{rWY}{cWY}{}{}{}{$r_1$}
\cellA{rWY}{cgY}{}{}{}{$-\sigma $}
\cellA{rgY}{cXY}{}{}{}{$1$}
\cellA{rgY}{cWY}{}{}{}{$-\sigma $}
\cellA{rgY}{cgY}{}{}{}{$1$}
}

\begin{align*}
\begin{tikzpicture}[yscale=.8]
\principalmatrixA
\Gaugedp
\end{tikzpicture}
\begin{tikzpicture}[yscale=.8]
\dualmatrixA
\Gaugedd
\end{tikzpicture}
\end{align*}

We next transfer these elements to the lopsided two-sided graph planar algebra, and make a gauge choice so that all of the coefficients lie in the field generated by the index $d^2 = \lambda$.
Thus, we define $$K^{(i)}_{\text{lopsided},0} = \natural(K^{(i)})$$ for $i=1$ and $2$, and see that these are given by

\begin{align*}
\begin{tikzpicture}[yscale=.8] 
\principalmatrixA 
\cellA{r1X}{c1X}{}{}{}{$\frac{1}{d}$}
\cellA{r1X}{cZX}{}{}{}{$\frac{1}{d}$}
\cellA{r1X}{cYX}{}{}{}{$\frac{1}{d}$}
\cellA{rZW}{cZW}{}{}{}{$-\frac{r_1}{d}$}
\cellA{rZW}{cZX}{}{}{}{$\frac{s_1}{\sqrt{d} \sqrt{e}}$}
\cellA{rZW}{cYW}{}{}{}{$-\frac{r_2}{d}$}
\cellA{rZW}{cYX}{}{}{}{$\frac{s_2}{\sqrt{d} \sqrt{e}}$}
\cellA{rZX}{c1X}{}{}{}{$\frac{1}{d}$}
\cellA{rZX}{cZW}{}{}{}{$\frac{\sqrt{e} s_1}{d^{3/2}}$}
\cellA{rZX}{cZX}{}{}{}{$\frac{r_1}{d}$}
\cellA{rZX}{cYW}{}{}{}{$\frac{\sqrt{e} s_2}{d^{3/2}}$}
\cellA{rZX}{cYX}{}{}{}{$\frac{r_2}{d}$}
\cellA{rYX}{c1X}{}{}{}{$\frac{1}{d}$}
\cellA{rYX}{cZW}{}{}{}{$\frac{\sqrt{e} s_2}{d^{3/2}}$}
\cellA{rYX}{cZX}{}{}{}{$\frac{r_2}{d}$}
\cellA{rYX}{cYW}{}{}{}{$-\frac{e^{3/2} s_1}{d^{5/2}}$}
\cellA{rYX}{cYg}{}{}{}{$\frac{1}{d^2}$}
\cellA{rYX}{cYX}{}{}{}{$\frac{r_3}{d}$}
\cellA{rYg}{cYW}{}{}{}{$-\frac{e \sigma }{d^2}$}
\cellA{rYg}{cYg}{}{}{}{$\frac{1}{d^2}$}
\cellA{rYg}{cYX}{}{}{}{$\frac{1}{d}$}
\cellA{rYW}{cZW}{}{}{}{$-\frac{r_2}{d}$}
\cellA{rYW}{cZX}{}{}{}{$\frac{s_2}{\sqrt{d} \sqrt{e}}$}
\cellA{rYW}{cYW}{}{}{}{$\frac{e r_1}{d^2}$}
\cellA{rYW}{cYg}{}{}{}{$-\frac{\sigma }{d^2}$}
\cellA{rYW}{cYX}{}{}{}{$-\frac{\sqrt{e} s_1}{d^{3/2}}$}
\end{tikzpicture}
\begin{tikzpicture}[yscale=.8] 
\dualmatrixA 
\cellA{rXZ}{cXY}{}{}{}{$r_2$}
\cellA{rXZ}{cXZ}{}{}{}{$\frac{e r_1}{d}$}
\cellA{rXZ}{cX1}{}{}{}{$\frac{1}{d}$}
\cellA{rXZ}{cWZ}{}{}{}{$\frac{\sqrt{e} s_1}{\sqrt{d}}$}
\cellA{rXZ}{cWY}{}{}{}{$\frac{\sqrt{d} s_2}{\sqrt{e}}$}
\cellA{rX1}{cXY}{}{}{}{$1$}
\cellA{rX1}{cXZ}{}{}{}{$\frac{e}{d}$}
\cellA{rX1}{cX1}{}{}{}{$\frac{1}{d}$}
\cellA{rXY}{cXY}{}{}{}{$r_3$}
\cellA{rXY}{cXZ}{}{}{}{$\frac{e r_2}{d}$}
\cellA{rXY}{cX1}{}{}{}{$\frac{1}{d}$}
\cellA{rXY}{cWZ}{}{}{}{$\frac{\sqrt{e} s_2}{\sqrt{d}}$}
\cellA{rXY}{cWY}{}{}{}{$-\frac{\sqrt{e} s_1}{\sqrt{d}}$}
\cellA{rXY}{cgY}{}{}{}{$1$}
\cellA{rWZ}{cXY}{}{}{}{$\frac{\sqrt{d} s_2}{\sqrt{e}}$}
\cellA{rWZ}{cXZ}{}{}{}{$\frac{\sqrt{e} s_1}{\sqrt{d}}$}
\cellA{rWZ}{cWZ}{}{}{}{$-r_1$}
\cellA{rWZ}{cWY}{}{}{}{$-\frac{d r_2}{e}$}
\cellA{rWY}{cXY}{}{}{}{$-\frac{\sqrt{e} s_1}{\sqrt{d}}$}
\cellA{rWY}{cXZ}{}{}{}{$\frac{\sqrt{e} s_2}{\sqrt{d}}$}
\cellA{rWY}{cWZ}{}{}{}{$-r_2$}
\cellA{rWY}{cWY}{}{}{}{$r_1$}
\cellA{rWY}{cgY}{}{}{}{$-\sigma $}
\cellA{rgY}{cXY}{}{}{}{$1$}
\cellA{rgY}{cWY}{}{}{}{$-\sigma $}
\cellA{rgY}{cgY}{}{}{}{$1$}
\end{tikzpicture}
\end{align*}

Next, we choose elements of the complex gauge group $\mu^{(i)}$ given by the formulas (unspecified entries are all $1$)
\begin{align*}
\mu^{(1)}(1, X) & = \lambda_{1,-2,-1,1}^{(2.25)} & 
 \mu^{(1)}(Z, X) & = \lambda_{1,-2,-1,1}^{(2.25)} & 
 \mu^{(1)}(Z, W) & = \lambda_{13,0,-88,0,-17,0,1}^{(2.637)} \\ 
\mu^{(1)}(Y, X) & = \lambda_{13,-20,9,-1}^{(0.7645)} & 
 \mu^{(1)}(Y, W) & = \lambda_{169,-15,-16,1}^{(0.3244)} & 
 \mu^{(1)}(Y, g) & = \lambda_{13,-25,4,1}^{(1.718)} \\ 
\mu^{(1)}(W, \hat{Z}) & = \lambda_{1,0,-5,0,-22,0,13}^{(2.77)} & 
 \mu^{(1)}(\bar{X}, Y) & = \lambda_{1,1,-16,13}^{(2.94)} &
\mu^{(1)}(\bar{W}, Z) & = \lambda_{13,0,38,0,-45,0,1}^{(0.9413)} \\
 \mu^{(1)}(\bar{g}, Y) & = \lambda_{1,1,-16,13}^{(2.94)} &
 \mu^{(1)}(\hat{Z}, \bar{W}) & = \lambda_{1,0,-5,0,-22,0,13}^{(2.77)} & & 
\end{align*}
and
\begin{align*}
\mu^{(2)}(1, X) & = \lambda_{1,-2,-1,1}^{(2.25)} & 
 \mu^{(2)}(Z, X) & = \lambda_{1,-2,-1,1}^{(2.25)} & 
 \mu^{(2)}(Z, W) & = \lambda_{1,0,-16,0,-29,0,1}^{(4.200)} \\ 
\mu^{(2)}(Y, X) & = \lambda_{1,-6,5,-1}^{(5.049)} & 
 \mu^{(2)}(Y, W) & = \lambda_{1,-15,12,1}^{(14.1)} & 
 \mu^{(2)}(Y, g) & = \lambda_{1,-11,-4,1}^{(11.34)} \\ 
\mu^{(2)}(W, \hat{Z}) & = \lambda_{1,0,-1,0,-2,0,1}^{(0.6671)} & 
 \mu^{(2)}(\bar{X}, Y) & = \lambda_{1,-1,-2,1}^{(0.445)} &
\mu^{(2)}(\bar{W}, Z) & = \lambda_{1,0,-2,0,-1,0,1}^{(1.499)} \\
 \mu^{(2)}(\bar{g}, Y) & = \lambda_{1,-1,-2,1}^{(0.445)}  &
 \mu^{(2)}(\hat{Z}, \bar{W}) & = \lambda_{1,0,-1,0,-2,0,1}^{(0.6671)} & &
\end{align*}

(Here, $\lambda_p^{(x)}$ denotes the root of $\sum p_i x^{n-i}$ which is approximately equal to $x$.)

Applying these, we have  $K^{(1)}_{\text{lopsided}} = \mu^{(1)} (K^{(1)}_{\text{lopsided},0})$  given by
$$
\begin{tikzpicture}[xscale=2, yscale=.8] 
\principalmatrixA 
\longcellA{r1X}{c1X}{}{}{}{$1$}
\longcellA{r1X}{cZX}{}{}{}{$1$}
\longcellA{r1X}{cYX}{}{}{}{$1$}
\longcellA{rZX}{c1X}{}{}{}{$1$}
\longcellA{rZX}{cZW}{}{}{}{$1$}
\longcellA{rZX}{cZX}{}{}{}{$\{-4,21,-5\}$}
\longcellA{rZX}{cYW}{}{}{}{$1$}
\longcellA{rZX}{cYX}{}{}{}{$\{-3,16,-4\}$}
\longcellA{rZW}{cZW}{}{}{}{$\{4,-21,12\}$}
\longcellA{rZW}{cZX}{}{}{}{$1$}
\longcellA{rZW}{cYW}{}{}{}{$\left\{-\frac{11}{13},\frac{61}{13},-\frac{32}{13}\right\}$}
\longcellA{rZW}{cYX}{}{}{}{$1$}
\longcellA{rYg}{cYW}{}{}{}{$\left\{-\frac{11}{13},\frac{61}{13},-\frac{32}{13}\right\}$}
\longcellA{rYg}{cYg}{}{}{}{$1$}
\longcellA{rYg}{cYX}{}{}{}{$1$}
\longcellA{rYX}{c1X}{}{}{}{$1$}
\longcellA{rYX}{cZW}{}{}{}{$1$}
\longcellA{rYX}{cZX}{}{}{}{$\{-3,16,-4\}$}
\longcellA{rYX}{cYW}{}{}{}{$\left\{-\frac{15}{13},\frac{82}{13},-\frac{33}{13}\right\}$}
\longcellA{rYX}{cYg}{}{}{}{$1$}
\longcellA{rYX}{cYX}{}{}{}{$\{-3,16,-5\}$}
\longcellA{rYW}{cZW}{}{}{}{$\left\{-\frac{11}{13},\frac{61}{13},-\frac{32}{13}\right\}$}
\longcellA{rYW}{cZX}{}{}{}{$1$}
\longcellA{rYW}{cYW}{}{}{}{$\left\{\frac{105}{169},-\frac{613}{169},\frac{400}{169}\right\}$}
\longcellA{rYW}{cYg}{}{}{}{$\left\{-\frac{11}{13},\frac{61}{13},-\frac{32}{13}\right\}$}
\longcellA{rYW}{cYX}{}{}{}{$\left\{-\frac{15}{13},\frac{82}{13},-\frac{33}{13}\right\}$}
\end{tikzpicture}
$$
$$
\begin{tikzpicture}[xscale=2, yscale=.8] 
\dualmatrixA 
\longcellA{rXY}{cXY}{}{}{}{$\{2,-11,3\}$}
\longcellA{rXY}{cXZ}{}{}{}{$\{-2,11,-4\}$}
\longcellA{rXY}{cX1}{}{}{}{$1$}
\longcellA{rXY}{cWZ}{}{}{}{$\{-1,6,-3\}$}
\longcellA{rXY}{cWY}{}{}{}{$\left\{\frac{4}{13},-\frac{21}{13},\frac{1}{13}\right\}$}
\longcellA{rXY}{cgY}{}{}{}{$\{1,-5,2\}$}
\longcellA{rXZ}{cXY}{}{}{}{$\{3,-16,5\}$}
\longcellA{rXZ}{cXZ}{}{}{}{$\{-3,16,-6\}$}
\longcellA{rXZ}{cX1}{}{}{}{$1$}
\longcellA{rXZ}{cWZ}{}{}{}{$\{-1,6,-3\}$}
\longcellA{rXZ}{cWY}{}{}{}{$\{1,-5,2\}$}
\longcellA{rX1}{cXY}{}{}{}{$\{1,-5,2\}$}
\longcellA{rX1}{cXZ}{}{}{}{$\{-1,6,-3\}$}
\longcellA{rX1}{cX1}{}{}{}{$1$}
\longcellA{rWY}{cXY}{}{}{}{$\left\{-\frac{12}{13},\frac{63}{13},-\frac{16}{13}\right\}$}
\longcellA{rWY}{cXZ}{}{}{}{$\{1,-5,2\}$}
\longcellA{rWY}{cWZ}{}{}{}{$\left\{-\frac{4}{13},\frac{21}{13},-\frac{14}{13}\right\}$}
\longcellA{rWY}{cWY}{}{}{}{$\left\{-\frac{85}{169},\frac{456}{169},-\frac{187}{169}\right\}$}
\longcellA{rWY}{cgY}{}{}{}{$\left\{-\frac{1}{13},\frac{2}{13},\frac{3}{13}\right\}$}
\longcellA{rWZ}{cXY}{}{}{}{$\{-1,6,-2\}$}
\longcellA{rWZ}{cXZ}{}{}{}{$\{1,-5,2\}$}
\longcellA{rWZ}{cWZ}{}{}{}{$\{3,-13,7\}$}
\longcellA{rWZ}{cWY}{}{}{}{$\left\{-\frac{1}{13},\frac{2}{13},\frac{3}{13}\right\}$}
\longcellA{rgY}{cXY}{}{}{}{$\{0,1,0\}$}
\longcellA{rgY}{cWY}{}{}{}{$\left\{-\frac{5}{13},\frac{23}{13},-\frac{11}{13}\right\}$}
\longcellA{rgY}{cgY}{}{}{}{$\{0,1,0\}$}
\end{tikzpicture}
$$

\vspace{24pt}

and  $K^{(2)}_{\text{lopsided}} = \mu^{(2)} (K^{(2)}_{\text{lopsided},0})$  given by
\begin{align*}
\begin{tikzpicture}[yscale=.8] 
\principalmatrixA 
\cellA{r1X}{c1X}{}{}{}{$1$}
\cellA{r1X}{cZX}{}{}{}{$1$}
\cellA{r1X}{cYX}{}{}{}{$1$}
\cellA{rZX}{c1X}{}{}{}{$1$}
\cellA{rZX}{cZW}{}{}{}{$1$}
\cellA{rZX}{cZX}{}{}{}{$\{2,-11,5\}$}
\cellA{rZX}{cYW}{}{}{}{$1$}
\cellA{rZX}{cYX}{}{}{}{$\{1,-6,4\}$}
\cellA{rZW}{cZW}{}{}{}{$\{2,-11,4\}$}
\cellA{rZW}{cZX}{}{}{}{$1$}
\cellA{rZW}{cYW}{}{}{}{$\{1,-5,2\}$}
\cellA{rZW}{cYX}{}{}{}{$1$}
\cellA{rYW}{cZW}{}{}{}{$\{1,-5,2\}$}
\cellA{rYW}{cZX}{}{}{}{$1$}
\cellA{rYW}{cYW}{}{}{}{$\{1,-5,2\}$}
\cellA{rYW}{cYX}{}{}{}{$\{1,-6,3\}$}
\cellA{rYW}{cYg}{}{}{}{$\{1,-5,2\}$}
\cellA{rYg}{cYW}{}{}{}{$\{1,-5,2\}$}
\cellA{rYg}{cYX}{}{}{}{$1$}
\cellA{rYg}{cYg}{}{}{}{$1$}
\cellA{rYX}{c1X}{}{}{}{$1$}
\cellA{rYX}{cZW}{}{}{}{$1$}
\cellA{rYX}{cZX}{}{}{}{$\{1,-6,4\}$}
\cellA{rYX}{cYW}{}{}{}{$\{1,-6,3\}$}
\cellA{rYX}{cYX}{}{}{}{$\{1,-6,5\}$}
\cellA{rYX}{cYg}{}{}{}{$1$}
\end{tikzpicture}
\begin{tikzpicture}[yscale=.8] 
\dualmatrixA 
\cellA{rXY}{cX1}{}{}{}{$1$}
\cellA{rXY}{cXZ}{}{}{}{$\{-2,11,-6\}$}
\cellA{rXY}{cXY}{}{}{}{$\{2,-11,5\}$}
\cellA{rXY}{cWZ}{}{}{}{$\{-1,6,-3\}$}
\cellA{rXY}{cWY}{}{}{}{$\{0,-1,1\}$}
\cellA{rXY}{cgY}{}{}{}{$\{1,-5,2\}$}
\cellA{rXZ}{cX1}{}{}{}{$1$}
\cellA{rXZ}{cXZ}{}{}{}{$\{-1,6,-4\}$}
\cellA{rXZ}{cXY}{}{}{}{$\{1,-6,3\}$}
\cellA{rXZ}{cWZ}{}{}{}{$\{-1,6,-3\}$}
\cellA{rXZ}{cWY}{}{}{}{$\{1,-5,2\}$}
\cellA{rX1}{cX1}{}{}{}{$1$}
\cellA{rX1}{cXZ}{}{}{}{$\{-1,6,-3\}$}
\cellA{rX1}{cXY}{}{}{}{$\{1,-5,2\}$}
\cellA{rWY}{cXZ}{}{}{}{$\{1,-5,2\}$}
\cellA{rWY}{cXY}{}{}{}{$\{0,-1,0\}$}
\cellA{rWY}{cWZ}{}{}{}{$\{0,1,0\}$}
\cellA{rWY}{cWY}{}{}{}{$\{1,-4,1\}$}
\cellA{rWY}{cgY}{}{}{}{$\{1,-4,1\}$}
\cellA{rWZ}{cXZ}{}{}{}{$\{1,-5,2\}$}
\cellA{rWZ}{cXY}{}{}{}{$\{-1,6,-2\}$}
\cellA{rWZ}{cWZ}{}{}{}{$\{-1,5,-1\}$}
\cellA{rWZ}{cWY}{}{}{}{$\{1,-4,1\}$}
\cellA{rgY}{cXY}{}{}{}{$\{0,1,0\}$}
\cellA{rgY}{cWY}{}{}{}{$\{1,-3,1\}$}
\cellA{rgY}{cgY}{}{}{}{$\{0,1,0\}$}
\end{tikzpicture}
\end{align*}
We use the abbreviation $\{a,b,c\}$ as shorthand for $a d^4 + b d^2 + c$, where as usual $d^2 \approx 5.0489$ is the index.

These may look hideous, but secretly something wonderful has taken place. These explicit connections have all their entries in a fixed (and rather small) algebraic number field. This means that all subsequent calculations (looking for flat elements, particularly) can be efficiently performed by a computer, without the time consuming difficulties of exact arithmetic on arbitrary algebraic numbers.

\subsection{Bi-unitary connections on $\Gamma(\cB)$}\label{biunitaryhex}
In this section, we show that there is a unique bi-unitary connection on $\Gamma(\cB)$, up to gauge equivalence. Moreover, we find a bi-invertible connection in the same complex gauge orbit whose entries lie in $\mathbb{Q}[\delta^2]$. The calculation is straightforward, and does not use flatness in any way.

\label{biunitarysu3}
$$
\begin{tikzpicture}[inner sep=.7mm, yscale=1.7]
\begin{scope}[xshift=-4mm]
\node at (-1,0) {$1$};
\node at (-1,2) {$\hat{1}$};
\node at (-1,4) {$1$};
\node at (0,1) {$\bar{f}$};
\node at (0,3) {$f$};
\node at (1,0) {$A$};
\node at (1,2) {$\hat{H}$};
\node at (1,4) {$A$};
\node at (2,1) {$\bar{B}$};
\node at (2,3) {$B$};
\node at (3,1) {$\bar{F}$};
\node at (3,3) {$F$};
\node at (4,0) {$G$};
\node at (4,2) {$\hat{I}$};
\node at (4,4) {$G$};
\node at (5,0) {$C$};
\node at (5,2) {$\hat{J}$};
\node at (5,4) {$C$};
\node at (6,0) {$E$};
\node at (6,2) {$\hat{K}$};
\node at (6,4) {$E$};
\node at (7,3) {$z$};
\node at (7,1) {$\bar{z}$};
\node at (8,3) {$D$};
\node at (8,1) {$\bar{D}$};
\end{scope}

\node(1) at (-1,0) [circle,fill] {};
\node(h1) at (-1,2) [circle,fill] {};
\node(11) at (-1,4) [circle,fill] {};
\node(bf) at (0,1) [circle,fill] {};
\node(f) at (0,3) [circle,fill] {};
\node(A) at (1,0) [circle,fill] {};
\node(hH) at (1,2) [circle,fill] {};
\node(AA) at (1,4) [circle,fill] {};
\node(bB) at (2,1) [circle,fill] {};
\node(B) at (2,3)  [circle,fill] {};
\node(bF) at (3,1)  [circle,fill] {};
\node(F) at (3,3)  [circle,fill] {};
\node(G) at (4,0)  [circle,fill] {};
\node(hI) at (4,2)  [circle,fill] {};
\node(GG) at (4,4)  [circle,fill] {};
\node(C) at (5,0)  [circle,fill] {};
\node(hJ) at (5,2)  [circle,fill] {};
\node(CC) at (5,4)  [circle,fill] {};
\node(E) at (6,0)  [circle,fill] {};
\node(hK) at (6,2)  [circle,fill] {};
\node(EE) at (6,4)  [circle,fill] {};
\node(z) at (7,3)  [circle,fill] {};
\node(bz) at (7,1)  [circle,fill] {};
\node(D) at (8,3)  [circle,fill] {};
\node(bD) at (8,1)  [circle,fill] {};

\draw (11)--(f)--(AA)--(B)--(CC)--(D)--(EE)--(F)--(AA);
\draw (F)--(GG)--(z);
\draw (h1)--(f)--(hH)--(B)--(hJ)--(D)--(hK)--(F)--(hH);
\draw (F)--(hI)--(z);
\draw (h1)--(bf)--(hH)--(bB)--(hI)--(bD)--(hK)--(bF)--(hH);
\draw (bF)--(hJ)--(bz);
\draw (1)--(bf)--(A)--(bB)--(G)--(bD)--(E)--(bF)--(A);
\draw (bF)--(C)--(bz);
\end{tikzpicture}
$$

Notice that the dual principal matrix has only 1-by-1 and 2-by-2 blocks in it.  Further, each 2-by-2 block contains an entry which is in a 1-by-1 block in the principal matrix; hence, the norms of all the entries in the dual principal (and therefore also principal) matrix are easily determined.  They are the following:

(Here $\lambda_1$, $\lambda_2$, and $\lambda_3$ are 
the roots of $x^6-2x^4-x^2+1$, $x^6-x^4-2x^2+1$ and $x^6+9x^4-x^2-1$ which are approximately $-0.744955, 0.667115$ and $0.619712$, respectively.
As above, we use the abbreviation $\{a,b,c\}$ as shorthand for $a d^4 + b d^2 + c$.)

\begin{equation*}
\begin{tikzpicture}[yscale=.8]
\principalmatrixB 
\normcellB{r1f}{c1f}{}{}{}{$1$}
\normcellB{r1f}{cAf}{}{}{}{$1$}
\normcellB{rHf}{c1f}{}{}{}{$1$}
\normcellB{rHf}{cAF}{}{}{}{$-\lambda _1$}
\normcellB{rHf}{cAf}{}{}{}{$\{1,-5,0\}$}
\normcellB{rHf}{cAB}{}{}{}{$\lambda _3$}
\normcellB{rHF}{cAF}{}{}{}{$\{-2,11,-4\}$}
\normcellB{rHF}{cAf}{}{}{}{$-\lambda _1$}
\normcellB{rHF}{cAB}{}{}{}{$\lambda _2 \{-2,11,-4\}$}
\normcellB{rHF}{cCB}{}{}{}{$1$}
\normcellB{rHF}{cEF}{}{}{}{$1$}
\normcellB{rHB}{cAF}{}{}{}{$\lambda _2 \{-2,11,-4\}$}
\normcellB{rHB}{cAf}{}{}{}{$\lambda _3$}
\normcellB{rHB}{cAB}{}{}{}{$\{3,-16,5\}$}
\normcellB{rHB}{cGF}{}{}{}{$1$}
\normcellB{rID}{cGF}{}{}{}{$-\lambda _1$}
\normcellB{rID}{cGz}{}{}{}{$\lambda _2$}
\normcellB{rID}{cEF}{}{}{}{$1$}
\normcellB{rIB}{cAF}{}{}{}{$1$}
\normcellB{rIB}{cGF}{}{}{}{$\lambda _2$}
\normcellB{rIB}{cGz}{}{}{}{$-\lambda _1$}
\normcellB{rJF}{cAB}{}{}{}{$1$}
\normcellB{rJF}{cCB}{}{}{}{$\lambda _2$}
\normcellB{rJF}{cCD}{}{}{}{$-\lambda _1$}
\normcellB{rJF}{cED}{}{}{}{$1$}
\normcellB{rJz}{cCB}{}{}{}{$-\lambda _1$}
\normcellB{rJz}{cCD}{}{}{}{$\lambda _2$}
\normcellB{rKF}{cAF}{}{}{}{$1$}
\normcellB{rKF}{cCD}{}{}{}{$1$}
\normcellB{rKF}{cEF}{}{}{}{$\{-1,6,-4\}$}
\normcellB{rKF}{cED}{}{}{}{$\lambda _1 \{1,-6,4\}$}
\normcellB{rKD}{cGF}{}{}{}{$1$}
\normcellB{rKD}{cEF}{}{}{}{$\lambda _1 \{1,-6,4\}$}
\normcellB{rKD}{cED}{}{}{}{$\{-1,6,-4\}$}
\end{tikzpicture}
\end{equation*}
\begin{equation*}
\begin{tikzpicture}[yscale=.8]
\dualmatrixB 
\normcellB{rf1}{cfA}{}{}{}{$\lambda _3 \{2,-11,6\}$}
\normcellB{rf1}{cf1}{}{}{}{$\{2,-11,5\}$}
\normcellB{rfH}{cfA}{}{}{}{$\{2,-11,5\}$}
\normcellB{rfH}{cf1}{}{}{}{$\lambda _3 \{2,-11,6\}$}
\normcellB{rfH}{cBA}{}{}{}{$1$}
\normcellB{rfH}{cFA}{}{}{}{$1$}
\normcellB{rBH}{cfA}{}{}{}{$1$}
\normcellB{rBH}{cBA}{}{}{}{$1$}
\normcellB{rBH}{cFA}{}{}{}{$\{2,-11,5\}$}
\normcellB{rBH}{cFC}{}{}{}{$\lambda _3 \{2,-11,6\}$}
\normcellB{rBJ}{cFA}{}{}{}{$\lambda _3 \{2,-11,6\}$}
\normcellB{rBJ}{cFC}{}{}{}{$\{2,-11,5\}$}
\normcellB{rBJ}{czC}{}{}{}{$1$}
\normcellB{rFH}{cfA}{}{}{}{$1$}
\normcellB{rFH}{cBA}{}{}{}{$\{2,-11,5\}$}
\normcellB{rFH}{cBG}{}{}{}{$\lambda _3 \{2,-11,6\}$}
\normcellB{rFH}{cFA}{}{}{}{$\{-2,11,-4\}$}
\normcellB{rFH}{cFE}{}{}{}{$\lambda _2 \{1,-5,1\}$}
\normcellB{rFK}{cFA}{}{}{}{$\lambda _2 \{1,-5,1\}$}
\normcellB{rFK}{cFE}{}{}{}{$\{-2,11,-4\}$}
\normcellB{rFK}{cDE}{}{}{}{$\{-2,11,-4\}$}
\normcellB{rFK}{cDG}{}{}{}{$\lambda _2 \{1,-5,1\}$}
\normcellB{rFI}{cBA}{}{}{}{$\lambda _3 \{2,-11,6\}$}
\normcellB{rFI}{cBG}{}{}{}{$\{2,-11,5\}$}
\normcellB{rFI}{cDE}{}{}{}{$\lambda _2 \{1,-5,1\}$}
\normcellB{rFI}{cDG}{}{}{}{$\{-2,11,-4\}$}
\normcellB{rzI}{cBG}{}{}{}{$1$}
\normcellB{rzI}{cDG}{}{}{}{$1$}
\normcellB{rDK}{cFC}{}{}{}{$\lambda _2 \{1,-5,1\}$}
\normcellB{rDK}{cFE}{}{}{}{$\{-2,11,-4\}$}
\normcellB{rDK}{cDE}{}{}{}{$1$}
\normcellB{rDJ}{cFC}{}{}{}{$\{-2,11,-4\}$}
\normcellB{rDJ}{cFE}{}{}{}{$\lambda _2 \{1,-5,1\}$}
\normcellB{rDJ}{czC}{}{}{}{$1$}
\end{tikzpicture}
\end{equation*}

From each 2-by-2 block, in either matrix, we get a single equation involving the four phases of that matrix.  Solving these we obtain the following matrices. Here each $\alpha_i$ is an unknown unit complex number. (A computer has made arbitrary choices about which phases to write in terms of others; don't expect to see patterns here.)

$$
\hspace{-1.2cm}
\begin{tikzpicture}[xscale=1.5,yscale=.8]
\principalmatrixB 
\mediumcellB{r1f}{c1f}{}{}{}{$-\frac{\alpha _1 \alpha _2}{\alpha _3}$}
\mediumcellB{r1f}{cAf}{}{}{}{$\alpha _1$}
\mediumcellB{rHf}{c1f}{}{}{}{$\alpha _2$}
\mediumcellB{rHf}{cAf}{}{}{}{$\alpha _3 \{1,-5,0\}$}
\mediumcellB{rHf}{cAF}{}{}{}{$-\alpha _{11} \lambda _1$}
\mediumcellB{rHf}{cAB}{}{}{}{$\alpha _6 \lambda _3$}
\mediumcellB{rHB}{cAf}{}{}{}{$\alpha _4 \lambda _3$}
\mediumcellB{rHB}{cAF}{}{}{}{$\frac{\alpha _{12} \alpha _{14} \alpha _{18} \alpha _{20} \lambda _2 \{2,-11,4\}}{\alpha _{15} \alpha _{17} \alpha _{19}}$}
\mediumcellB{rHB}{cAB}{}{}{}{$\alpha _7 \{3,-16,5\}$}
\mediumcellB{rHB}{cGF}{}{}{}{$\alpha _{12}$}
\mediumcellB{rHF}{cAf}{}{}{}{$-\alpha _5 \lambda _1$}
\mediumcellB{rHF}{cAF}{}{}{}{$\frac{\alpha _{13} \alpha _{16} \alpha _{25} \{-2,11,-4\}}{\alpha _{18} \alpha _{24}}$}
\mediumcellB{rHF}{cAB}{}{}{}{$\frac{\alpha _8 \alpha _9 \alpha _{22} \alpha _{24} \lambda _2 \{2,-11,4\}}{\alpha _{10} \alpha _{21} \alpha _{23}}$}
\mediumcellB{rHF}{cCB}{}{}{}{$\alpha _8$}
\mediumcellB{rHF}{cEF}{}{}{}{$\alpha _{13}$}
\mediumcellB{rID}{cGF}{}{}{}{$\frac{\alpha _{15} \alpha _{17} \lambda _1}{\alpha _{18}}$}
\mediumcellB{rID}{cGz}{}{}{}{$\alpha _{20} \lambda _2$}
\mediumcellB{rID}{cEF}{}{}{}{$\alpha _{15}$}
\mediumcellB{rIB}{cAF}{}{}{}{$\alpha _{14}$}
\mediumcellB{rIB}{cGF}{}{}{}{$\frac{\alpha _{15} \alpha _{17} \alpha _{19} \lambda _2}{\alpha _{18} \alpha _{20}}$}
\mediumcellB{rIB}{cGz}{}{}{}{$-\alpha _{19} \lambda _1$}
\mediumcellB{rJz}{cCB}{}{}{}{$-\alpha _{10} \lambda _1$}
\mediumcellB{rJz}{cCD}{}{}{}{$\alpha _{22} \lambda _2$}
\mediumcellB{rJF}{cAB}{}{}{}{$\alpha _9$}
\mediumcellB{rJF}{cCB}{}{}{}{$\frac{\alpha _{10} \alpha _{21} \alpha _{23} \lambda _2}{\alpha _{22} \alpha _{24}}$}
\mediumcellB{rJF}{cCD}{}{}{}{$\frac{\alpha _{21} \alpha _{23} \lambda _1}{\alpha _{24}}$}
\mediumcellB{rJF}{cED}{}{}{}{$\alpha _{21}$}
\mediumcellB{rKF}{cAF}{}{}{}{$\alpha _{16}$}
\mediumcellB{rKF}{cCD}{}{}{}{$\alpha _{23}$}
\mediumcellB{rKF}{cED}{}{}{}{$\alpha _{24} \lambda _1 \{1,-6,4\}$}
\mediumcellB{rKF}{cEF}{}{}{}{$\frac{\alpha _{18} \alpha _{24} \{1,-6,4\}}{\alpha _{25}}$}
\mediumcellB{rKD}{cGF}{}{}{}{$\alpha _{17}$}
\mediumcellB{rKD}{cED}{}{}{}{$\alpha _{25} \{-1,6,-4\}$}
\mediumcellB{rKD}{cEF}{}{}{}{$\alpha _{18} \lambda _1 \{1,-6,4\}$}
\end{tikzpicture}
$$

$$
\hspace{-1.2cm}
\begin{tikzpicture}[xscale=1.5, yscale=.8]
\dualmatrixB 
\mediumcellB{rf1}{cfA}{}{}{}{$\alpha _1 \lambda _3 \{2,-11,6\}$}
\mediumcellB{rf1}{cf1}{}{}{}{$\frac{\alpha _1 \alpha _2 \{-2,11,-5\}}{\alpha _3}$}
\mediumcellB{rfH}{cfA}{}{}{}{$\alpha _3 \{2,-11,5\}$}
\mediumcellB{rfH}{cf1}{}{}{}{$\alpha _2 \lambda _3 \{2,-11,6\}$}
\mediumcellB{rfH}{cBA}{}{}{}{$\alpha _4$}
\mediumcellB{rfH}{cFA}{}{}{}{$\alpha _5$}
\mediumcellB{rBJ}{cFA}{}{}{}{$\alpha _9 \lambda _3 \{2,-11,6\}$}
\mediumcellB{rBJ}{cFC}{}{}{}{$\frac{\alpha _{10} \alpha _{21} \alpha _{23} \{2,-11,5\}}{\alpha _{22} \alpha _{24}}$}
\mediumcellB{rBJ}{czC}{}{}{}{$\alpha _{10}$}
\mediumcellB{rBH}{cfA}{}{}{}{$\alpha _6$}
\mediumcellB{rBH}{cBA}{}{}{}{$\alpha _7$}
\mediumcellB{rBH}{cFA}{}{}{}{$\frac{\alpha _8 \alpha _9 \alpha _{22} \alpha _{24} \{-2,11,-5\}}{\alpha _{10} \alpha _{21} \alpha _{23}}$}
\mediumcellB{rBH}{cFC}{}{}{}{$\alpha _8 \lambda _3 \{2,-11,6\}$}
\mediumcellB{rFH}{cfA}{}{}{}{$\alpha _{11}$}
\mediumcellB{rFH}{cBA}{}{}{}{$\frac{\alpha _{12} \alpha _{14} \alpha _{18} \alpha _{20} \{-2,11,-5\}}{\alpha _{15} \alpha _{17} \alpha _{19}}$}
\mediumcellB{rFH}{cBG}{}{}{}{$\alpha _{12} \lambda _3 \{2,-11,6\}$}
\mediumcellB{rFH}{cFA}{}{}{}{$\frac{\alpha _{13} \alpha _{16} \alpha _{25} \{-2,11,-4\}}{\alpha _{18} \alpha _{24}}$}
\mediumcellB{rFH}{cFE}{}{}{}{$\alpha _{13} \lambda _2 \{1,-5,1\}$}
\mediumcellB{rFK}{cFA}{}{}{}{$\alpha _{16} \lambda _2 \{1,-5,1\}$}
\mediumcellB{rFK}{cFE}{}{}{}{$\frac{\alpha _{18} \alpha _{24} \{2,-11,4\}}{\alpha _{25}}$}
\mediumcellB{rFK}{cDG}{}{}{}{$\alpha _{17} \lambda _2 \{1,-5,1\}$}
\mediumcellB{rFK}{cDE}{}{}{}{$\alpha _{18} \{-2,11,-4\}$}
\mediumcellB{rFI}{cBA}{}{}{}{$\alpha _{14} \lambda _3 \{2,-11,6\}$}
\mediumcellB{rFI}{cBG}{}{}{}{$\frac{\alpha _{15} \alpha _{17} \alpha _{19} \{2,-11,5\}}{\alpha _{18} \alpha _{20}}$}
\mediumcellB{rFI}{cDG}{}{}{}{$\frac{\alpha _{15} \alpha _{17} \{2,-11,4\}}{\alpha _{18}}$}
\mediumcellB{rFI}{cDE}{}{}{}{$\alpha _{15} \lambda _2 \{1,-5,1\}$}
\mediumcellB{rzI}{cBG}{}{}{}{$\alpha _{19}$}
\mediumcellB{rzI}{cDG}{}{}{}{$\alpha _{20}$}
\mediumcellB{rDK}{cFE}{}{}{}{$\alpha _{24} \{-2,11,-4\}$}
\mediumcellB{rDK}{cFC}{}{}{}{$\alpha _{23} \lambda _2 \{1,-5,1\}$}
\mediumcellB{rDK}{cDE}{}{}{}{$\alpha _{25}$}
\mediumcellB{rDJ}{cFE}{}{}{}{$\alpha _{21} \lambda _2 \{1,-5,1\}$}
\mediumcellB{rDJ}{cFC}{}{}{}{$\frac{\alpha _{21} \alpha _{23} \{2,-11,4\}}{\alpha _{24}}$}
\mediumcellB{rDJ}{czC}{}{}{}{$\alpha _{22}$}
\end{tikzpicture}$$

We still have further constraints:  the phase equations coming from the  three-by-three matrix in the principal graph connection.  The equations coming from orthogonality of columns are all of the form
$$m_1 \phi_1 + m_2 \phi_2 + m_3 \phi_3=0, \qquad m_i \in \mathbb{R}_+, \qquad \phi_i \in \mathbb{T}.$$ 
In each of these equations, we find that two of the $m_i$ sum to the remaining third; thus the corresponding equations on phases are of the form $\phi_i = \phi_j = - \phi_k$.  
So the orthogonality of the first and third columns let us determine $\alpha_3$ and $\alpha_5$ in terms of the phases of the remaining entries;
similarly the orthogonality of the second and third columns give us expressions for $\alpha_6$ and $\alpha_7$.  The connection is then:

\begin{figure}
\hspace*{-1.5cm}
\rotatebox{90}{
\begin{tikzpicture}[xscale=2.3, yscale=.8]
\principalmatrixB 
\longcellB{r1f}{c1f}{}{}{}{$\frac{\alpha _1 \alpha _2 \alpha _{12} \alpha _{14} \alpha _{18} \alpha _{20}}{\alpha _4 \alpha _{11} \alpha _{15} \alpha _{17} \alpha _{19}}$}
\longcellB{r1f}{cAf}{}{}{}{$\alpha _1$}
\longcellB{rHf}{c1f}{}{}{}{$\alpha _2$}
\longcellB{rHf}{cAB}{}{}{}{$\frac{\alpha _8 \alpha _9 \alpha _{11} \alpha _{18} \alpha _{22} \alpha _{24}^2 \lambda _3}{\alpha _{10} \alpha _{13} \alpha _{16} \alpha _{21} \alpha _{23} \alpha _{25}}$}
\longcellB{rHf}{cAF}{}{}{}{$-\alpha _{11} \lambda _1$}
\longcellB{rHf}{cAf}{}{}{}{$\frac{\alpha _4 \alpha _{11} \alpha _{15} \alpha _{17} \alpha _{19} \{-1,5,0\}}{\alpha _{12} \alpha _{14} \alpha _{18} \alpha _{20}}$}
\longcellB{rHF}{cAB}{}{}{}{$\frac{\alpha _8 \alpha _9 \alpha _{22} \alpha _{24} \lambda _2 \{2,-11,4\}}{\alpha _{10} \alpha _{21} \alpha _{23}}$}
\longcellB{rHF}{cAF}{}{}{}{$\frac{\alpha _{13} \alpha _{16} \alpha _{25} \{-2,11,-4\}}{\alpha _{18} \alpha _{24}}$}
\longcellB{rHF}{cAf}{}{}{}{$-\frac{\alpha _4 \alpha _{13} \alpha _{15} \alpha _{16} \alpha _{17} \alpha _{19} \alpha _{25} \lambda _1}{\alpha _{12} \alpha _{14} \alpha _{18}^2 \alpha _{20} \alpha _{24}}$}
\longcellB{rHF}{cCB}{}{}{}{$\alpha _8$}
\longcellB{rHF}{cEF}{}{}{}{$\alpha _{13}$}
\longcellB{rHB}{cAB}{}{}{}{$\frac{\alpha _8 \alpha _9 \alpha _{12} \alpha _{14} \alpha _{18}^2 \alpha _{20} \alpha _{22} \alpha _{24}^2 \{3,-16,5\}}{\alpha _{10} \alpha _{13} \alpha _{15} \alpha _{16} \alpha _{17} \alpha _{19} \alpha _{21} \alpha _{23} \alpha _{25}}$}
\longcellB{rHB}{cAF}{}{}{}{$\frac{\alpha _{12} \alpha _{14} \alpha _{18} \alpha _{20} \lambda _2 \{2,-11,4\}}{\alpha _{15} \alpha _{17} \alpha _{19}}$}
\longcellB{rHB}{cAf}{}{}{}{$\alpha _4 \lambda _3$}
\longcellB{rHB}{cGF}{}{}{}{$\alpha _{12}$}
\longcellB{rID}{cGz}{}{}{}{$\alpha _{20} \lambda _2$}
\longcellB{rID}{cGF}{}{}{}{$\frac{\alpha _{15} \alpha _{17} \lambda _1}{\alpha _{18}}$}
\longcellB{rID}{cEF}{}{}{}{$\alpha _{15}$}
\longcellB{rIB}{cAF}{}{}{}{$\alpha _{14}$}
\longcellB{rIB}{cGz}{}{}{}{$-\alpha _{19} \lambda _1$}
\longcellB{rIB}{cGF}{}{}{}{$\frac{\alpha _{15} \alpha _{17} \alpha _{19} \lambda _2}{\alpha _{18} \alpha _{20}}$}
\longcellB{rJz}{cCB}{}{}{}{$-\alpha _{10} \lambda _1$}
\longcellB{rJz}{cCD}{}{}{}{$\alpha _{22} \lambda _2$}
\longcellB{rJF}{cAB}{}{}{}{$\alpha _9$}
\longcellB{rJF}{cCB}{}{}{}{$\frac{\alpha _{10} \alpha _{21} \alpha _{23} \lambda _2}{\alpha _{22} \alpha _{24}}$}
\longcellB{rJF}{cCD}{}{}{}{$\frac{\alpha _{21} \alpha _{23} \lambda _1}{\alpha _{24}}$}
\longcellB{rJF}{cED}{}{}{}{$\alpha _{21}$}
\longcellB{rKF}{cAF}{}{}{}{$\alpha _{16}$}
\longcellB{rKF}{cCD}{}{}{}{$\alpha _{23}$}
\longcellB{rKF}{cED}{}{}{}{$\alpha _{24} \lambda _1 \{1,-6,4\}$}
\longcellB{rKF}{cEF}{}{}{}{$\frac{\alpha _{18} \alpha _{24} \{1,-6,4\}}{\alpha _{25}}$}
\longcellB{rKD}{cGF}{}{}{}{$\alpha _{17}$}
\longcellB{rKD}{cED}{}{}{}{$\alpha _{25} \{-1,6,-4\}$}
\longcellB{rKD}{cEF}{}{}{}{$\alpha _{18} \lambda _1 \{1,-6,4\}$}
\begin{scope}[yshift=-12cm]
\dualmatrixB 
\longcellB{rf1}{cf1}{}{}{}{$\frac{\alpha _1 \alpha _2 \alpha _{12} \alpha _{14} \alpha _{18} \alpha _{20} \{2,-11,5\}}{\alpha _4 \alpha _{11} \alpha _{15} \alpha _{17} \alpha _{19}}$}
\longcellB{rf1}{cfA}{}{}{}{$\alpha _1 \lambda _3 \{2,-11,6\}$}
\longcellB{rfH}{cf1}{}{}{}{$\alpha _2 \lambda _3 \{2,-11,6\}$}
\longcellB{rfH}{cfA}{}{}{}{$\frac{\alpha _4 \alpha _{11} \alpha _{15} \alpha _{17} \alpha _{19} \{-2,11,-5\}}{\alpha _{12} \alpha _{14} \alpha _{18} \alpha _{20}}$}
\longcellB{rfH}{cBA}{}{}{}{$\alpha _4$}
\longcellB{rfH}{cFA}{}{}{}{$\frac{\alpha _4 \alpha _{13} \alpha _{15} \alpha _{16} \alpha _{17} \alpha _{19} \alpha _{25}}{\alpha _{12} \alpha _{14} \alpha _{18}^2 \alpha _{20} \alpha _{24}}$}
\longcellB{rBH}{cfA}{}{}{}{$\frac{\alpha _8 \alpha _9 \alpha _{11} \alpha _{18} \alpha _{22} \alpha _{24}^2}{\alpha _{10} \alpha _{13} \alpha _{16} \alpha _{21} \alpha _{23} \alpha _{25}}$}
\longcellB{rBH}{cBA}{}{}{}{$\frac{\alpha _8 \alpha _9 \alpha _{12} \alpha _{14} \alpha _{18}^2 \alpha _{20} \alpha _{22} \alpha _{24}^2}{\alpha _{10} \alpha _{13} \alpha _{15} \alpha _{16} \alpha _{17} \alpha _{19} \alpha _{21} \alpha _{23} \alpha _{25}}$}
\longcellB{rBH}{cFC}{}{}{}{$\alpha _8 \lambda _3 \{2,-11,6\}$}
\longcellB{rBH}{cFA}{}{}{}{$\frac{\alpha _8 \alpha _9 \alpha _{22} \alpha _{24} \{-2,11,-5\}}{\alpha _{10} \alpha _{21} \alpha _{23}}$}
\longcellB{rBJ}{cFC}{}{}{}{$\frac{\alpha _{10} \alpha _{21} \alpha _{23} \{2,-11,5\}}{\alpha _{22} \alpha _{24}}$}
\longcellB{rBJ}{cFA}{}{}{}{$\alpha _9 \lambda _3 \{2,-11,6\}$}
\longcellB{rBJ}{czC}{}{}{}{$\alpha _{10}$}
\longcellB{rFH}{cfA}{}{}{}{$\alpha _{11}$}
\longcellB{rFH}{cBA}{}{}{}{$\frac{\alpha _{12} \alpha _{14} \alpha _{18} \alpha _{20} \{-2,11,-5\}}{\alpha _{15} \alpha _{17} \alpha _{19}}$}
\longcellB{rFH}{cBG}{}{}{}{$\alpha _{12} \lambda _3 \{2,-11,6\}$}
\longcellB{rFH}{cFA}{}{}{}{$\frac{\alpha _{13} \alpha _{16} \alpha _{25} \{-2,11,-4\}}{\alpha _{18} \alpha _{24}}$}
\longcellB{rFH}{cFE}{}{}{}{$\alpha _{13} \lambda _2 \{1,-5,1\}$}
\longcellB{rFK}{cFA}{}{}{}{$\alpha _{16} \lambda _2 \{1,-5,1\}$}
\longcellB{rFK}{cFE}{}{}{}{$\frac{\alpha _{18} \alpha _{24} \{2,-11,4\}}{\alpha _{25}}$}
\longcellB{rFK}{cDG}{}{}{}{$\alpha _{17} \lambda _2 \{1,-5,1\}$}
\longcellB{rFK}{cDE}{}{}{}{$\alpha _{18} \{-2,11,-4\}$}
\longcellB{rFI}{cBA}{}{}{}{$\alpha _{14} \lambda _3 \{2,-11,6\}$}
\longcellB{rFI}{cBG}{}{}{}{$\frac{\alpha _{15} \alpha _{17} \alpha _{19} \{2,-11,5\}}{\alpha _{18} \alpha _{20}}$}
\longcellB{rFI}{cDG}{}{}{}{$\frac{\alpha _{15} \alpha _{17} \{2,-11,4\}}{\alpha _{18}}$}
\longcellB{rFI}{cDE}{}{}{}{$\alpha _{15} \lambda _2 \{1,-5,1\}$}
\longcellB{rzI}{cBG}{}{}{}{$\alpha _{19}$}
\longcellB{rzI}{cDG}{}{}{}{$\alpha _{20}$}
\longcellB{rDJ}{cFC}{}{}{}{$\frac{\alpha _{21} \alpha _{23} \{2,-11,4\}}{\alpha _{24}}$}
\longcellB{rDJ}{cFE}{}{}{}{$\alpha _{21} \lambda _2 \{1,-5,1\}$}
\longcellB{rDJ}{czC}{}{}{}{$\alpha _{22}$}
\longcellB{rDK}{cFC}{}{}{}{$\alpha _{23} \lambda _2 \{1,-5,1\}$}
\longcellB{rDK}{cFE}{}{}{}{$\alpha _{24} \{-2,11,-4\}$}
\longcellB{rDK}{cDE}{}{}{}{$\alpha _{25}$}
\end{scope}
\end{tikzpicture}}
\end{figure}

We now pass to the lopsided planar algebra using the map $\natural$, and choose an element of the complex gauge group which simultaneously demonstrates that the connections shown above all form a single gauge orbit, and gives us a nice representative with all entries in the number field $\mathbb{Q}[d^2]$.  Gauge entries that are not specified below are all equal to $1$.

\begin{align*}
\mu(1, f) & = \frac{\alpha _4 \alpha _{11} \alpha _{15} \alpha _{17} \alpha _{19} \lambda_{1,-2,-1,1}^{(2.25)}}{\alpha _1 \alpha _2 \alpha _{12} \alpha _{14} \alpha _{18} \alpha _{20}} & 
 \mu(A, f) & = \frac{\lambda_{1,-2,-1,1}^{(2.25)}}{\alpha _1} \\ 
\mu(A, B) & = \frac{\alpha _{10} \alpha _{13} \alpha _{16} \alpha _{21} \alpha _{23} \alpha _{25} \lambda_{1,-5,6,-1}^{(3.25)}}{\alpha _8 \alpha _9 \alpha _{11} \alpha _{18} \alpha _{22} \alpha _{24}^2} & 
 \mu(A, F) & = \frac{\lambda_{1,-2,-1,1}^{(2.25)}}{\alpha _{11}} \\ 
\mu(C, B) & = \frac{\alpha _{13} \lambda_{1,0,-4,0,3,0,1}^{(1.202)}}{\alpha _8} & 
 \mu(C, D) & = \frac{\alpha _{24} \lambda_{1,3,-4,1}^{(-4.049)}}{\alpha _{21} \alpha _{23}} \\ 
\mu(G, F) & = \frac{\lambda_{1,0,-1,0,-9,0,1}^{(1.869)}}{\alpha _{12}} & 
 \mu(G, z) & = \frac{\alpha _{17} \alpha _{25} \lambda_{1,0,-335,0,-44,0,-1}^{(-18.31)}}{\alpha _{12} \alpha _{18} \alpha _{20} \alpha _{24}} \\ 
 \mu(E, D) & = \frac{\lambda_{1,0,-9,0,-1,0,1}^{(3.016)}}{\alpha _{21}} \
&
 \mu(f, \hat{H}) & = \frac{\alpha _1 \alpha _{12} \alpha _{14} \alpha _{18} \alpha _{20}}{\alpha _4 \alpha _{11} \alpha _{15} \alpha _{17} \alpha _{19}} \\ 
 \mu(B, \hat{J}) & = \frac{\alpha _8 \alpha _{22} \alpha _{24} \lambda_{1,0,12,0,-15,0,1}^{(1.037)}}{\alpha _{10} \alpha _{13} \alpha _{21} \alpha _{23}} 
&
 \mu(F, \hat{I}) & = \frac{\alpha _{25} \lambda_{1,0,-4,0,3,0,1}^{(-1.674)}}{\alpha _{15} \alpha _{24}} \\ 
\mu(F, \hat{K}) & = \frac{\alpha _{25} \lambda_{1,1,-2,-1}^{(1.25)}}{\alpha _{18} \alpha _{24}} & 
 \mu(D, \hat{K}) & = \frac{\alpha _{21} \lambda_{1,0,-2,0,-1,0,1}^{(-0.7450)}}{\alpha _{24}} \\ 
\mu(\bar{B}, A) & = \frac{\alpha _{11} \alpha _{15} \alpha _{17} \alpha _{19} \lambda_{1,-1,-2,1}^{(1.80)}}{\alpha _{12} \alpha _{14} \alpha _{18} \alpha _{20}} & 
\mu(\bar{F}, A) & = \frac{\alpha _{11} \alpha _{18} \alpha _{24} \lambda_{1,2,-1,-1}^{(0.802)}}{\alpha _{16} \alpha _{25}} \\ 
 \mu(\bar{z}, C) & = \frac{\alpha _{21} \alpha _{23} \lambda_{1,-4,3,1}^{(1.45)}}{\alpha _{22} \alpha _{24}} &
\mu(\bar{D}, G) & = \frac{\alpha _{12} \alpha _{18} \alpha _{24} \lambda_{1,0,11,0,-4,0,-1}^{(0.7181)}}{\alpha _{17} \alpha _{25}} \\ 
 \mu(\bar{D}, E) & = -\frac{\alpha _{24}}{\alpha _{25}} &
 \mu(\hat{H}, \bar{F}) & = \frac{\lambda_{1,-2,-1,1}^{(2.25)}}{\alpha _{13}} \\ 
\mu(\hat{I}, \bar{B}) & = \frac{\alpha _{12} \alpha _{18} \alpha _{20} \alpha _{24} \lambda_{1,0,15,0,12,0,-1}^{(-0.2758)}}{\alpha _{17} \alpha _{19} \alpha _{25}} & 
\end{align*}

We call this bi-invertible connection $K^B_{\text{lopsided}}$.
$$
\begin{tikzpicture}[yscale=.8]
\principalmatrixB 
\cellB{r1f}{c1f}{}{}{}{$1$}
\cellB{r1f}{cAf}{}{}{}{$1$}
\cellB{rHF}{cAB}{}{}{}{$\{1,-6,4\}$}
\cellB{rHF}{cAf}{}{}{}{$1$}
\cellB{rHF}{cAF}{}{}{}{$1$}
\cellB{rHF}{cCB}{}{}{}{$1$}
\cellB{rHF}{cEF}{}{}{}{$1$}
\cellB{rHB}{cAB}{}{}{}{$\{-1,6,-3\}$}
\cellB{rHB}{cAf}{}{}{}{$1$}
\cellB{rHB}{cAF}{}{}{}{$\{1,-6,4\}$}
\cellB{rHB}{cGF}{}{}{}{$1$}
\cellB{rHf}{c1f}{}{}{}{$1$}
\cellB{rHf}{cAB}{}{}{}{$1$}
\cellB{rHf}{cAf}{}{}{}{$\{-1,5,0\}$}
\cellB{rHf}{cAF}{}{}{}{$1$}
\cellB{rID}{cGz}{}{}{}{$\{2,-12,7\}$}
\cellB{rID}{cGF}{}{}{}{$1$}
\cellB{rID}{cEF}{}{}{}{$1$}
\cellB{rIB}{cAF}{}{}{}{$1$}
\cellB{rIB}{cGz}{}{}{}{$1$}
\cellB{rIB}{cGF}{}{}{}{$\{-3,16,-4\}$}
\cellB{rJz}{cCB}{}{}{}{$1$}
\cellB{rJz}{cCD}{}{}{}{$\{2,-12,7\}$}
\cellB{rJF}{cAB}{}{}{}{$1$}
\cellB{rJF}{cCB}{}{}{}{$\{-3,16,-4\}$}
\cellB{rJF}{cCD}{}{}{}{$1$}
\cellB{rJF}{cED}{}{}{}{$1$}
\cellB{rKF}{cAF}{}{}{}{$1$}
\cellB{rKF}{cCD}{}{}{}{$1$}
\cellB{rKF}{cED}{}{}{}{$\{-2,11,-5\}$}
\cellB{rKF}{cEF}{}{}{}{$\{-2,11,-5\}$}
\cellB{rKD}{cGF}{}{}{}{$1$}
\cellB{rKD}{cED}{}{}{}{$\{-1,6,-4\}$}
\cellB{rKD}{cEF}{}{}{}{$\{-2,11,-5\}$}
\end{tikzpicture}$$
$$
\begin{tikzpicture}[yscale=.8]
\dualmatrixB 
\cellB{rfH}{cf1}{}{}{}{$1$}
\cellB{rfH}{cfA}{}{}{}{$-1$}
\cellB{rfH}{cBA}{}{}{}{$\{0,1,-1\}$}
\cellB{rfH}{cFA}{}{}{}{$\{0,1,-1\}$}
\cellB{rf1}{cf1}{}{}{}{$1$}
\cellB{rf1}{cfA}{}{}{}{$\{0,1,-1\}$}
\cellB{rBJ}{cFA}{}{}{}{$\{1,-5,3\}$}
\cellB{rBJ}{cFC}{}{}{}{$\{-2,11,-4\}$}
\cellB{rBJ}{czC}{}{}{}{$\{-1,6,-3\}$}
\cellB{rBH}{cfA}{}{}{}{$\{1,-5,3\}$}
\cellB{rBH}{cBA}{}{}{}{$\{-1,7,-4\}$}
\cellB{rBH}{cFA}{}{}{}{$\{2,-12,7\}$}
\cellB{rBH}{cFC}{}{}{}{$\{-1,6,-3\}$}
\cellB{rFH}{cfA}{}{}{}{$\{1,-5,2\}$}
\cellB{rFH}{cBG}{}{}{}{$\{1,-5,1\}$}
\cellB{rFH}{cBA}{}{}{}{$\{1,-6,3\}$}
\cellB{rFH}{cFA}{}{}{}{$\{1,-5,2\}$}
\cellB{rFH}{cFE}{}{}{}{$\{-2,11,-3\}$}
\cellB{rFK}{cFA}{}{}{}{$\{1,-5,2\}$}
\cellB{rFK}{cFE}{}{}{}{$\{-3,16,-5\}$}
\cellB{rFK}{cDE}{}{}{}{$\{-3,16,-5\}$}
\cellB{rFK}{cDG}{}{}{}{$\{1,-5,1\}$}
\cellB{rFI}{cBG}{}{}{}{$\{6,-32,9\}$}
\cellB{rFI}{cBA}{}{}{}{$\{1,-5,2\}$}
\cellB{rFI}{cDE}{}{}{}{$\{-2,11,-3\}$}
\cellB{rFI}{cDG}{}{}{}{$\{1,-5,1\}$}
\cellB{rzI}{cBG}{}{}{}{$\{0,1,0\}$}
\cellB{rzI}{cDG}{}{}{}{$\{0,-3,2\}$}
\cellB{rDJ}{cFC}{}{}{}{$\{1,-5,2\}$}
\cellB{rDJ}{cFE}{}{}{}{$\{-1,6,-2\}$}
\cellB{rDJ}{czC}{}{}{}{$\{1,-7,4\}$}
\cellB{rDK}{cFC}{}{}{}{$\{1,-5,2\}$}
\cellB{rDK}{cFE}{}{}{}{$\{-1,5,-1\}$}
\cellB{rDK}{cDE}{}{}{}{$\{1,-5,2\}$}
\end{tikzpicture}$$

 \section{Flat low weight vectors}
 \label{flatlowweight}
This section contains the final ingredients necessary in the proofs of two of our main theorems:

 \begin{thm}\label{thm:uniquehat}
 There is a unique subfactor with principal graph $\Gamma(\cA)$.
 \end{thm}
  
 \begin{thm}\label{thm:uniquehex}
 There is a unique subfactor with principal graph $\Gamma(\cB)$.
 \end{thm}
 
 \begin{proof}[Proof of Theorem \ref{thm:uniquehat}]
 In \S \ref{biunitaryhat} we showed that there are two biunitary connections on $\cA$  which pass the branch-point eigenvalue test.
 
 Since the two vertices at depth 2 in $\Gamma(\cA)$ are self-dual, a flat low weigth $2$-box must have rotational eigenvalue $+1$.
 
 Theorem \ref{thm:flat-generators-A1} shows that $K^{(1)}_{\text{lopsided}}$ has no flat, low-weight $2$-boxes with eigenvalue $+1$, hence by Theorem \ref{thm:spa-and-flat-conns} this connection is not flat.  
 
Thus there is at most one gauge equivalence class of flat connections on the principal graph for $\cA$, and $K^{(2)}_{\text{lopsided}}$ provides a representative of its complex gauge group orbit. Since we know the subfactor exists (easily constructed from quantum groups), this connection must actually be flat by Theorem \ref{thm:spa-and-flat-conns}. Theorem \ref{thm:flat-generators-A2} explicitly describes the flat subalgebra.
 \end{proof}

\begin{proof}[Proof of Theorem \ref{thm:uniquehex}]
In \S \ref{biunitaryhex} we showed that there is one biunitary connection on $\cB$ which passes the branch-point eigenvalue test.  Since we know the subfactor exists (easily constructed from quantum groups), this connection must actually be flat by Theorem \ref{thm:spa-and-flat-conns}. Theorem \ref{thm:flat-generators-B} explicitly describes the flat subalgebra.
\end{proof}
 
 The work in this section consists of calculating the flat low weight vectors in the graph planar algebra, for each of the three connections (two on $\Gamma(\cA)$ and one on $\Gamma(\cB)$) described above.
 
 In order to do this, we will need to write down rather complicated elements of the graph planar algebra. We express these as a collection of matrices; the usual multiplication structure on a graph planar algebra space $\cG(\Gamma)_{n,+}$ breaks up as a direct sum of matrix algebras. Each matrix algebra is indexed by a pair of even vertices $a, b$ on $\Gamma$, and the rows and columns are indexed by paths of length $n$ from $a$ to $b$. Thus if $\pi$ and $\rho$ are paths on $\Gamma$ from $a$ to $b$, we denote by $ \bar{\rho}\pi$ the concatenation of $\pi$ with the reverse of $\rho$, a loop based at $a$, and the $(\rho,\pi)$ entry of the $(a,b)$ matrix gives the coefficient of the loop $\bar{\rho}\pi$.
 
 We need to specify the ordering of paths from $a$ to $b$, in order to fix an ordering of the rows and columns of these matrices. For $\cA$, we use
 \newcommand{\paths}{\operatorname{paths}}
 \begin{align*}
\paths^2_{\mathcal{A}}(1, 1) & = \left\{(1X1)\right\} \displaybreak[1]\\
\paths^2_{\mathcal{A}}(1, Z) & = \left\{(1XZ)\right\} \displaybreak[1]\\
\paths^2_{\mathcal{A}}(1, Y) & = \left\{(1XY)\right\} \displaybreak[1]\\
\paths^2_{\mathcal{A}}(Z, 1) & = \left\{(ZX1)\right\} \displaybreak[1]\\
\paths^2_{\mathcal{A}}(Z, Z) & = \left\{(ZXZ), (ZWZ)\right\} \displaybreak[1]\\
\paths^2_{\mathcal{A}}(Z, Y) & = \left\{(ZXY), (ZWY)\right\} \displaybreak[1]\\
\paths^2_{\mathcal{A}}(Y, 1) & = \left\{(YX1)\right\} \displaybreak[1]\\
\paths^2_{\mathcal{A}}(Y, Z) & = \left\{(YXZ), (YWZ)\right\} \displaybreak[1]\\
\paths^2_{\mathcal{A}}(Y, Y) & = \left\{(YXY), (YWY), (YgY)\right\} \displaybreak[1]\\\end{align*}
and for $\cB$
\begin{align*}
\paths^3_{\mathcal{B}}(1, f) & = \left\{(1f1f), (1fAf)\right\} \displaybreak[1]\\
\paths^3_{\mathcal{B}}(1, B) & = \left\{(1fAB)\right\} \displaybreak[1]\\
\paths^3_{\mathcal{B}}(1, F) & = \left\{(1fAF)\right\} \displaybreak[1]\\
\paths^3_{\mathcal{B}}(A, f) & = \left\{(Af1f), (AfAf), (ABAf), (AFAf)\right\} \displaybreak[1]\\
\paths^3_{\mathcal{B}}(A, B) & = \left\{(AfAB), (ABAB), (ABCB), (AFAB)\right\} \displaybreak[1]\\
\paths^3_{\mathcal{B}}(A, F) & = \left\{(AfAF), (ABAF), (AFAF), (AFGF), (AFEF)\right\} \displaybreak[1]\\
\paths^3_{\mathcal{B}}(A, z) & = \left\{(AFGz)\right\} \displaybreak[1]\\
\paths^3_{\mathcal{B}}(A, D) & = \left\{(ABCD), (AFED)\right\} \displaybreak[1]\\
\paths^3_{\mathcal{B}}(G, f) & = \left\{(GFAf)\right\} \displaybreak[1]\\
\paths^3_{\mathcal{B}}(G, B) & = \left\{(GFAB)\right\} \displaybreak[1]\\
\paths^3_{\mathcal{B}}(G, F) & = \left\{(GFAF), (GFGF), (GFEF), (GzGF)\right\} \displaybreak[1]\\
\paths^3_{\mathcal{B}}(G, z) & = \left\{(GFGz), (GzGz)\right\} \displaybreak[1]\\
\paths^3_{\mathcal{B}}(G, D) & = \left\{(GFED)\right\} \displaybreak[1]\\
\paths^3_{\mathcal{B}}(C, f) & = \left\{(CBAf)\right\} \displaybreak[1]\\
\paths^3_{\mathcal{B}}(C, B) & = \left\{(CBAB), (CBCB), (CDCB)\right\} \displaybreak[1]\\
\paths^3_{\mathcal{B}}(C, F) & = \left\{(CBAF), (CDEF)\right\} \displaybreak[1]\\
\paths^3_{\mathcal{B}}(C, D) & = \left\{(CBCD), (CDCD), (CDED)\right\} \displaybreak[1]\\
\paths^3_{\mathcal{B}}(E, f) & = \left\{(EFAf)\right\} \displaybreak[1]\\
\paths^3_{\mathcal{B}}(E, B) & = \left\{(EFAB), (EDCB)\right\} \displaybreak[1]\\
\paths^3_{\mathcal{B}}(E, F) & = \left\{(EFAF), (EFGF), (EFEF), (EDEF)\right\} \displaybreak[1]\\
\paths^3_{\mathcal{B}}(E, z) & = \left\{(EFGz)\right\} \displaybreak[1]\\
\paths^3_{\mathcal{B}}(E, D) & = \left\{(EFED), (EDCD), (EDED)\right\} \displaybreak[1]\\
 \end{align*}
 
 \begin{thm}
 \label{thm:lws-A}
 The lowest weight eigenspace with eigenvalue $+1$ in the graph planar algebra $\cG(\Gamma(\cA))_{2,+}$ is four dimensional, with basis $\{S_1, S_2, S_3, S_4\}$ given in \S \ref{appendix:low-weight-A}.
 \end{thm}
 \begin{proof}
 Recall that the conditions for $S$ being a lowest weight eigenvector with eigenvalue $\mu$ are
 \begin{align*}
 \begin{tikzpicture}[baseline]
	\draw (-.3,-1)--(-.3,0);
	\draw (.3,-1)--(.3,0);
	\draw (-.3,0) --(-.3,.5) .. controls (-.3,1) and (.3,1) .. (.3,.5)--(.3,0);
	\node[minimum size=1cm, shape=rectangle, fill=white, draw]  at (0,0) {$S$};
\end{tikzpicture} & = 0 & 
 \begin{tikzpicture}[baseline]
	\draw (-.3,-1.5)--(-.3,1.5);
	\draw (.3,.3)--(.3,.5)  .. controls (.3,1) and (.9,1) .. (.9,.5)--(.9,-.5) .. controls (.9,-1) and (.3, -1) .. (.3,-.5) -- (.3,-.3);
	\node[minimum size=1cm, shape=rectangle, fill=white, draw]  at (0,0) {$S$};
\end{tikzpicture} &  = 0
\end{align*}
and
\begin{align*}
 \begin{tikzpicture}[baseline]
	\draw (.3,.3)--(.3,.5) .. controls (.3, .8) and (.9,.8) .. (.9,.5) -- (.9,-.5) .. controls (.9,-1) and (-.3,-1) .. (-.3,-1.5);
	\draw (-.3,-.3)--(-.3,-.5) .. controls (-.3, -.8) and (-.9,-.8) .. (-.9,-.5) -- (-.9,.5) .. controls (-.9,1) and (.3,1) .. (.3,1.5);
	\draw (-.3,.3)--(-.3,.5) .. controls (-.3, 1.1) and (1.5,1.1) .. (1.5,.5) -- (1.5,-.5) .. controls (1.5,-1) and (.3,-1) .. (.3,-1.5);
	\draw (.3,-.3)--(.3,-.5) .. controls (.3, -1.1) and (-1.5,-1.1) .. (-1.5,-.5) -- (-1.5,.5) .. controls (-1.5,1) and (-.3,1) .. (-.3,1.5);
	\node[minimum size=1cm, shape=rectangle, fill=white, draw]  at (0,0) {$S$};
\end{tikzpicture}\;
= \mu \;\;
\begin{tikzpicture}[baseline]
	\draw (-.3,-1)--(-.3,1);
	\draw (.3,-1)--(.3,1);
	\node[minimum size=1cm, shape=rectangle, fill=white, draw]  at (0,0) {$S$};
\end{tikzpicture}
 \end{align*}
 Setting $\mu=1$ and writing this explicitly in the lopsided graph planar algebra, we have 
 \begin{align*}
0 & = \left(d^4-5 d^2+2\right) S(1,X,Y,X)+\left(-d^4+6 d^2-3\right) S(1,X,Z,X)+S(1,X,1,X) \displaybreak[1]\\
0 & = S(Z,X,1,X)+\left(d^4-5 d^2+2\right) S(Z,X,Y,X)+\left(-d^4+6 d^2-3\right) S(Z,X,Z,X) \displaybreak[1]\\
0 & = S(Y,X,1,X)+\left(-d^4+6 d^2-3\right) S(Y,X,Z,X)+\left(d^4-5 d^2+2\right) S(Y,X,Y,X) \displaybreak[1]\\
0 & = \left(2 d^4-10 d^2+3\right) S(Z,W,Y,W)+\left(-d^4+6 d^2-2\right) S(Z,W,Z,W) \displaybreak[1]\\
0 & = \left(-d^4+6 d^2-2\right) S(Y,W,Z,W)+\left(2 d^4-10 d^2+3\right) S(Y,W,Y,W) \displaybreak[1]\\
0 & = \left(d^4-3 d^2+1\right) S(Y,g,Y,g) \displaybreak[1]\\
0 & = S(1,X,1,X) \displaybreak[1]\\
0 & = \left(-2 d^4+11 d^2-4\right) S(Z,W,Z,X)+\left(-2 d^4+11 d^2-4\right) S(Z,X,Z,X) \displaybreak[1]\\
0 & = \left(-2 d^4+11 d^2-4\right) S(Z,X,Z,W)+\left(-2 d^4+11 d^2-4\right) S(Z,W,Z,W) \displaybreak[1]\\
0 & = \left(2 d^4-11 d^2+5\right) S(Y,g,Y,X)+\left(2 d^4-11 d^2+5\right) S(Y,W,Y,X)+\left(2 d^4-11 d^2+5\right) S(Y,X,Y,X) \displaybreak[1]\\
0 & = \left(2 d^4-11 d^2+5\right) S(Y,g,Y,W)+\left(2 d^4-11 d^2+5\right) S(Y,X,Y,W)+\left(2 d^4-11 d^2+5\right) S(Y,W,Y,W) \displaybreak[1]\\
0 & = \left(2 d^4-11 d^2+5\right) S(Y,W,Y,g)+\left(2 d^4-11 d^2+5\right) S(Y,X,Y,g)+\left(2 d^4-11 d^2+5\right) S(Y,g,Y,g) \displaybreak[1]\\
0 & = S(Z,X,1,X)-S(1,X,Z,X) \displaybreak[1]\\
0 & = S(Y,X,1,X)-S(1,X,Y,X) \displaybreak[1]\\
0 & = S(1,X,Z,X)-S(Z,X,1,X) \displaybreak[1]\\
0 & = \left(d^4-5 d^2+1\right) S(Z,W,Z,X)-S(Z,X,Z,W) \displaybreak[1]\\
0 & = \left(-d^4+6 d^2-4\right) S(Z,X,Z,W)-S(Z,W,Z,X) \displaybreak[1]\\
0 & = S(Y,X,Z,X)-S(Z,X,Y,X) \displaybreak[1]\\
0 & = \left(d^4-5 d^2+1\right) S(Y,W,Z,X)-S(Z,X,Y,W) \displaybreak[1]\\
0 & = \left(-d^4+6 d^2-4\right) S(Y,X,Z,W)-S(Z,W,Y,X) \displaybreak[1]\\
0 & = S(Y,W,Z,W)-S(Z,W,Y,W) \displaybreak[1]\\
0 & = S(1,X,Y,X)-S(Y,X,1,X) \displaybreak[1]\\
0 & = S(Z,X,Y,X)-S(Y,X,Z,X) \displaybreak[1]\\
0 & = \left(d^4-5 d^2+1\right) S(Z,W,Y,X)-S(Y,X,Z,W) \displaybreak[1]\\
0 & = \left(-d^4+6 d^2-4\right) S(Z,X,Y,W)-S(Y,W,Z,X) \displaybreak[1]\\
0 & = S(Z,W,Y,W)-S(Y,W,Z,W) \displaybreak[1]\\
0 & = \left(d^4-5 d^2+1\right) S(Y,W,Y,X)-S(Y,X,Y,W) \displaybreak[1]\\
0 & = \left(d^4-5 d^2+2\right) S(Y,g,Y,X)-S(Y,X,Y,g) \displaybreak[1]\\
0 & = \left(-d^4+6 d^2-4\right) S(Y,X,Y,W)-S(Y,W,Y,X) \displaybreak[1]\\
0 & = \left(-d^4+6 d^2-3\right) S(Y,g,Y,W)-S(Y,W,Y,g) \displaybreak[1]\\
0 & = \left(2 d^4-11 d^2+5\right) S(Y,X,Y,g)-S(Y,g,Y,X) \displaybreak[1]\\
0 & = \left(-2 d^4+11 d^2-4\right) S(Y,W,Y,g)-S(Y,g,Y,W) \displaybreak[1]\\
\end{align*}
Solving these linear equations gives the desired answer.
 \end{proof}
 
 \begin{thm}
 \label{thm:flat-generators-A2}
 Inside this subspace, the flat elements with respect to the connection $K^{(2)}_{\text{lopsided}}$ are one-dimensional, spanned by the element $T$ given in \S \ref{appendix:flat-generators-A}.
 \end{thm}
 \begin{proof}
 An element $x \in \cG(\Gamma)_{n,+}$ in a graph planar algebra is flat with respect to a connection $K$ exactly if there exists an element $y \in \cG(\Gamma')_{n,-}$ satisfying
\begin{equation*}
\xabove = \ybelow.
\end{equation*}
 It is easy to see that if $x$ is a lowest weight eigenvector, $y$ must be also, with the same eigenvalue. In our case, the principal and dual principal graphs are the same, so Theorem \ref{thm:lws-A} also suffices to describe the lowest weight eigenspace on $\Gamma'$. Thus we take $x = \sum_{i=1}^4 c_{0,i} S_i$ and $y = \sum_{i=1}^4 c_{1,i} \rho^{1/2}(S_i)$ (the `half-click' rotation $\rho^{1/2}$ is necessary here since $S_i$ itself lives in $\cG(\Gamma)_{n,+} \iso \cG(\Gamma')_{n,+}$, and we need a basis for the lowest weight eigenvectors in $\cG(\Gamma')_{n,-}$). The flatness condition reduces to the 129 equations listed in \S \ref{appendix:flat-generators-A}. Solving these linear equations, we find a one dimensional space of solutions spanned by
\begin{align*}
c_{0,1} & = 1\\
c_{0,2} & = -3 d^4+17 d^2-10\\ 
c_{0,3} & = 2 d^4-11 d^2+6\\ 
c_{0,4} & = -2 d^4+12 d^2-7\\ 
c_{1,1} & = d^4-5 d^2+2\\ 
c_{1,2} & = -2 d^4+11 d^2-6\\ 
c_{1,3} & = d^4-5 d^2+3\\ 
c_{1,4} & = -d^4+7 d^2-4
\end{align*}
Substituting these into the formula for $x$ above, we get the element $T$ claimed in the statement.
 \end{proof}
 
 Remember that there is a whole gauge orbit of flat connections equivalent to $K^{(2)}_{\text{lopsided}}$. As we change the connection, the corresponding flat elements change as described in Lemma \ref{lem:flat-elements-transform}. This gauge group action gives rise to a 1-parameter family of embeddings of the $\cA$ planar algebra in its graph planar algebra; the formulas above give just one point on this curve. In previous cases where we've explicitly identified the embedding in the graph planar algebra \cite{MR2679382, 0909.4099} there's just been a discrete set of embeddings. As noted previously, the gauge group action on flat elements is always trivial for principal graphs with no loops. One sees a related phenomenon in solving the equation $S^2 = (1-r) S + r f^{(n)}$ (here $r$ is the ratio of dimensions past an initial branch point, $n$ the depth of those vertices) which must be satisfied by the lowest weight $n$-box in a $(n-1)$-supertransitive excess 1 planar algebra. For $\cA$, there is a one-parameter family of solutions, exactly agreeing with the gauge orbit of the flat element described above, while in previously studied examples there had been a discrete set of solutions.

 \begin{thm}
 \label{thm:flat-generators-A1}
 Inside the subspace described in Theorem \ref{thm:lws-A}, there are no flat elements with respect to the connection $K^{(1)}_{\text{lopsided}}$.
 \end{thm}
 \begin{proof}
 This is essentially identical to the calculation described in Theorem \ref{thm:flat-generators-A2}
 \end{proof}
 
 Thus, we've proved one of our main theorems:

\begin{thm}
\label{thm:flat-generators-B}
The lowest weight $\mu$-eigenspaces in $\cG(\Gamma(\cB))_{3,+}$ are each 7 dimensional, for $\mu= 1, \omega$ or $\omega^{-1}$. There is a one-dimensional space of flat elements with respect to the connection $K^B_{\text{lopsided}}$ in the $1$-eigenspace, and no flat elements in the other eigenspaces. Those flat elements are spanned by the element $T$ given in \S \ref{appendix:flat-generators-B}.
\end{thm}
\begin{proof}
Exactly analogous to the calculations above. We omit writing down the bases for the lowest weight eigenspaces, but they can be found explicitly calculated in the {\tt Mathematica} notebook {\tt connections-and-flat-elements.nb} available with the {\tt arXiv} sources of this article.
\end{proof}

\appendix
\section{Low weight vectors and flat elements}\label{appendix:flatlowweight}
\subsection{Low weight vectors for $\cA$}
\label{appendix:low-weight-A}
The low weight vectors in $\cG(\Gamma(\cA))_{2,+}$ with eigenvalue 1 form a four-dimensional space, with basis given below.
 \begin{align*}
(S_{1})_{1,1} & = \left(
\begin{array}{c}
 0
\end{array}
\right) \displaybreak[1]\\
(S_{1})_{1,Z} & = \left(
\begin{array}{c}
 d^4-5 d^2+2
\end{array}
\right) \displaybreak[1]\\
(S_{1})_{1,Y} & = \left(
\begin{array}{c}
 d^4-6 d^2+3
\end{array}
\right) \displaybreak[1]\\
(S_{1})_{Z,1} & = \left(
\begin{array}{c}
 d^4-5 d^2+2
\end{array}
\right) \displaybreak[1]\\
(S_{1})_{Z,Z} & = \left(
\begin{array}{cc}
 0 & 0 \\
 0 & 0
\end{array}
\right) \displaybreak[1]\\
(S_{1})_{Z,Y} & = \left(
\begin{array}{cc}
 -1 & 0 \\
 0 & 0
\end{array}
\right) \displaybreak[1]\\
(S_{1})_{Y,1} & = \left(
\begin{array}{c}
 d^4-6 d^2+3
\end{array}
\right) \displaybreak[1]\\
(S_{1})_{Y,Z} & = \left(
\begin{array}{cc}
 -1 & 0 \\
 0 & 0
\end{array}
\right) \displaybreak[1]\\
(S_{1})_{Y,Y} & = \left(
\begin{array}{ccc}
 -2 d^4+12 d^2-8 & -1 & d^4-6 d^2+3 \\
 d^4-6 d^2+4 & 0 & -d^4+6 d^2-3 \\
 d^4-6 d^2+4 & 1 & 0
\end{array}
\right) \displaybreak[1]\\
\\(S_{2})_{1,1} & = \left(
\begin{array}{c}
 0
\end{array}
\right) \displaybreak[1]\\
(S_{2})_{1,Z} & = \left(
\begin{array}{c}
 0
\end{array}
\right) \displaybreak[1]\\
(S_{2})_{1,Y} & = \left(
\begin{array}{c}
 0
\end{array}
\right) \displaybreak[1]\\
(S_{2})_{Z,1} & = \left(
\begin{array}{c}
 0
\end{array}
\right) \displaybreak[1]\\
(S_{2})_{Z,Z} & = \left(
\begin{array}{cc}
 d^4-5 d^2+1 & 2 d^4-11 d^2+3 \\
 -d^4+5 d^2-1 & -2 d^4+11 d^2-3
\end{array}
\right) \displaybreak[1]\\
(S_{2})_{Z,Y} & = \left(
\begin{array}{cc}
 -1 & 0 \\
 0 & -d^4+5 d^2-1
\end{array}
\right) \displaybreak[1]\\
(S_{2})_{Y,1} & = \left(
\begin{array}{c}
 0
\end{array}
\right) \displaybreak[1]\\
(S_{2})_{Y,Z} & = \left(
\begin{array}{cc}
 -1 & 0 \\
 0 & -d^4+5 d^2-1
\end{array}
\right) \displaybreak[1]\\
(S_{2})_{Y,Y} & = \left(
\begin{array}{ccc}
 -d^4+6 d^2-4 & -1 & 0 \\
 d^4-6 d^2+4 & 1 & 0 \\
 0 & 0 & 0
\end{array}
\right) \displaybreak[1]\\
\\(S_{3})_{1,1} & = \left(
\begin{array}{c}
 0
\end{array}
\right) \displaybreak[1]\\
(S_{3})_{1,Z} & = \left(
\begin{array}{c}
 0
\end{array}
\right) \displaybreak[1]\\
(S_{3})_{1,Y} & = \left(
\begin{array}{c}
 0
\end{array}
\right) \displaybreak[1]\\
(S_{3})_{Z,1} & = \left(
\begin{array}{c}
 0
\end{array}
\right) \displaybreak[1]\\
(S_{3})_{Z,Z} & = \left(
\begin{array}{cc}
 0 & 0 \\
 0 & 0
\end{array}
\right) \displaybreak[1]\\
(S_{3})_{Z,Y} & = \left(
\begin{array}{cc}
 0 & 0 \\
 -d^4+6 d^2-4 & 0
\end{array}
\right) \displaybreak[1]\\
(S_{3})_{Y,1} & = \left(
\begin{array}{c}
 0
\end{array}
\right) \displaybreak[1]\\
(S_{3})_{Y,Z} & = \left(
\begin{array}{cc}
 0 & 1 \\
 0 & 0
\end{array}
\right) \displaybreak[1]\\
(S_{3})_{Y,Y} & = \left(
\begin{array}{ccc}
 0 & 0 & 0 \\
 0 & 0 & 0 \\
 0 & 0 & 0
\end{array}
\right) \displaybreak[1]\\
\\(S_{4})_{1,1} & = \left(
\begin{array}{c}
 0
\end{array}
\right) \displaybreak[1]\\
(S_{4})_{1,Z} & = \left(
\begin{array}{c}
 0
\end{array}
\right) \displaybreak[1]\\
(S_{4})_{1,Y} & = \left(
\begin{array}{c}
 0
\end{array}
\right) \displaybreak[1]\\
(S_{4})_{Z,1} & = \left(
\begin{array}{c}
 0
\end{array}
\right) \displaybreak[1]\\
(S_{4})_{Z,Z} & = \left(
\begin{array}{cc}
 0 & 0 \\
 0 & 0
\end{array}
\right) \displaybreak[1]\\
(S_{4})_{Z,Y} & = \left(
\begin{array}{cc}
 0 & d^4-5 d^2+1 \\
 0 & 0
\end{array}
\right) \displaybreak[1]\\
(S_{4})_{Y,1} & = \left(
\begin{array}{c}
 0
\end{array}
\right) \displaybreak[1]\\
(S_{4})_{Y,Z} & = \left(
\begin{array}{cc}
 0 & 0 \\
 1 & 0
\end{array}
\right) \displaybreak[1]\\
(S_{4})_{Y,Y} & = \left(
\begin{array}{ccc}
 0 & 0 & 0 \\
 0 & 0 & 0 \\
 0 & 0 & 0
\end{array}
\right) \displaybreak[1]\\
\\
\end{align*}

\subsection{Flat generators for $\cA$}
\label{appendix:flat-generators-A}
Only the second connection on $\Gamma(\cA)$ has any flat lowest weight eigenvectors with eigenvalue 1, and it has a one dimensional space of such, spanned by the element $T$ specified below.
 \begin{align*}
T_{1,1} & = \left(
\begin{array}{c}
 0
\end{array}
\right) \displaybreak[1]\\
T_{1,Z} & = \left(
\begin{array}{c}
 6 d^4-32 d^2+9
\end{array}
\right) \displaybreak[1]\\
T_{1,Y} & = \left(
\begin{array}{c}
 3 d^4-16 d^2+4
\end{array}
\right) \displaybreak[1]\\
T_{Z,1} & = \left(
\begin{array}{c}
 6 d^4-32 d^2+9
\end{array}
\right) \displaybreak[1]\\
T_{Z,Z} & = \left(
\begin{array}{cc}
 -5 d^4+27 d^2-9 & -8 d^4+43 d^2-13 \\
 5 d^4-27 d^2+9 & 8 d^4-43 d^2+13
\end{array}
\right) \displaybreak[1]\\
T_{Z,Y} & = \left(
\begin{array}{cc}
 4 d^4-21 d^2+4 & -2 d^4+11 d^2-4 \\
 d^4-6 d^2+5 & 5 d^4-27 d^2+9
\end{array}
\right) \displaybreak[1]\\
T_{Y,1} & = \left(
\begin{array}{c}
 3 d^4-16 d^2+4
\end{array}
\right) \displaybreak[1]\\
T_{Y,Z} & = \left(
\begin{array}{cc}
 4 d^4-21 d^2+4 & d^4-5 d^2 \\
 2 d^4-11 d^2+5 & 5 d^4-27 d^2+9
\end{array}
\right) \displaybreak[1]\\
T_{Y,Y} & = \left(
\begin{array}{ccc}
 5 d^4-26 d^2+4 & 4 d^4-21 d^2+4 & 3 d^4-16 d^2+4 \\
 5-d^2 & 4 d^4-22 d^2+9 & -3 d^4+16 d^2-4 \\
 -5 d^4+27 d^2-9 & -8 d^4+43 d^2-13 & 0
\end{array}
\right) \displaybreak[1]\\
 \end{align*}

This element was found by solving the linear equations described below, equivalent to the flatness of a linear combination $\sum_{i=1}^4 c_{1,i} S_i$ of the lowest weight eigenvectors specified in the previous section.

We denote by $\cF$ the element of the two-sided graph planar algebra
\begin{equation*}
\xabove
-
\ybelow
\end{equation*}
(Here $x$ and $y$ are as described in Theorem  \ref{thm:flat-generators-A2}.)
Each of its coefficients, which are indexed by a loop on the 4-partite principal graph (reading around the boundary clockwise from the left, first an $N-N$ bimodule, then an $M-N$, $M-M$, $M-N$, $N-N$ then $N-M$ bimodule), and are given below, must be zero.

\begin{align*}
\cF(1\bar{X}\hat{1}\bar{X}1X) & = 0 \displaybreak[1] \\ 
\cF(1\bar{X}\hat{1}\bar{X}ZX) & = \left(d^4-5 d^2+2\right) c_{0,1}-c_{1,1} \displaybreak[1] \\ 
\cF(1\bar{X}\hat{1}\bar{X}YX) & = \left(d^4-6 d^2+3\right) c_{0,1}+\left(-d^4+6 d^2-4\right) c_{1,1} \displaybreak[1] \\ 
\cF(1\bar{X}\hat{Z}\bar{X}1X) & = 0 \displaybreak[1] \\ 
\cF(1\bar{X}\hat{Z}\bar{X}ZX) & = \left(-d^4+6 d^2-3\right) c_{0,1}+\left(-2 d^4+11 d^2-6\right) c_{1,1}-c_{1,2} \displaybreak[1] \\ 
\cF(1\bar{X}\hat{Z}\bar{X}YX) & = \left(-2 d^4+12 d^2-7\right) c_{0,1}+\left(-3 d^4+17 d^2-10\right) c_{1,1}+\left(-d^4+6 d^2-4\right) c_{1,2} \displaybreak[1] \\ 
\cF(1\bar{X}\hat{Z}\bar{W}ZX) & = \left(d^2-1\right) c_{0,1}+\left(d^4-6 d^2+4\right) c_{1,3}+c_{1,2} \displaybreak[1] \\ 
\cF(1\bar{X}\hat{Z}\bar{W}YX) & = \left(-d^4+5 d^2-3\right) c_{0,1}+\left(2 d^4-11 d^2+6\right) c_{1,3}+c_{1,2} \displaybreak[1] \\ 
\cF(1\bar{X}\hat{Y}\bar{X}1X) & = 0 \displaybreak[1] \\ 
\cF(1\bar{X}\hat{Y}\bar{X}ZX) & = \left(1-d^2\right) c_{0,1}+\left(2 d^4-11 d^2+7\right) c_{1,1}+c_{1,2} \displaybreak[1] \\ 
\cF(1\bar{X}\hat{Y}\bar{X}YX) & = \left(d^4-6 d^2+4\right) c_{0,1}+\left(4 d^4-23 d^2+14\right) c_{1,1}+\left(d^4-6 d^2+4\right) c_{1,2} \displaybreak[1] \\ 
\cF(1\bar{X}\hat{Y}\bar{W}ZX) & = d^2 c_{0,1}+c_{1,1}+c_{1,2} - c_{1,4} \displaybreak[1] \\ 
\cF(1\bar{X}\hat{Y}\bar{W}YX) & = \left(d^4-4 d^2+2\right) c_{0,1}+\left(d^4-6 d^2+3\right) c_{1,1}+\left(d^4-6 d^2+3\right) c_{1,2} -c_{1,4} \displaybreak[1] \\ 
\cF(1\bar{X}\hat{Y}\bar{g}YX) & = \left(1-d^2\right) c_{0,1}+\left(-d^4+6 d^2-3\right) c_{1,1} \displaybreak[1] \\ 
\cF(Z\bar{X}\hat{1}\bar{X}1X) & = \left(d^4-5 d^2+2\right) c_{0,1}-c_{1,1} \displaybreak[1] \\ 
\cF(Z\bar{X}\hat{1}\bar{X}ZX) & = \left(d^4-5 d^2+1\right) c_{0,2}+\left(-3 d^4+17 d^2-9\right) c_{1,1} \displaybreak[1] \\ 
\cF(Z\bar{X}\hat{1}\bar{X}ZW) & = \left(2 d^4-11 d^2+3\right) c_{0,2}+\left(-2 d^4+11 d^2-5\right) c_{1,1} \displaybreak[1] \\ 
\cF(Z\bar{X}\hat{1}\bar{X}YX) & = \left(-4 d^4+23 d^2-14\right) c_{1,1}-c_{0,1}-c_{0,2} \displaybreak[1] \\ 
\cF(Z\bar{X}\hat{1}\bar{X}YW) & = \left(d^4-5 d^2+1\right) c_{0,4}+\left(-2 d^4+11 d^2-6\right) c_{1,1} \displaybreak[1] \\ 
\cF(Z\bar{X}\hat{Z}\bar{X}1X) & = \left(-d^4+6 d^2-3\right) c_{0,1}+\left(-2 d^4+11 d^2-6\right) c_{1,1}-c_{1,2} \displaybreak[1] \\ 
\cF(Z\bar{X}\hat{Z}\bar{X}ZX) & = \left(d^4-6 d^2+3\right) c_{0,2}+\left(-6 d^4+34 d^2-19\right) c_{1,1}+\left(-3 d^4+17 d^2-9\right) c_{1,2} \displaybreak[1] \\ 
\cF(Z\bar{X}\hat{Z}\bar{X}ZW) & = \left(d^4-5 d^2+2\right) c_{0,2}+\left(-3 d^4+17 d^2-9\right) c_{1,1}+\left(-2 d^4+11 d^2-5\right) c_{1,2} \displaybreak[1] \\ 
\cF(Z\bar{X}\hat{Z}\bar{X}YX) & = \left(3 d^4-17 d^2+10\right) c_{0,1}+\left(3 d^4-17 d^2+10\right) c_{0,2}+\left(2 d^4-11 d^2+6\right) c_{0,3}\\ & \quad +\left(-9 d^4+51 d^2-29\right) c_{1,1}+\left(-4 d^4+23 d^2-14\right) c_{1,2} \displaybreak[1] \\ 
\cF(Z\bar{X}\hat{Z}\bar{X}YW) & = \left(-d^4+5 d^2-2\right) c_{0,2}+\left(d^4-6 d^2+4\right) c_{0,4}+\left(-3 d^4+17 d^2-10\right) c_{1,1}\\ & \quad +\left(-2 d^4+11 d^2-6\right) c_{1,2} \displaybreak[1] \\ 
\cF(Z\bar{X}\hat{Z}\bar{W}ZX) & = \left(d^4-5 d^2+2\right) c_{0,2}+\left(2 d^4-11 d^2+5\right) c_{1,2}+\left(3 d^4-17 d^2+10\right) c_{1,3} \displaybreak[1] \\ 
\cF(Z\bar{X}\hat{Z}\bar{W}ZW) & = \left(d^4-6 d^2+2\right) c_{0,2}+\left(2 d^4-11 d^2+4\right) c_{1,2}+\left(d^4-6 d^2+4\right) c_{1,3} \displaybreak[1] \\ 
\cF(Z\bar{X}\hat{Z}\bar{W}YX) & = \left(d^4-6 d^2+4\right) c_{0,1}+\left(d^4-6 d^2+4\right) c_{0,2}+\left(d^4-5 d^2+3\right) c_{0,3}\\ & \quad +\left(2 d^4-11 d^2+5\right) c_{1,2}+\left(4 d^4-23 d^2+13\right) c_{1,3} \displaybreak[1] \\ 
\cF(Z\bar{X}\hat{Z}\bar{W}YW) & = d^2 \left(-c_{0,2}\right)+\left(d^4-5 d^2+2\right) c_{1,2}+c_{0,4}\\ & \quad +\left(d^4-6 d^2+4\right) c_{1,3} \displaybreak[1] \\ 
\cF(Z\bar{X}\hat{Y}\bar{X}1X) & = \left(1-d^2\right) c_{0,1}+\left(2 d^4-11 d^2+7\right) c_{1,1}+c_{1,2} \displaybreak[1] \\ 
\cF(Z\bar{X}\hat{Y}\bar{X}ZX) & = \left(-2 d^4+11 d^2-4\right) c_{0,2}+\left(9 d^4-51 d^2+28\right) c_{1,1}+\left(3 d^4-17 d^2+9\right) c_{1,2} \displaybreak[1] \\ 
\cF(Z\bar{X}\hat{Y}\bar{X}ZW) & = \left(-3 d^4+16 d^2-5\right) c_{0,2}+\left(5 d^4-28 d^2+14\right) c_{1,1}+\left(2 d^4-11 d^2+5\right) c_{1,2} \displaybreak[1] \\ 
\cF(Z\bar{X}\hat{Y}\bar{X}YX) & = \left(-3 d^4+17 d^2-9\right) c_{0,1}+\left(-3 d^4+17 d^2-9\right) c_{0,2}+\left(-2 d^4+11 d^2-6\right) c_{0,3}\\ & \quad +\left(13 d^4-74 d^2+43\right) c_{1,1}+\left(4 d^4-23 d^2+14\right) c_{1,2} \displaybreak[1] \\ 
\cF(Z\bar{X}\hat{Y}\bar{X}YW) & = \left(d^4-5 d^2+2\right) c_{0,2}+\left(-2 d^4+11 d^2-5\right) c_{0,4}+\left(5 d^4-28 d^2+16\right) c_{1,1}\\ & \quad +\left(2 d^4-11 d^2+6\right) c_{1,2} \displaybreak[1] \\ 
\cF(Z\bar{X}\hat{Y}\bar{W}ZX) & = d^2 \left(-c_{0,2}\right)+\left(d^4-6 d^2+4\right) c_{1,1}+\left(d^4-6 d^2+4\right) c_{1,2}\\ & \quad +\left(-2 d^4+11 d^2-5\right) c_{1,4} \displaybreak[1] \\ 
\cF(Z\bar{X}\hat{Y}\bar{W}ZW) & = \left(d^4-4 d^2+1\right) c_{0,2}+c_{1,1}+c_{1,2}\\ & \quad +\left(-2 d^4+11 d^2-4\right) c_{1,4} \displaybreak[1] \\ 
\cF(Z\bar{X}\hat{Y}\bar{W}YX) & = \left(-d^4+5 d^2-3\right) c_{0,1}+\left(-d^4+5 d^2-3\right) c_{0,2}+\left(-d^4+6 d^2-3\right) c_{0,3}\\ & \quad +\left(2 d^4-11 d^2+6\right) c_{1,1}+\left(2 d^4-11 d^2+6\right) c_{1,2}+\left(-2 d^4+11 d^2-5\right) c_{1,4} \displaybreak[1] \\ 
\cF(Z\bar{X}\hat{Y}\bar{W}YW) & = \left(d^2-1\right) c_{0,4}+\left(d^4-6 d^2+2\right) c_{0,2}+c_{1,1}\\ & \quad +\left(-d^4+5 d^2-2\right) c_{1,4}+c_{1,2} \displaybreak[1] \\ 
\cF(Z\bar{X}\hat{Y}\bar{g}YX) & = \left(-d^4+6 d^2-3\right) c_{0,1}+\left(-d^4+6 d^2-3\right) c_{0,2}+\left(-d^4+6 d^2-3\right) c_{0,3}\\ & \quad +\left(-2 d^4+11 d^2-6\right) c_{1,1} \displaybreak[1] \\ 
\cF(Z\bar{X}\hat{Y}\bar{g}YW) & = \left(d^4-6 d^2+2\right) c_{0,2}+\left(-d^4+5 d^2-2\right) c_{0,4}+\left(-d^4+6 d^2-3\right) c_{1,1} \displaybreak[1] \\ 
\cF(Z\bar{W}\hat{Z}\bar{X}1X) & = \left(d^4-5 d^2+3\right) c_{0,1}+\left(-d^4+6 d^2-4\right) c_{1,2}+\left(3 d^4-17 d^2+9\right) c_{1,4} \displaybreak[1] \\ 
\cF(Z\bar{W}\hat{Z}\bar{X}ZX) & = \left(-d^4+6 d^2-3\right) c_{0,2}+\left(-3 d^4+17 d^2-9\right) c_{1,2}+\left(6 d^4-34 d^2+19\right) c_{1,4} \displaybreak[1] \\ 
\cF(Z\bar{W}\hat{Z}\bar{X}ZW) & = \left(-d^4+5 d^2-2\right) c_{0,2}+\left(-2 d^4+11 d^2-5\right) c_{1,2}+\left(3 d^4-17 d^2+9\right) c_{1,4} \displaybreak[1] \\ 
\cF(Z\bar{W}\hat{Z}\bar{X}YX) & = \left(-4 d^4+23 d^2-13\right) c_{0,1}+\left(-4 d^4+23 d^2-13\right) c_{0,2}+\left(-3 d^4+17 d^2-10\right) c_{0,3}\\ & \quad +\left(-3 d^4+17 d^2-10\right) c_{1,2}+\left(9 d^4-51 d^2+28\right) c_{1,4} \displaybreak[1] \\ 
\cF(Z\bar{W}\hat{Z}\bar{X}YW) & = \left(-2 d^4+11 d^2-6\right) c_{0,4}+\left(-2 d^4+11 d^2-5\right) c_{1,2}+c_{0,2}\\ & \quad +\left(3 d^4-17 d^2+10\right) c_{1,4} \displaybreak[1] \\ 
\cF(Z\bar{W}\hat{Z}\bar{W}ZX) & = \left(d^4-5 d^2+1\right) c_{0,2}+\left(2 d^4-11 d^2+4\right) c_{1,2} \displaybreak[1] \\ 
\cF(Z\bar{W}\hat{Z}\bar{W}ZW) & = \left(2 d^4-11 d^2+3\right) c_{0,2}+\left(3 d^4-16 d^2+5\right) c_{1,2} \displaybreak[1] \\ 
\cF(Z\bar{W}\hat{Z}\bar{W}YX) & = \left(-2 d^4+11 d^2-6\right) c_{0,1}+\left(-2 d^4+11 d^2-6\right) c_{0,2}+\left(-2 d^4+11 d^2-6\right) c_{0,3}\\ & \quad +\left(d^4-6 d^2+3\right) c_{1,2} \displaybreak[1] \\ 
\cF(Z\bar{W}\hat{Z}\bar{W}YW) & = \left(d^4-5 d^2+2\right) c_{0,2}+\left(-d^4+6 d^2-3\right) c_{0,4}+\left(d^4-5 d^2+2\right) c_{1,2} \displaybreak[1] \\ 
\cF(Z\bar{W}\hat{Y}\bar{X}1X) & = \left(-d^4+6 d^2-3\right) c_{0,1}+\left(-3 d^4+17 d^2-9\right) c_{1,1}+\left(-3 d^4+17 d^2-9\right) c_{1,2}\\ & \quad +\left(-3 d^4+17 d^2-10\right) c_{1,3} \displaybreak[1] \\ 
\cF(Z\bar{W}\hat{Y}\bar{X}ZX) & = \left(d^4-6 d^2+3\right) c_{0,2}+\left(-6 d^4+34 d^2-19\right) c_{1,1}+\left(-6 d^4+34 d^2-19\right) c_{1,2}\\ & \quad +\left(-6 d^4+34 d^2-19\right) c_{1,3} \displaybreak[1] \\ 
\cF(Z\bar{W}\hat{Y}\bar{X}ZW) & = \left(d^4-5 d^2+2\right) c_{0,2}+\left(-3 d^4+17 d^2-9\right) c_{1,1}+\left(-3 d^4+17 d^2-9\right) c_{1,2}\\ & \quad +\left(-3 d^4+17 d^2-9\right) c_{1,3} \displaybreak[1] \\ 
\cF(Z\bar{W}\hat{Y}\bar{X}YX) & = \left(4 d^4-23 d^2+14\right) c_{0,1}+\left(4 d^4-23 d^2+14\right) c_{0,2}+\left(4 d^4-23 d^2+13\right) c_{0,3}\\ & \quad +\left(-9 d^4+51 d^2-28\right) c_{1,1}+\left(-9 d^4+51 d^2-28\right) c_{1,2}+\left(-7 d^4+40 d^2-23\right) c_{1,3} \displaybreak[1] \\ 
\cF(Z\bar{W}\hat{Y}\bar{X}YW) & = \left(-d^4+6 d^2-3\right) c_{0,2}+\left(d^4-6 d^2+5\right) c_{0,4}+\left(-3 d^4+17 d^2-10\right) c_{1,1}\\ & \quad +\left(-3 d^4+17 d^2-10\right) c_{1,2}+\left(-3 d^4+17 d^2-9\right) c_{1,3} \displaybreak[1] \\ 
\cF(Z\bar{W}\hat{Y}\bar{W}ZX) & = \left(-d^4+5 d^2-1\right) c_{0,2}+\left(-2 d^4+11 d^2-4\right) c_{1,2} \displaybreak[1] \\ 
\cF(Z\bar{W}\hat{Y}\bar{W}ZW) & = \left(-2 d^4+11 d^2-3\right) c_{0,2}+\left(-3 d^4+16 d^2-5\right) c_{1,2} \displaybreak[1] \\ 
\cF(Z\bar{W}\hat{Y}\bar{W}YX) & = \left(2 d^4-11 d^2+6\right) c_{0,1}+\left(2 d^4-11 d^2+6\right) c_{0,2}+\left(2 d^4-11 d^2+6\right) c_{0,3}\\ & \quad +\left(-d^4+6 d^2-3\right) c_{1,2} \displaybreak[1] \\ 
\cF(Z\bar{W}\hat{Y}\bar{W}YW) & = \left(-d^4+5 d^2-2\right) c_{0,2}+\left(d^4-6 d^2+3\right) c_{0,4}+\left(-d^4+5 d^2-2\right) c_{1,2} \displaybreak[1] \\ 
\cF(Z\bar{W}\hat{Y}\bar{g}YX) & = \left(d^4-6 d^2+4\right) c_{0,1}+\left(d^4-6 d^2+4\right) c_{0,2}+\left(2 d^4-11 d^2+6\right) c_{0,3}\\ & \quad +\left(d^4-6 d^2+4\right) c_{1,1} \displaybreak[1] \\ 
\cF(Z\bar{W}\hat{Y}\bar{g}YW) & = \left(-d^4+5 d^2-2\right) c_{0,2}+\left(d^4-6 d^2+3\right) c_{1,1}+c_{0,4} \displaybreak[1] \\ 
\cF(Y\bar{X}\hat{1}\bar{X}1X) & = \left(d^4-6 d^2+3\right) c_{0,1}+\left(-d^4+6 d^2-4\right) c_{1,1} \displaybreak[1] \\ 
\cF(Y\bar{X}\hat{1}\bar{X}ZX) & = \left(-4 d^4+23 d^2-14\right) c_{1,1}-c_{0,1}-c_{0,2} \displaybreak[1] \\ 
\cF(Y\bar{X}\hat{1}\bar{X}ZW) & = \left(-3 d^4+17 d^2-10\right) c_{1,1}+c_{0,3} \displaybreak[1] \\ 
\cF(Y\bar{X}\hat{1}\bar{X}YX) & = \left(-2 d^4+12 d^2-8\right) c_{0,1}+\left(-d^4+6 d^2-4\right) c_{0,2}+\left(-7 d^4+40 d^2-24\right) c_{1,1} \displaybreak[1] \\ 
\cF(Y\bar{X}\hat{1}\bar{X}YW) & = \left(-4 d^4+23 d^2-14\right) c_{1,1}-c_{0,1}-c_{0,2} \displaybreak[1] \\ 
\cF(Y\bar{X}\hat{1}\bar{X}Yg) & = \left(d^4-6 d^2+3\right) c_{0,1}+\left(-d^4+6 d^2-4\right) c_{1,1} \displaybreak[1] \\ 
\cF(Y\bar{X}\hat{Z}\bar{X}1X) & = \left(-2 d^4+12 d^2-7\right) c_{0,1}+\left(-3 d^4+17 d^2-10\right) c_{1,1}+\left(-d^4+6 d^2-4\right) c_{1,2} \displaybreak[1] \\ 
\cF(Y\bar{X}\hat{Z}\bar{X}ZX) & = \left(3 d^4-17 d^2+10\right) c_{0,1}+\left(3 d^4-17 d^2+10\right) c_{0,2}+\left(-d^4+6 d^2-4\right) c_{0,4}\\ & \quad +\left(-9 d^4+51 d^2-29\right) c_{1,1}+\left(-4 d^4+23 d^2-14\right) c_{1,2} \displaybreak[1] \\ 
\cF(Y\bar{X}\hat{Z}\bar{X}ZW) & = \left(-3 d^4+17 d^2-10\right) c_{0,3}+\left(-6 d^4+34 d^2-19\right) c_{1,1}-c_{0,2}\\ & \quad +\left(-3 d^4+17 d^2-10\right) c_{1,2} \displaybreak[1] \\ 
\cF(Y\bar{X}\hat{Z}\bar{X}YX) & = \left(15 d^4-85 d^2+48\right) c_{0,1}+\left(6 d^4-34 d^2+19\right) c_{0,2}+\left(-13 d^4+74 d^2-43\right) c_{1,1}\\ & \quad +\left(-7 d^4+40 d^2-24\right) c_{1,2} \displaybreak[1] \\ 
\cF(Y\bar{X}\hat{Z}\bar{X}YW) & = \left(4 d^4-23 d^2+13\right) c_{0,1}+\left(3 d^4-17 d^2+9\right) c_{0,2}+\left(-7 d^4+40 d^2-23\right) c_{1,1}\\ & \quad +\left(-4 d^4+23 d^2-14\right) c_{1,2} \displaybreak[1] \\ 
\cF(Y\bar{X}\hat{Z}\bar{X}Yg) & = \left(-3 d^4+17 d^2-10\right) c_{0,1}+\left(-d^4+6 d^2-4\right) c_{1,1}+\left(-d^4+6 d^2-4\right) c_{1,2} \displaybreak[1] \\ 
\cF(Y\bar{X}\hat{Z}\bar{W}ZX) & = \left(2 d^4-11 d^2+6\right) c_{0,1}+\left(2 d^4-11 d^2+6\right) c_{0,2}+\left(-2 d^4+11 d^2-5\right) c_{0,4}\\ & \quad +\left(d^4-6 d^2+4\right) c_{1,2}+\left(4 d^4-23 d^2+14\right) c_{1,3} \displaybreak[1] \\ 
\cF(Y\bar{X}\hat{Z}\bar{W}ZW) & = \left(-2 d^4+11 d^2-4\right) c_{0,2}+\left(-2 d^4+11 d^2-6\right) c_{0,3}+\left(5 d^2-d^4\right) c_{1,2}\\ & \quad +\left(3 d^4-17 d^2+10\right) c_{1,3} \displaybreak[1] \\ 
\cF(Y\bar{X}\hat{Z}\bar{W}YX) & = \left(6 d^4-35 d^2+20\right) c_{0,1}+\left(2 d^4-12 d^2+7\right) c_{0,2}+\left(d^4-6 d^2+4\right) c_{1,2}\\ & \quad +\left(6 d^4-34 d^2+19\right) c_{1,3} \displaybreak[1] \\ 
\cF(Y\bar{X}\hat{Z}\bar{W}YW) & = \left(2 d^4-11 d^2+6\right) c_{0,1}+\left(d^4-5 d^2+3\right) c_{0,2}+c_{1,2}\\ & \quad +\left(3 d^4-17 d^2+10\right) c_{1,3} \displaybreak[1] \\ 
\cF(Y\bar{X}\hat{Z}\bar{W}Yg) & = \left(-d^4+7 d^2-4\right) c_{0,1}+\left(d^4-6 d^2+3\right) c_{1,3} \displaybreak[1] \\ 
\cF(Y\bar{X}\hat{Y}\bar{X}1X) & = \left(d^4-6 d^2+4\right) c_{0,1}+\left(4 d^4-23 d^2+14\right) c_{1,1}+\left(d^4-6 d^2+4\right) c_{1,2} \displaybreak[1] \\ 
\cF(Y\bar{X}\hat{Y}\bar{X}ZX) & = \left(-3 d^4+17 d^2-9\right) c_{0,1}+\left(-3 d^4+17 d^2-9\right) c_{0,2}+\left(d^4-6 d^2+4\right) c_{0,4}\\ & \quad +\left(13 d^4-74 d^2+43\right) c_{1,1}+\left(4 d^4-23 d^2+14\right) c_{1,2} \displaybreak[1] \\ 
\cF(Y\bar{X}\hat{Y}\bar{X}ZW) & = \left(3 d^4-17 d^2+9\right) c_{0,3}+\left(9 d^4-51 d^2+29\right) c_{1,1}+c_{0,2}\\ & \quad +\left(3 d^4-17 d^2+10\right) c_{1,2} \displaybreak[1] \\ 
\cF(Y\bar{X}\hat{Y}\bar{X}YX) & = \left(-13 d^4+73 d^2-40\right) c_{0,1}+\left(-5 d^4+28 d^2-15\right) c_{0,2}+\left(20 d^4-114 d^2+67\right) c_{1,1}\\ & \quad +\left(7 d^4-40 d^2+24\right) c_{1,2} \displaybreak[1] \\ 
\cF(Y\bar{X}\hat{Y}\bar{X}YW) & = \left(-4 d^4+23 d^2-12\right) c_{0,1}+\left(-3 d^4+17 d^2-8\right) c_{0,2}+\left(11 d^4-63 d^2+37\right) c_{1,1}\\ & \quad +\left(4 d^4-23 d^2+14\right) c_{1,2} \displaybreak[1] \\ 
\cF(Y\bar{X}\hat{Y}\bar{X}Yg) & = \left(2 d^4-11 d^2+7\right) c_{0,1}+\left(2 d^4-12 d^2+8\right) c_{1,1}+\left(d^4-6 d^2+4\right) c_{1,2} \displaybreak[1] \\ 
\cF(Y\bar{X}\hat{Y}\bar{W}ZX) & = \left(-2 d^4+11 d^2-5\right) c_{0,1}+\left(-2 d^4+11 d^2-5\right) c_{0,2}+\left(d^4-6 d^2+3\right) c_{0,4}\\ & \quad +\left(d^4-6 d^2+5\right) c_{1,1}+\left(d^4-6 d^2+5\right) c_{1,2}+\left(-d^4+6 d^2-4\right) c_{1,4} \displaybreak[1] \\ 
\cF(Y\bar{X}\hat{Y}\bar{W}ZW) & = \left(d^4-5 d^2+2\right) c_{0,2}+\left(2 d^4-11 d^2+5\right) c_{0,3}+\left(d^4-6 d^2+4\right) c_{1,1}\\ & \quad +\left(d^4-6 d^2+4\right) c_{1,2}+\left(d^4-5 d^2\right) c_{1,4} \displaybreak[1] \\ 
\cF(Y\bar{X}\hat{Y}\bar{W}YX) & = \left(-4 d^4+22 d^2-13\right) c_{0,1}+\left(-d^4+6 d^2-4\right) c_{0,2}+\left(3 d^4-17 d^2+9\right) c_{1,1}\\ & \quad +\left(3 d^4-17 d^2+9\right) c_{1,2}+\left(-d^4+6 d^2-4\right) c_{1,4} \displaybreak[1] \\ 
\cF(Y\bar{X}\hat{Y}\bar{W}YW) & = \left(-d^4+7 d^2-4\right) c_{0,1}+\left(d^4-6 d^2+4\right) c_{1,1}-c_{0,2}\\ & \quad +\left(d^4-6 d^2+4\right) c_{1,2}-c_{1,4} \displaybreak[1] \\ 
\cF(Y\bar{X}\hat{Y}\bar{W}Yg) & = \left(d^4-6 d^2+3\right) c_{0,1}+\left(d^4-5 d^2+2\right) c_{1,1}+\left(d^4-5 d^2+2\right) c_{1,2} \displaybreak[1] \\ 
\cF(Y\bar{X}\hat{Y}\bar{g}YX) & = \left(-3 d^4+17 d^2-9\right) c_{0,1}+\left(-d^4+6 d^2-3\right) c_{0,2}+\left(-3 d^4+17 d^2-9\right) c_{1,1} \displaybreak[1] \\ 
\cF(Y\bar{X}\hat{Y}\bar{g}YW) & = \left(-d^4+6 d^2-3\right) c_{0,1}+\left(-d^4+5 d^2-2\right) c_{0,2}+\left(-2 d^4+11 d^2-6\right) c_{1,1} \displaybreak[1] \\ 
\cF(Y\bar{X}\hat{Y}\bar{g}Yg) & = \left(1-d^2\right) c_{0,1}+\left(-d^4+6 d^2-3\right) c_{1,1} \displaybreak[1] \\ 
\cF(Y\bar{W}\hat{Z}\bar{X}1X) & = \left(4 d^4-23 d^2+13\right) c_{0,1}+\left(-3 d^4+17 d^2-9\right) c_{1,2}+\left(6 d^4-34 d^2+19\right) c_{1,4} \displaybreak[1] \\ 
\cF(Y\bar{W}\hat{Z}\bar{X}ZX) & = \left(-6 d^4+34 d^2-19\right) c_{0,1}+\left(-6 d^4+34 d^2-19\right) c_{0,2}+\left(3 d^4-17 d^2+10\right) c_{0,4}\\ & \quad +\left(-4 d^4+23 d^2-14\right) c_{1,2}+\left(13 d^4-74 d^2+42\right) c_{1,4} \displaybreak[1] \\ 
\cF(Y\bar{W}\hat{Z}\bar{X}ZW) & = \left(d^4-6 d^2+4\right) c_{0,2}+\left(6 d^4-34 d^2+19\right) c_{0,3}+\left(-3 d^4+17 d^2-9\right) c_{1,2}\\ & \quad +\left(9 d^4-51 d^2+28\right) c_{1,4} \displaybreak[1] \\ 
\cF(Y\bar{W}\hat{Z}\bar{X}YX) & = \left(-25 d^4+142 d^2-81\right) c_{0,1}+\left(-9 d^4+51 d^2-29\right) c_{0,2}+\left(-6 d^4+34 d^2-19\right) c_{1,2}\\ & \quad +\left(19 d^4-108 d^2+61\right) c_{1,4} \displaybreak[1] \\ 
\cF(Y\bar{W}\hat{Z}\bar{X}YW) & = \left(-7 d^4+40 d^2-23\right) c_{0,1}+\left(-4 d^4+23 d^2-13\right) c_{0,2}+\left(-3 d^4+17 d^2-9\right) c_{1,2}\\ & \quad +\left(10 d^4-57 d^2+33\right) c_{1,4} \displaybreak[1] \\ 
\cF(Y\bar{W}\hat{Z}\bar{X}Yg) & = \left(5 d^4-28 d^2+16\right) c_{0,1}+\left(3 d^4-17 d^2+9\right) c_{1,4} \displaybreak[1] \\ 
\cF(Y\bar{W}\hat{Z}\bar{W}ZX) & = \left(-3 d^4+17 d^2-10\right) c_{0,1}+\left(-3 d^4+17 d^2-10\right) c_{0,2}+\left(3 d^4-17 d^2+9\right) c_{0,4}\\ & \quad +\left(d^4-6 d^2+4\right) c_{1,2} \displaybreak[1] \\ 
\cF(Y\bar{W}\hat{Z}\bar{W}ZW) & = \left(2 d^4-11 d^2+5\right) c_{0,2}+\left(3 d^4-17 d^2+10\right) c_{0,3}+\left(2 d^4-11 d^2+5\right) c_{1,2} \displaybreak[1] \\ 
\cF(Y\bar{W}\hat{Z}\bar{W}YX) & = \left(-10 d^4+57 d^2-33\right) c_{0,1}+\left(-3 d^4+17 d^2-10\right) c_{0,2}+\left(d^4-6 d^2+5\right) c_{1,2} \displaybreak[1] \\ 
\cF(Y\bar{W}\hat{Z}\bar{W}YW) & = \left(-3 d^4+17 d^2-10\right) c_{0,1}+\left(-d^4+6 d^2-4\right) c_{0,2}+\left(d^4-6 d^2+4\right) c_{1,2} \displaybreak[1] \\ 
\cF(Y\bar{W}\hat{Z}\bar{W}Yg) & = \left(2 d^4-11 d^2+6\right) c_{0,1}+c_{1,2} \displaybreak[1] \\ 
\cF(Y\bar{W}\hat{Y}\bar{X}1X) & = \left(-4 d^4+23 d^2-13\right) c_{0,1}+\left(-6 d^4+34 d^2-19\right) c_{1,1}+\left(-6 d^4+34 d^2-19\right) c_{1,2}\\ & \quad +\left(-6 d^4+34 d^2-19\right) c_{1,3} \displaybreak[1] \\ 
\cF(Y\bar{W}\hat{Y}\bar{X}ZX) & = \left(7 d^4-40 d^2+23\right) c_{0,1}+\left(7 d^4-40 d^2+23\right) c_{0,2}+\left(-4 d^4+23 d^2-14\right) c_{0,4}\\ & \quad +\left(-13 d^4+74 d^2-42\right) c_{1,1}+\left(-13 d^4+74 d^2-42\right) c_{1,2}+\left(-10 d^4+57 d^2-33\right) c_{1,3} \displaybreak[1] \\ 
\cF(Y\bar{W}\hat{Y}\bar{X}ZW) & = \left(-d^4+6 d^2-5\right) c_{0,2}+\left(-7 d^4+40 d^2-23\right) c_{0,3}+\left(-9 d^4+51 d^2-28\right) c_{1,1}\\ & \quad +\left(-9 d^4+51 d^2-28\right) c_{1,2}+\left(-6 d^4+34 d^2-19\right) c_{1,3} \displaybreak[1] \\ 
\cF(Y\bar{W}\hat{Y}\bar{X}YX) & = \left(30 d^4-171 d^2+98\right) c_{0,1}+\left(10 d^4-57 d^2+33\right) c_{0,2}+\left(-19 d^4+108 d^2-61\right) c_{1,1}\\ & \quad +\left(-19 d^4+108 d^2-61\right) c_{1,2}+\left(-13 d^4+74 d^2-42\right) c_{1,3} \displaybreak[1] \\ 
\cF(Y\bar{W}\hat{Y}\bar{X}YW) & = \left(9 d^4-51 d^2+29\right) c_{0,1}+\left(4 d^4-23 d^2+14\right) c_{0,2}+\left(-10 d^4+57 d^2-33\right) c_{1,1}\\ & \quad +\left(-10 d^4+57 d^2-33\right) c_{1,2}+\left(-6 d^4+34 d^2-19\right) c_{1,3} \displaybreak[1] \\ 
\cF(Y\bar{W}\hat{Y}\bar{X}Yg) & = \left(-6 d^4+34 d^2-19\right) c_{0,1}+\left(-3 d^4+17 d^2-9\right) c_{1,1}+\left(-3 d^4+17 d^2-9\right) c_{1,2} \displaybreak[1] \\ 
\cF(Y\bar{W}\hat{Y}\bar{W}ZX) & = \left(3 d^4-17 d^2+10\right) c_{0,1}+\left(3 d^4-17 d^2+10\right) c_{0,2}+\left(-3 d^4+17 d^2-9\right) c_{0,4}\\ & \quad +\left(-d^4+6 d^2-4\right) c_{1,2} \displaybreak[1] \\ 
\cF(Y\bar{W}\hat{Y}\bar{W}ZW) & = \left(-2 d^4+11 d^2-5\right) c_{0,2}+\left(-3 d^4+17 d^2-10\right) c_{0,3}+\left(-2 d^4+11 d^2-5\right) c_{1,2} \displaybreak[1] \\ 
\cF(Y\bar{W}\hat{Y}\bar{W}YX) & = \left(10 d^4-57 d^2+33\right) c_{0,1}+\left(3 d^4-17 d^2+10\right) c_{0,2}+\left(-d^4+6 d^2-5\right) c_{1,2} \displaybreak[1] \\ 
\cF(Y\bar{W}\hat{Y}\bar{W}YW) & = \left(3 d^4-17 d^2+10\right) c_{0,1}+\left(d^4-6 d^2+4\right) c_{0,2}+\left(-d^4+6 d^2-4\right) c_{1,2} \displaybreak[1] \\ 
\cF(Y\bar{W}\hat{Y}\bar{W}Yg) & = \left(-2 d^4+11 d^2-6\right) c_{0,1}-c_{1,2} \displaybreak[1] \\ 
\cF(Y\bar{W}\hat{Y}\bar{g}YX) & = \left(5 d^4-29 d^2+17\right) c_{0,1}+\left(d^4-6 d^2+4\right) c_{0,2}+\left(3 d^4-17 d^2+10\right) c_{1,1} \displaybreak[1] \\ 
\cF(Y\bar{W}\hat{Y}\bar{g}YW) & = \left(2 d^4-11 d^2+6\right) c_{0,1}+\left(3 d^4-17 d^2+9\right) c_{1,1}+c_{0,2} \displaybreak[1] \\ 
\cF(Y\bar{W}\hat{Y}\bar{g}Yg) & = \left(-d^4+6 d^2-3\right) c_{0,1}+\left(d^4-6 d^2+4\right) c_{1,1} \displaybreak[1] \\ 
\cF(Y\bar{g}\hat{Y}\bar{X}1X) & = \left(d^4-6 d^2+3\right) c_{0,1}+\left(-d^4+6 d^2-4\right) c_{1,1} \displaybreak[1] \\ 
\cF(Y\bar{g}\hat{Y}\bar{X}ZX) & = \left(-d^4+6 d^2-4\right) c_{0,1}+\left(-d^4+6 d^2-4\right) c_{0,2}+\left(2 d^4-11 d^2+5\right) c_{0,4}\\ & \quad +\left(-3 d^4+17 d^2-10\right) c_{1,1} \displaybreak[1] \\ 
\cF(Y\bar{g}\hat{Y}\bar{X}ZW) & = \left(2 d^4-11 d^2+4\right) c_{0,2}+\left(d^4-6 d^2+4\right) c_{0,3}+\left(-3 d^4+17 d^2-9\right) c_{1,1} \displaybreak[1] \\ 
\cF(Y\bar{g}\hat{Y}\bar{X}YX) & = \left(-4 d^4+23 d^2-14\right) c_{0,1}+\left(-d^4+6 d^2-4\right) c_{0,2}+\left(-4 d^4+23 d^2-14\right) c_{1,1} \displaybreak[1] \\ 
\cF(Y\bar{g}\hat{Y}\bar{X}YW) & = \left(-d^4+6 d^2-4\right) c_{0,1}+\left(-3 d^4+17 d^2-10\right) c_{1,1}-c_{0,2} \displaybreak[1] \\ 
\cF(Y\bar{g}\hat{Y}\bar{X}Yg) & = \left(d^4-6 d^2+3\right) c_{0,1}+\left(-d^4+6 d^2-4\right) c_{1,1} \displaybreak[1] \\ 
\cF(Y\bar{g}\hat{Y}\bar{W}ZX) & = -c_{0,1}-c_{0,2}+c_{0,4} -c_{1,1} \displaybreak[1] \\ 
\cF(Y\bar{g}\hat{Y}\bar{W}ZW) & = \left(-d^4+5 d^2-1\right) c_{0,2}+c_{0,3}-c_{1,1} \displaybreak[1] \\ 
\cF(Y\bar{g}\hat{Y}\bar{W}YX) & = \left(-2 d^4+10 d^2-5\right) c_{0,1}+\left(-d^4+5 d^2-2\right) c_{0,2}+\left(-d^4+6 d^2-3\right) c_{1,1} \displaybreak[1] \\ 
\cF(Y\bar{g}\hat{Y}\bar{W}YW) & = \left(d^2-1\right) c_{0,1}+\left(-d^4+6 d^2-2\right) c_{0,2}-c_{1,1} \displaybreak[1] \\ 
\cF(Y\bar{g}\hat{Y}\bar{W}Yg) & = d^2 c_{0,1}+\left(-d^4+5 d^2-2\right) c_{1,1} \displaybreak[1] \\ 
\cF(Y\bar{g}\hat{Y}\bar{g}YX) & = 0 \displaybreak[1] \\ 
\cF(Y\bar{g}\hat{Y}\bar{g}YW) & = 0 \displaybreak[1] \\ 
\cF(Y\bar{g}\hat{Y}\bar{g}Yg) & = 0 \displaybreak[1] \\ 
\end{align*}

\subsection{Flat generators for $\cB$}
\label{appendix:flat-generators-B}
Only the low weight eigenspace with eigenvalue $1$ in $\cG(\Gamma(\cB))_{3,+}$ contains any flat vectors, and these are spanned by the element $T$ below. (As usual, $d^2$ is the index of $\cA$ and $\cB$, the largest real root of $\lambda^3-6\lambda^2+5\lambda -1=0$, approximately $5.04892$.)

{\small
\begin{align*}
T_{1,f} & = \left(
\begin{array}{cc}
 0 & 0 \\
 0 & 0
\end{array}
\right) \displaybreak[1]\\
T_{1,B} & = \left(
\begin{array}{c}
 -5 d^4+34 d^2-20
\end{array}
\right) \displaybreak[1]\\
T_{1,F} & = \left(
\begin{array}{c}
 2 d^4-15 d^2+8
\end{array}
\right) \displaybreak[1]\\
T_{A,f} & = \left(
\begin{array}{cccc}
 0 & 0 & 9 d^4-50 d^2+29 & -5 d^4+27 d^2-13 \\
 0 & 0 & -9 d^4+50 d^2-29 & 5 d^4-27 d^2+13 \\
 4 d^4-16 d^2+9 & 3 d^4-19 d^2+12 & d^4-4 d^2+4 & -d^4+4 d^2+3 \\
 -4 d^4+16 d^2-9 & -3 d^4+19 d^2-12 & 8 d^4-46 d^2+25 & -4 d^4+23 d^2-16
\end{array}
\right) \displaybreak[1]\\
T_{A,B} & = \left(
\begin{array}{cccc}
 -9 d^4+50 d^2-29 & -3 d^4+19 d^2-12 & -6 d^4+31 d^2-17 & -4 d^4+23 d^2-16 \\
 d^4-4 d^2+4 & 4 d^4-23 d^2+16 & 8 d^4-46 d^2+25 & -6 d^4+31 d^2-10 \\
 -d^4+4 d^2-4 & -4 d^4+23 d^2-16 & -8 d^4+46 d^2-25 & 6 d^4-31 d^2+10 \\
 8 d^4-46 d^2+25 & -d^4+4 d^2-4 & -2 d^4+15 d^2-8 & 10 d^4-54 d^2+26
\end{array}
\right) \displaybreak[1]\\
T_{A,F} & = \left(
\begin{array}{ccccc}
 5 d^4-27 d^2+13 & -4 d^4+23 d^2-16 & 9 d^4-50 d^2+22 & 6 d^4-31 d^2+17 & -d^4+4 d^2-4 \\
 -d^4+4 d^2+3 & -6 d^4+31 d^2-10 & -10 d^4+54 d^2-12 & -2 d^4+8 d^2-1 & -9 d^4+50 d^2-22 \\
 -4 d^4+23 d^2-16 & 10 d^4-54 d^2+26 & d^4-4 d^2-10 & -4 d^4+23 d^2-16 & 10 d^4-54 d^2+26 \\
 6 d^4-31 d^2+17 & -d^4+4 d^2-4 & 9 d^4-50 d^2+22 & 5 d^4-27 d^2+13 & -4 d^4+23 d^2-16 \\
 -2 d^4+8 d^2-1 & -9 d^4+50 d^2-22 & -10 d^4+54 d^2-12 & -d^4+4 d^2+3 & -6 d^4+31 d^2-10
\end{array}
\right) \displaybreak[1]\\
T_{A,z} & = \left(
\begin{array}{c}
 2 d^4-15 d^2+8
\end{array}
\right) \displaybreak[1]\\
T_{A,D} & = \left(
\begin{array}{cc}
 -5 d^4+27 d^2-13 & 7 d^4-35 d^2+21 \\
 7 d^4-35 d^2+14 & 2 d^4-8 d^2+1
\end{array}
\right) \displaybreak[1]\\
T_{C,f} & = \left(
\begin{array}{c}
 2 d^4-15 d^2+8
\end{array}
\right) \displaybreak[1]\\
T_{C,B} & = \left(
\begin{array}{ccc}
 8 d^4-46 d^2+25 & 9 d^4-50 d^2+29 & 3 d^4-19 d^2+12 \\
 -8 d^4+46 d^2-25 & -9 d^4+50 d^2-29 & -3 d^4+19 d^2-12 \\
 8 d^4-46 d^2+25 & 9 d^4-50 d^2+29 & 3 d^4-19 d^2+12
\end{array}
\right) \displaybreak[1]\\
T_{C,F} & = \left(
\begin{array}{cc}
 2 d^4-8 d^2+1 & 7 d^4-35 d^2+14 \\
 7 d^4-35 d^2+21 & -5 d^4+27 d^2-13
\end{array}
\right) \displaybreak[1]\\
T_{C,D} & = \left(
\begin{array}{ccc}
 -3 d^4+19 d^2-12 & -d^4+4 d^2-4 & -3 d^4+19 d^2-12 \\
 3 d^4-19 d^2+12 & d^4-4 d^2+4 & 3 d^4-19 d^2+12 \\
 -3 d^4+19 d^2-12 & -d^4+4 d^2-4 & -3 d^4+19 d^2-12
\end{array}
\right) \displaybreak[1]\\
T_{G,f} & = \left(
\begin{array}{c}
 -5 d^4+34 d^2-20
\end{array}
\right) \displaybreak[1]\\
T_{G,B} & = \left(
\begin{array}{c}
 2 d^4-15 d^2+8
\end{array}
\right) \displaybreak[1]\\
T_{G,F} & = \left(
\begin{array}{cccc}
 -4 d^4+23 d^2-16 & -3 d^4+19 d^2-12 & 8 d^4-46 d^2+25 & -4 d^4+16 d^2-9 \\
 5 d^4-27 d^2+13 & 0 & -9 d^4+50 d^2-29 & 0 \\
 -d^4+4 d^2+3 & 3 d^4-19 d^2+12 & d^4-4 d^2+4 & 4 d^4-16 d^2+9 \\
 -5 d^4+27 d^2-13 & 0 & 9 d^4-50 d^2+29 & 0
\end{array}
\right) \displaybreak[1]\\
T_{G,z} & = \left(
\begin{array}{cc}
 0 & 0 \\
 0 & 0
\end{array}
\right) \displaybreak[1]\\
T_{G,D} & = \left(
\begin{array}{c}
 2 d^4-15 d^2+8
\end{array}
\right) \displaybreak[1]\\
T_{E,f} & = \left(
\begin{array}{c}
 2 d^4-15 d^2+8
\end{array}
\right) \displaybreak[1]\\
T_{E,B} & = \left(
\begin{array}{cc}
 -12 d^4+69 d^2-41 & -21 d^4+133 d^2-77 \\
 7 & 9 d^4-50 d^2+29
\end{array}
\right) \displaybreak[1]\\
T_{E,F} & = \left(
\begin{array}{cccc}
 10 d^4-54 d^2+26 & 8 d^4-46 d^2+25 & -d^4+4 d^2-4 & -2 d^4+15 d^2-8 \\
 -4 d^4+23 d^2-16 & -9 d^4+50 d^2-29 & -3 d^4+19 d^2-12 & -6 d^4+31 d^2-17 \\
 -6 d^4+31 d^2-10 & d^4-4 d^2+4 & 4 d^4-23 d^2+16 & 8 d^4-46 d^2+25 \\
 6 d^4-31 d^2+10 & -d^4+4 d^2-4 & -4 d^4+23 d^2-16 & -8 d^4+46 d^2-25
\end{array}
\right) \displaybreak[1]\\
T_{E,z} & = \left(
\begin{array}{c}
 -5 d^4+34 d^2-20
\end{array}
\right) \displaybreak[1]\\
T_{E,D} & = \left(
\begin{array}{ccc}
 8 d^4-46 d^2+25 & 3 d^4-19 d^2+12 & 9 d^4-50 d^2+29 \\
 8 d^4-46 d^2+25 & 3 d^4-19 d^2+12 & 9 d^4-50 d^2+29 \\
 -8 d^4+46 d^2-25 & -3 d^4+19 d^2-12 & -9 d^4+50 d^2-29
\end{array}
\right) \displaybreak[1]\\
\end{align*}
}

\bibliographystyle{alpha}
\bibliography{../../bibliography/bibliography}

\end{document}

%% file: preamble.tex
\usepackage{amsmath,amssymb,amsfonts}
\usepackage{ifpdf}
\usepackage{enumerate}
\usepackage{ulem}
\usepackage{leftidx}

\usepackage{breqn}

\usepackage{tikz}
\usetikzlibrary{calc}
\usepackage{tikz-qtree}

\ifpdf
\usepackage[pdftex,plainpages=false,hypertexnames=false,pdfpagelabels]{hyperref}
\else
\usepackage[dvips,plainpages=false,hypertexnames=false]{hyperref}
\fi

\newcommand{\arxiv}[1]{\href{http://arxiv.org/abs/#1}{\tt arXiv:\nolinkurl{#1}}}
\newcommand{\doi}[1]{\href{http://dx.doi.org/#1}{{\tt DOI:#1}}}

\newcommand{\googlebooks}[1]{(preview at \href{http://books.google.com/books?id=#1}{google books})}

\theoremstyle{plain}
\newtheorem{prop}{Proposition}[section]
\newtheorem{conj}[prop]{Conjecture}
\newtheorem{thm}[prop]{Theorem}

\newtheorem{lem}[prop]{Lemma}

\newtheorem{fact}[prop]{Fact}
\newtheorem*{cor*}{Corollary}

\numberwithin{equation}{section}

\theoremstyle{remark}

\newtheorem*{rem*}{Remark}               
\newtheorem*{ex*}{Example}                

\theoremstyle{definition}
\newtheorem*{defn*}{Definition}             

\theoremstyle{plain}

\newcounter{comment}
\newcommand{\noop}[1]{}

\def\clap#1{\hbox to 0pt{\hss#1\hss}}

\def\semicolon{;}
\def\applytolist#1{
    \expandafter\def\csname multi#1\endcsname##1{
        \def\multiack{##1}\ifx\multiack\semicolon
            \def\next{\relax}
        \else
            \csname #1\endcsname{##1}
            \def\next{\csname multi#1\endcsname}
        \fi
        \next}
    \csname multi#1\endcsname}

\def\calc#1{\expandafter\def\csname c#1\endcsname{{\mathcal #1}}}
\applytolist{calc}QWERTYUIOPLKJHGFDSAZXCVBNM;
\def\bbc#1{\expandafter\def\csname bb#1\endcsname{{\mathbb #1}}}
\applytolist{bbc}QWERTYUIOPLKJHGFDSAZXCVBNM;
\def\bfc#1{\expandafter\def\csname bf#1\endcsname{{\mathbf #1}}}
\applytolist{bfc}QWERTYUIOPLKJHGFDSAZXCVBNM;

\renewcommand{\imath}{\mathfrak{i}}
\renewcommand{\jmath}{\mathfrak{j}}

\newcommand{\iso}{\cong}

\makeatletter
\newcommand{\hashdef}[2]{\@namedef{#1}{#2}}
\newcommand{\hashlookup}[1]{\@nameuse{#1}}
\makeatother

\IfFileExists{../../graphs/lookup.tex}%
	{\newcommand{\pathtographs}{../../graphs/}}%
	{\newcommand{\pathtographs}{diagrams/graphs/}}

\input{\pathtographs lookup.tex}

\newcommand{\bigraph}[1]{{\hspace{-3pt}\begin{array}{c}%
  \raisebox{-2.5pt}{\includegraphics[height=7.5mm]{\pathtographs \hashlookup{#1}}}%
\end{array}\hspace{-3pt}}}

\newcommand{\tensor}{\otimes}

%% file: diagrams/graphs/lookup.tex
\hashdef{bwd1duals1}{FF1481245D7F1881}%
\hashdef{bwd1v1p1duals1v1x2}{7CC0C785EDD0C3AA}%
\hashdef{bwd1v1p1duals1v2x1}{FE6BDA913A50F4D9}%
\hashdef{bwd1v1p1p1duals1v1x2x3}{CDEBF7CA2995B855}%
\hashdef{bwd1v1p1p1duals1v1x3x2}{30F84974997CC5A4}%
\hashdef{bwd1v1p1p1duals1v2x1x3}{99E8BDA8F0BBDAF7}%
\hashdef{bwd1v1p1p1v0x0x1p0x0x1duals1v2x1x3}{3CA5711FC31ECF6F}%
\hashdef{bwd1v1p1p1v0x1x0p1x0x0duals1v1x2x3}{3D38AB67FF2C735F}%
\hashdef{bwd1v1p1p1v0x1x0p1x0x0duals1v2x1x3}{65FFA160CF70AFC9}%
\hashdef{bwd1v1p1p1v0x1x0p1x0x0p1x0x0duals1v1x2x3}{D4CD08741B5CE702}%
\hashdef{bwd1v1p1p1v1x0x0p0x0x1p0x1x0p0x0x1duals1v2x1x3}{BC7E69FCACAE8C8C}%
\hashdef{bwd1v1p1p1v1x0x0p0x1x0duals1v1x2x3}{11AC443C8AA89B99}%
\hashdef{bwd1v1p1p1v1x0x0p0x1x0duals1v2x1x3}{D0DF5849C36C78A7}%
\hashdef{bwd1v1p1p1v1x0x0p0x1x0p0x0x1duals1v1x2x3}{78E85D16AD069189}%
\hashdef{bwd1v1p1p1v1x0x0p0x1x0p0x0x1duals1v1x3x2}{0F5E992D92F7208A}%
\hashdef{bwd1v1p1p1v1x0x0p0x1x0p0x0x1p0x0x1duals1v1x2x3}{24226E32DF46CAB7}%
\hashdef{bwd1v1p1p1v1x0x0p0x1x0p0x0x1p0x0x1duals1v2x1x3}{CD29E8644C8752D3}%
\hashdef{bwd1v1p1p1v1x0x0p0x1x0p0x1x0duals1v1x2x3}{3ABE61D9E542A652}%
\hashdef{bwd1v1p1p1v1x0x0p0x1x0p0x1x0p0x0x1duals1v1x2x3}{2B300291D31903D9}%
\hashdef{bwd1v1p1p1v1x0x0p1x0x0duals1v1x2x3}{2D773EAFE22CE973}%
\hashdef{bwd1v1p1p1v1x0x0p1x0x0p0x1x0duals1v1x2x3}{9A4923C1A8637863}%
\hashdef{bwd1v1p1p1v1x0x0p1x0x0p0x1x0p0x0x1duals1v1x2x3}{744230A83B7BE0BF}%
\hashdef{bwd1v1p1p1v1x0x0p1x0x0p0x1x0p0x0x1duals1v1x3x2}{01629756DDECC9EF}%
\hashdef{bwd1v1p1p1v1x0x0p1x0x0p0x1x0p0x1x0duals1v1x2x3}{EC866C1609AB282D}%
\hashdef{bwd1v1p1p1v1x0x0p1x0x0p0x1x0p0x1x0duals1v2x1x3}{FDB8DC380CB350FA}%
\hashdef{bwd1v1p1p1v1x0x0p1x0x0p1x0x0duals1v1x2x3}{084866B09692D0AA}%
\hashdef{bwd1v1p1v0x1p1x0p0x1p0x1duals1v1x2}{F6297A3E2593D31E}%
\hashdef{bwd1v1p1v1x0p0x1duals1v1x2}{FC413903EB6EB09E}%
\hashdef{bwd1v1p1v1x0p0x1p0x1duals1v1x2}{4C1337B54CECE742}%
\hashdef{bwd1v1p1v1x0p0x1p1x0duals1v1x2}{F7A3FE03E36A137F}%
\hashdef{bwd1v1p1v1x0p0x1p1x0p0x1duals1v1x2}{A6910D2611ECB26E}%
\hashdef{bwd1v1p1v1x0p0x1p1x0p0x1duals1v2x1}{B5EC082166445383}%
\hashdef{bwd1v1p1v1x0p0x1p1x0p0x1v1x0x0x0p0x1x0x0p0x0x1x0p0x0x0x1duals1v1x2v2x1x4x3}{E1EBF1A1EFE05600}%
\hashdef{bwd1v1p1v1x0p1x0duals1v1x2}{58AF9CED6D2C5F1A}%
\hashdef{bwd1v1p1v1x0p1x0p0x1duals1v1x2}{E787E3120C409620}%
\hashdef{bwd1v1p1v1x0p1x0p0x1p0x1duals1v1x2}{CBA085011AB3DF5E}%
\hashdef{bwd1v1p1v1x0p1x0p0x1p0x1duals1v2x1}{80F2359B11492211}%
\hashdef{bwd1v1p1v1x0p1x0p0x1p0x1v1x0x0x0p0x1x0x0p0x0x1x0p0x0x0x1duals1v1x2v3x4x1x2}{B19C088E9D6945B0}%
\hashdef{bwd1v1p1v1x0p1x0p1x0duals1v1x2}{01BD1A32780EB925}%
\hashdef{bwd1v1p1v1x0p1x1duals1v1x2}{FE060C2815A2CF0E}%
\hashdef{bwd1v1p1v1x1duals1v1x2}{8BDF697267FBBFEB}%
\hashdef{bwd1v1p1v1x1duals1v2x1}{9AB6D35FB0C8481F}%
\hashdef{bwd1v1p1v1x1v1duals1v1x2v1}{434A46414835DC7C}%
\hashdef{bwd1v1p1v1x1v1duals1v2x1v1}{4420D7BDD175DEC1}%
\hashdef{bwd1v1p1v1x1v1v1duals1v1x2v1}{B547C24BCA5DC540}%
\hashdef{bwd1v1p1v1x1v1v1duals1v2x1v1}{C3617890301423B1}%
\hashdef{bwd1v1v1p1v1x0p0x1p0x1v0x1x0p1x0x1duals1v1v2x1x3}{185DF41920CC8F1E}%
\hashdef{gbg1v1p1v1x1p0x1}{987209081F722714}%
\hashdef{gbg1v1p1v1x1v1v1}{F28613CF99086BB3}%

%% file: cells.tex

\newcommand{\maxwidth}[2][\linewidth]{
  \setbox0=\hbox{#2}
  \ifthenelse{\dimtest{\wd0}>{#1}}%
    {\resizebox{#1}{!}{#2}}
    {\usebox0}
}

\newcommand{\cellA}[6]{
\begin{scope}[shift=($(#1)+(#2)$)]
\path[#3] (-0.5,0.5) rectangle (0.5,-0.5);
\node[#4](#5) at (0,0) {\maxwidth[9mm]{#6}};
\end{scope}
}

\newcommand{\longcellA}[6]{
\begin{scope}[shift=($(#1)+(#2)$)]
\path[#3] (-0.5,0.5) rectangle (0.5,-0.5);
\node[#4](#5) at (0,0) {\maxwidth[19mm]{#6}};
\end{scope}
}

\newcommand{\emptycellA}[2]{
\cellA{#1}{#2}{fill=gray,opacity=.2}{}{}{}
}
\newcommand{\unitaryblockA}[4]{
\draw[ultra thick] ($(#1)+(#2)+(-0.5,0.5)$) rectangle ($(#3)+(#4)+(0.5,-0.5)$); 
}

\newcommand{\tileA}[3]{
\begin{scope}[shift=($(#1)+(#2)$)]
\path (-0.5,0.5) rectangle (0.5,-0.5);
\node[opacity=0.75] at (0,0) {\maxwidth[10mm]{\includegraphics[width=1cm]{diagrams/house/houses_tiles_#3}}};
\end{scope}
}

\newcommand{\cellB}[6]{
\begin{scope}[shift=($(#1)+(#2)$)]
\path[#3] (-0.5,0.5) rectangle (0.5,-0.5);
\node[#4](#5) at (0,0) {\maxwidth[9mm]{#6}};
\end{scope}
}

\newcommand{\mediumcellB}[6]{
\begin{scope}[shift=($(#1)+(#2)$)]
\path[#3] (-0.5,0.5) rectangle (0.5,-0.5);
\node[#4](#5) at (0,0) {\maxwidth[14mm]{#6}};
\end{scope}
}

\newcommand{\longcellB}[6]{
\begin{scope}[shift=($(#1)+(#2)$)]
\path[#3] (-0.5,0.5) rectangle (0.5,-0.5);
\node[#4](#5) at (0,0) {\maxwidth[22mm]{#6}};
\end{scope}
}

\usetikzlibrary{patterns}

\pgfdeclarepatternformonly{stripes}
{\pgfpointorigin}{\pgfpoint{1cm}{1cm}}
{\pgfpoint{1cm}{1cm}}
{
    \pgfpathmoveto{\pgfpoint{0cm}{0cm}}
    \pgfpathlineto{\pgfpoint{1cm}{1cm}}
    \pgfpathlineto{\pgfpoint{1cm}{0.5cm}}
    \pgfpathlineto{\pgfpoint{0.5cm}{0cm}}
    \pgfpathclose%
    \pgfusepath{fill}
    \pgfpathmoveto{\pgfpoint{0cm}{0.5cm}}
    \pgfpathlineto{\pgfpoint{0cm}{1cm}}
    \pgfpathlineto{\pgfpoint{0.5cm}{1cm}}
    \pgfpathclose%
    \pgfusepath{fill}
}

\newcommand{\normcellB}[6]{
\cellB{#1}{#2}{fill, pattern=stripes, pattern color=red!15}{#4}{#5}{#6}
}

\newcommand{\emptycellB}[2]{
\cellB{#1}{#2}{fill=gray,opacity=.2}{}{}{}
}
\newcommand{\unitaryblockB}[4]{
\draw[ultra thick] ($(#1)+(#2)+(-0.5,0.5)$) rectangle ($(#3)+(#4)+(0.5,-0.5)$); 
}

\newcommand{\principalmatrixA}{

\node(c1X) at (1,0) {$1X$};
\node(cZX) at (2,0) {$ZX$};
\node(cZW) at (3,0) {$ZW$};
\node(cYX) at (4,0) {$YX$};
\node(cYW) at (5,0) {$YW$};
\node(cYg) at (6,0) {$Y\smash{g}$};
\node(r1X) at (0,-1) {$\hat{1} \bar{X}$};
\node(rZX) at (0,-2) {$\hat{Z} \bar{X}$};
\node(rZW) at (0,-3) {$\hat{Z} \bar{W}$};
\node(rYX) at (0,-4) {$\hat{Y} \bar{X}$};
\node(rYW) at (0,-5) {$\hat{Y} \bar{W}$};
\node(rYg) at (0,-6) {$\hat{Y} \bar{g}$};
\begin{scope}[shift={(0.5,0.5)}]
\draw[xstep=1,ystep=1] ($(c1X)+(rYg)+(-0.00001,-0.00001)$) grid ($(cYg)+(r1X)+(1,1)$);
\end{scope}
\unitaryblockA{r1X}{c1X}{r1X}{c1X}
\unitaryblockA{rZX}{c1X}{rZX}{c1X}
\unitaryblockA{rYX}{c1X}{rYX}{c1X}
\unitaryblockA{r1X}{cZX}{r1X}{cZX}
\unitaryblockA{rZX}{cZX}{rZW}{cZW}
\unitaryblockA{rYX}{cZX}{rYW}{cZW}
\unitaryblockA{r1X}{cYX}{r1X}{cYX}
\unitaryblockA{rZX}{cYX}{rZW}{cYW}
\unitaryblockA{rYX}{cYX}{rYg}{cYg}
\emptycellA{r1X}{cZW}
\emptycellA{r1X}{cYW}
\emptycellA{r1X}{cYg}
\emptycellA{rZX}{cYg}
\emptycellA{rZW}{cYg}
\emptycellA{rZW}{c1X}
\emptycellA{rYW}{c1X}
\emptycellA{rYg}{c1X}
\emptycellA{rYg}{cZX}
\emptycellA{rYg}{cZW}
}

\newcommand{\dualmatrixA}{
\node(rX1) at (0,-1) {$X \hat{1}$};
\node(rXZ) at (0,-2) {$X\hat{Z}$};
\node(rXY) at (0,-3) {$X \hat{Y}$};
\node(rWZ) at (0,-4) {$W \hat{Z}$};
\node(rWY) at (0,-5) {$W \hat{Y}$};
\node(rgY) at (0,-6) {$g \hat{Y}$};
\node(cX1) at (1,0) {$ \bar{X} 1$};
\node(cXZ) at (2,0) {$\bar{X} Z$};
\node(cXY) at (3,0) {$\bar{X} Y$};
\node(cWZ) at (4,0) {$\bar{W} Z$};
\node(cWY) at (5,0) {$\bar{W} Y$};
\node(cgY) at (6,0) {$\smash{\bar{g}} Y$};
\begin{scope}[shift={(0.5,0.5)}]
\draw[xstep=1,ystep=1] ($(cX1)+(rgY)+(-0.00001,-0.00001)$) grid ($(cgY)+(rX1)+(1,1)$);
\end{scope}
\unitaryblockA{rX1}{cX1}{rXY}{cXY}
\unitaryblockA{rXZ}{cWZ}{rXY}{cWY}
\unitaryblockA{rXY}{cgY}{rXY}{cgY}
\unitaryblockA{rWZ}{cXZ}{rWY}{cXY}
\unitaryblockA{rWZ}{cWZ}{rWY}{cWY}
\unitaryblockA{rWY}{cgY}{rWY}{cgY}
\unitaryblockA{rgY}{cXY}{rgY}{cXY}
\unitaryblockA{rgY}{cWY}{rgY}{cWY}
\unitaryblockA{rgY}{cgY}{rgY}{cgY}
\emptycellA{rX1}{cWZ}
\emptycellA{rX1}{cWY}
\emptycellA{rX1}{cgY}
\emptycellA{rXZ}{cgY}
\emptycellA{rWZ}{cgY}
\emptycellA{rWZ}{cX1}
\emptycellA{rWY}{cX1}
\emptycellA{rgY}{cX1}
\emptycellA{rgY}{cXZ}
\emptycellA{rgY}{cWZ}
}

\newcommand{\principalmatrixB}{
\node(c1f) at (1,0) {$1f$};
\node(cAf) at (2,0) {$Af$};
\node(cAB) at (3,0) {$AB$};
\node(cAF) at (4,0) {$AF$};
\node(cGF) at (5,0) {$GF$};
\node(cGz) at (6,0) {$Gz$};
\node(cCB) at (7,0) {$CB$};
\node(cCD) at (8,0) {$CD$};
\node(cEF) at (9,0) {$EF$};
\node(cED) at (10,0) {$ED$};
\node(r1f) at (0,-1) {$\hat{1} \bar{f}$};
\node(rHf) at (0,-2) {$\hat{H} \bar{f}$};
\node(rHB) at (0,-3) {$\hat{H} \bar{B}$};
\node(rHF) at (0,-4) {$\hat{H} \bar{F}$};
\node(rIB) at (0,-5) {$\hat{I} \bar{B}$};
\node(rID) at (0,-6) {$\hat{I} \bar{D}$};
\node(rJF) at (0,-7) {$\hat{J} \bar{F}$};
\node(rJz) at (0,-8) {$\hat{J} \bar{z}$};
\node(rKF) at (0,-9) {$\hat{K} \bar{F}$};
\node(rKD) at (0,-10) {$\hat{K} \bar{D}$};
\begin{scope}[shift={(0.5,0.5)}]
\draw[xstep=1,ystep=1] ($(c1f)+(rKD)+(-0.00001,-0.00001)$) grid ($(cED)+(r1f)+(1,1)$);
\end{scope}
\unitaryblockB{r1f}{c1f}{r1f}{c1f}
\unitaryblockB{rHf}{c1f}{rHf}{c1f}
\unitaryblockB{r1f}{cAf}{r1f}{cAf}
\unitaryblockB{rHf}{cAf}{rHF}{cAF}
\unitaryblockB{rHB}{cGF}{rHB}{cGF}
\unitaryblockB{rHF}{cCB}{rHF}{cCB}
\unitaryblockB{rHF}{cEF}{rHF}{cEF}
\unitaryblockB{rIB}{cAF}{rIB}{cAF}
\unitaryblockB{rIB}{cGF}{rID}{cGz}
\unitaryblockB{rID}{cEF}{rID}{cEF}
\unitaryblockB{rJF}{cAB}{rJF}{cAB}
\unitaryblockB{rJF}{cCB}{rJz}{cCD}
\unitaryblockB{rJF}{cED}{rJF}{cED}
\unitaryblockB{rKF}{cAF}{rKF}{cAF}
\unitaryblockB{rKF}{cCD}{rKF}{cCD}
\unitaryblockB{rKD}{cGF}{rKD}{cGF}
\unitaryblockB{rKF}{cEF}{rKD}{cED}
\emptycellB{r1f}{cAB}
\emptycellB{r1f}{cAF}
\emptycellB{r1f}{cCB}
\emptycellB{r1f}{cCD}
\emptycellB{r1f}{cGz}
\emptycellB{r1f}{cGF}
\emptycellB{r1f}{cEF}
\emptycellB{r1f}{cED}
\emptycellB{rHf}{cCB}
\emptycellB{rHf}{cCD}
\emptycellB{rHf}{cGz}
\emptycellB{rHf}{cGF}
\emptycellB{rHf}{cEF}
\emptycellB{rHf}{cED}
\emptycellB{rHB}{c1f}
\emptycellB{rHB}{cCB}
\emptycellB{rHB}{cCD}
\emptycellB{rHB}{cGz}
\emptycellB{rHB}{cEF}
\emptycellB{rHB}{cED}
\emptycellB{rHF}{c1f}
\emptycellB{rHF}{cCD}
\emptycellB{rHF}{cGz}
\emptycellB{rHF}{cGF}
\emptycellB{rHF}{cED}
\emptycellB{rID}{c1f}
\emptycellB{rID}{cAf}
\emptycellB{rID}{cAB}
\emptycellB{rID}{cAF}
\emptycellB{rID}{cCB}
\emptycellB{rID}{cCD}
\emptycellB{rID}{cED}
\emptycellB{rIB}{c1f}
\emptycellB{rIB}{cAf}
\emptycellB{rIB}{cAB}
\emptycellB{rIB}{cCB}
\emptycellB{rIB}{cCD}
\emptycellB{rIB}{cEF}
\emptycellB{rIB}{cED}
\emptycellB{rJF}{c1f}
\emptycellB{rJF}{cAf}
\emptycellB{rJF}{cAF}
\emptycellB{rJF}{cGz}
\emptycellB{rJF}{cGF}
\emptycellB{rJF}{cEF}
\emptycellB{rJz}{c1f}
\emptycellB{rJz}{cAf}
\emptycellB{rJz}{cAB}
\emptycellB{rJz}{cAF}
\emptycellB{rJz}{cGz}
\emptycellB{rJz}{cGF}
\emptycellB{rJz}{cEF}
\emptycellB{rJz}{cED}
\emptycellB{rKF}{c1f}
\emptycellB{rKF}{cAf}
\emptycellB{rKF}{cAB}
\emptycellB{rKF}{cCB}
\emptycellB{rKF}{cGz}
\emptycellB{rKF}{cGF}
\emptycellB{rKD}{c1f}
\emptycellB{rKD}{cAf}
\emptycellB{rKD}{cAB}
\emptycellB{rKD}{cAF}
\emptycellB{rKD}{cCB}
\emptycellB{rKD}{cCD}
\emptycellB{rKD}{cGz}
}

\newcommand{\dualmatrixB}{
\node(cf1) at (1,0) {$\bar{f} 1$};
\node(cfA) at (2,0) {$\bar{f} A$};
\node(cBA) at (3,0) {$\bar{B} A$};
\node(cBG) at (4,0) {$\bar{B} G$};
\node(cFC) at (5,0) {$\bar{F} C$};
\node(cFA) at (6,0) {$\bar{F} A$};
\node(cFE) at (7,0) {$\bar{F} E$};
\node(czC) at (8,0) {$\bar{z} C$};
\node(cDG) at (9,0) {$\bar{D} G$};
\node(cDE) at (10,0) {$\bar{D} E$};
\node(rf1) at (0,-1) {$f \hat{1} $};
\node(rfH) at (0,-2) {$f \hat{H} $};
\node(rBH) at (0,-3) {$B \hat{H} $};
\node(rBJ) at (0,-4) {$B \hat{J} $};
\node(rFI) at (0,-5) {$F \hat{I} $};
\node(rFH) at (0,-6) {$F \hat{H} $};
\node(rFK) at (0,-7) {$F \hat{K} $};
\node(rzI) at (0,-8) {$z \hat{I} $};
\node(rDJ) at (0,-9) {$D \hat{J} $};
\node(rDK) at (0,-10) {$D \hat{K} $};
\begin{scope}[shift={(0.5,0.5)}]
\draw[xstep=1,ystep=1] ($(cf1)+(rDK)+(-0.00001,-0.00001)$) grid ($(cDE)+(rf1)+(1,1)$);
\end{scope}
\unitaryblockB{rf1}{cf1}{rfH}{cfA}
\unitaryblockB{rfH}{cBA}{rfH}{cBA}
\unitaryblockB{rBH}{cfA}{rBH}{cfA}
\unitaryblockB{rBH}{cBA}{rBH}{cBA}
\unitaryblockB{rfH}{cFA}{rfH}{cFA}
\unitaryblockB{rBH}{cFC}{rBJ}{cFA}
\unitaryblockB{rBJ}{czC}{rBJ}{czC}
\unitaryblockB{rFI}{cBA}{rFH}{cBG}
\unitaryblockB{rFH}{cfA}{rFH}{cfA}
\unitaryblockB{rFH}{cFA}{rFK}{cFE}
\unitaryblockB{rFI}{cDG}{rFK}{cDE}
\unitaryblockB{rzI}{cBG}{rzI}{cBG}
\unitaryblockB{rzI}{cDG}{rzI}{cDG}
\unitaryblockB{rDJ}{cFC}{rDK}{cFE}
\unitaryblockB{rDJ}{czC}{rDJ}{czC}
\unitaryblockB{rDK}{cDE}{rDK}{cDE}
\emptycellB{rf1}{cBA}
\emptycellB{rf1}{cBG}
\emptycellB{rf1}{cFA}
\emptycellB{rf1}{cFC}
\emptycellB{rf1}{cFE}
\emptycellB{rf1}{czC}
\emptycellB{rf1}{cDE}
\emptycellB{rf1}{cDG}
\emptycellB{rfH}{cBG}
\emptycellB{rfH}{cFC}
\emptycellB{rfH}{cFE}
\emptycellB{rfH}{czC}
\emptycellB{rfH}{cDE}
\emptycellB{rfH}{cDG}
\emptycellB{rBJ}{cfA}
\emptycellB{rBJ}{cf1}
\emptycellB{rBJ}{cBA}
\emptycellB{rBJ}{cBG}
\emptycellB{rBJ}{cFE}
\emptycellB{rBJ}{cDE}
\emptycellB{rBJ}{cDG}
\emptycellB{rBH}{cf1}
\emptycellB{rBH}{cBG}
\emptycellB{rBH}{cFE}
\emptycellB{rBH}{czC}
\emptycellB{rBH}{cDE}
\emptycellB{rBH}{cDG}
\emptycellB{rFH}{cf1}
\emptycellB{rFH}{cFC}
\emptycellB{rFH}{czC}
\emptycellB{rFH}{cDE}
\emptycellB{rFH}{cDG}
\emptycellB{rFK}{cfA}
\emptycellB{rFK}{cf1}
\emptycellB{rFK}{cBA}
\emptycellB{rFK}{cBG}
\emptycellB{rFK}{cFC}
\emptycellB{rFK}{czC}
\emptycellB{rFI}{cfA}
\emptycellB{rFI}{cf1}
\emptycellB{rFI}{cFA}
\emptycellB{rFI}{cFC}
\emptycellB{rFI}{cFE}
\emptycellB{rFI}{czC}
\emptycellB{rzI}{cfA}
\emptycellB{rzI}{cf1}
\emptycellB{rzI}{cBA}
\emptycellB{rzI}{cFA}
\emptycellB{rzI}{cFC}
\emptycellB{rzI}{cFE}
\emptycellB{rzI}{czC}
\emptycellB{rzI}{cDE}
\emptycellB{rDK}{cfA}
\emptycellB{rDK}{cf1}
\emptycellB{rDK}{cBA}
\emptycellB{rDK}{cBG}
\emptycellB{rDK}{cFA}
\emptycellB{rDK}{czC}
\emptycellB{rDK}{cDG}
\emptycellB{rDJ}{cfA}
\emptycellB{rDJ}{cf1}
\emptycellB{rDJ}{cBA}
\emptycellB{rDJ}{cBG}
\emptycellB{rDJ}{cFA}
\emptycellB{rDJ}{cDE}
\emptycellB{rDJ}{cDG}
}

\newcommand{\tileB}[3]{
  \cellB{#1}{#2}{}{opacity=0.75}{}{\includegraphics[width=0.975cm]{diagrams/house/houses_tiles_#3}}
}